\documentclass{amsart}

\usepackage{xcolor}
\usepackage{amsmath,amssymb}
\usepackage{amscd}
\usepackage{mathrsfs}
\usepackage{amsthm}
\theoremstyle{plain}
\usepackage{mathtools, todonotes}

\newtheorem{theorem}{\bf Theorem}[section]
\newtheorem{lemma}[theorem]{\bf Lemma}
\newtheorem{proposition}[theorem]{\bf Proposition}
\newtheorem{corollary}[theorem]{\bf Corollary}

\theoremstyle{definition}
\newtheorem{definition}[theorem]{\bf Definition}
\newtheorem{example}[theorem]{\bf Example}

\newtheorem{remark}[theorem]{\bf Remark}

\newcommand{\R}{\mathbb{R}}
\newcommand{\C}{\mathbb{C}}
\newcommand{\N}{\mathbb{N}}
\newcommand{\Z}{\mathbb{Z}}

\newcommand{\nai}[2]{\langle #1,#2\rangle}

\newcommand{\vNtensor}{\overline{\otimes}}

\newcommand{\B}{\mathbb B}
\newcommand{\Q}{\mathbb Q}

\newcommand{\Ga}{\Gamma}

\newcommand{\acts}{\curvearrowright}

\usepackage[all]{xy}

\title{Analysis of the Plancherel weight and factoriality of the group von Neumann algebras of non-unimodular almost unimodular groups}
\date{\today}

\author[Y.~Miyamoto]{Yuki Miyamoto}\thanks{Yuki~Miyamoto, Department of Mathematics and Informatics, Chiba University, 1-33 Yayoi-cho, Inage, Chiba, 263- 8522,
Japan;}\thanks{Email : 25wd0101@student.gs.chiba-u.jp}
\address{}
\email{}

\let\origmaketitle\maketitle
\def\maketitle{
  \begingroup
  \def\uppercasenonmath##1{} % this disables uppercasing title
  \let\MakeUppercase\relax % this disables uppercasing authors
  \origmaketitle
  \endgroup
}
\begin{document}
\maketitle

\begin{abstract}
Let $G$ be a locally compact group, $L(G)$ be its group von Neumann algebra equipped with the Plancherel weight $\varphi_G$. In this paper, we consider the following two questions. (1) When is the restriction of $\varphi_G$ to the subalgebra generated by a closed subgroup $H$ semifinite? If so, is it equal (up to a constant) to $\varphi_H$? (2) When is $L(G)$ a factor?
We give a complete answer to (1), and when $G$ is second countable, $G_1:= \text{ker}\Delta_G$ is open in $G$ (called almost unimodular) and admits a sufficiently large non-unimodular part, we provide an answer to (2). When $L(G)$ is a factor, we also provide the formula of the S-invariant of $L(G)$.
\end{abstract}

\noindent

\medskip

\noindent
{\bf Mathematics Subject Classification (2020) 43A65 43A70 46L10}.
\medskip

\tableofcontents 
\section{Introduction}
Given a locally compact group $G$ with a fixed left Haar measure $\mu_G$, there is the left regular representation $\lambda^G_\ell$ and the right regular representation $\lambda^G_r$ of $G$ on $L^2(G,\mu_G)$ defined by $$(\lambda_\ell^G(g)\xi)(h)=\xi(g^{-1}h),\,\, (\lambda_r^G(g)\xi)(h)=\Delta_G(g)^{1/2}\xi(hg)\,\,(g,h\in G)$$
for all $\xi \in L^2(G,\mu_G)$. Here, $\Delta_G\colon G\to (0,\infty)$ is the modular function of $G$. The group von Neumann algebra $L(G)$ is 
generated by the operators $\lambda_\ell^G(g)$, that is, $L(G):=\{\lambda_\ell^G(g) : g\in G\}'' \subseteq \B(L^2(G, \mu_G))$. 
Any group von Neumann algebra $L(G)$ admits a canonical normal faithful semifinite weight $\varphi_G$, called the Plancherel weight (Definition \ref{defofplweight}).

Group von Neumann algebras form a rich class of W$^*$-algebras and it has a long history of active research. The structure of $L(G)$ is particularly well behaved when $G$ is discrete: 
\begin{itemize}
    \item[(i)] The Plancherel weight $\varphi_G$ is a faithful normal tracial state, if and only if $G$ is discrete. 
    \item[(ii)] Any $x\in L(G)$ admits the so-called Fourier series expansion $x=\sum_{g\in G}x(g)\lambda(g)$ in a suitable Hilbert space topology.
\end{itemize}
These properties are extremely useful for the structural study of group von Neumann algebras. In fact, the group von Neumann algebra of the free group $\mathbb{F}_2$ was studied already by Murray--von Neumann to provide the very first example of a non-hyperfinite II$_1$ factor. Also, (ii) allows one to show that $L(G)$ is a factor if and only if $G$ is ICC (in which case $L(G)$ is a type II$_1$ factor). 
In contrast, there is no known general condition which characterizes the factoriality of $L(G)$. Besides, $\varphi_G$ is never bounded if $G$ is not discrete. The situation becomes more involved for non-unimodular groups. When $G$ is non-unimodular, $\varphi_G$ is not a tracial weight. In that case, it may happen that $L(G)$ may fail to be semifinite. It is thus important to understand the modular theoretic properties of the Plancherel weight. Because of these difficulties, much less is known about the structure of general group von Neumann algebras of non-discrete groups. 
One natural question to ask about $\varphi_G$ is when it is strictly semifinite \cite{CombesMR288589}. That is, when the restriction of $\varphi_G$ to its centralizer is semifinite. This is equivalent to the existence of a $\varphi_G$-preserving normal faithful conditional expectation of $L(G)$ onto the centralizer $L(G)^{\varphi_G}$. This question has been answered in a recent work by Garcia Guinto and Nelson \cite[Theorem 2.1]{guinto2025unimodulargroups}: $\varphi_G$ is strictly semifinite if and only if $G_1=\ker(\Delta_G)$ is open (such a group is called almost unimodular by the authors), in which case $\varphi_G$ is an almost periodic weight. The authors then initiated the study of group von Neumann algebras of almost unimodular groups. 

\par

In this context, motivated by the work of Garcia Guinto and Nelson, we would like to consider the following questions: 
\begin{itemize}
\item[(1)] Suppose $H$ is a closed subgroup of $G$. When is the restriction of $\varphi_G$ onto the subalgebra (isomorphic to) $L(H)$ semifinite? If so, is it equal to (a constant multiple of) $\varphi_H$? 
\item[(2)] When is $L(G)$ a factor?
\end{itemize}
It is clear that (1) is a generalization of the strict semifiniteness question. Note that $L(H)$ is always globally invariant by the modular flow $\sigma^{\varphi_G}$, while $\varphi_G$ may fail to be semifinite on $L(H)$. We also remark that the centralizer of a faithful normal semifinite weight can be of type III (see \cite{HaageruptypeIIIcentralizerMR430801}).

In this paper, we give a complete answer to (1). 
We also give an answer to (2) when $G$ is second countable almost unimodular and has sufficiently large non-unimodular part. Here, almost unimodular means $G_1 := \text{ker}\Delta_G$ is open in $G$. This notion is also defined in \cite[Definition 2.2]{guinto2025unimodulargroups}. 

For (1), we show in Theorem \ref{openiffsf} that the following three conditions are equivalent:
\begin{itemize}
\item[(i)] $H$ is open.
\item[(ii)] The restriction of $\varphi_G$ onto (a copy of) $L(H)$ is semifinite.
\item[(iii)] There exists $C>0$ such that the restriction of $\varphi_G$ onto (a copy of) $L(H)$ equals $C\cdot\varphi_H$.
\end{itemize}
Moreover, in the above cases, the explicit form of $C$ can be given upon fixing a $G$-invariant Radon measure $\mu$ on $G/H$. 
We note that (ii)$\implies$(i) has been shown in a more general setting of locally compact quantum groups in \cite[Theorem 7.5]{KKS2016}, while (ii)$\iff$(iii) can be shown using the uniqueness of the Haar weight for Kac algebras \cite[Theorem 2.7.7]{EnockSchwartz1992} (we thank Professor Reiji Tomatsu for informing us of these references). Thus, the result has been essentially known. Here, we give a more direct and elementary proof in the locally compact group case which does not involve quantum group theory. 

% For (1), we show in Theorem \ref{openiffsf} that the restriction of $\varphi_G$ onto (a copy of) $L(H)$ is semifinite, if and only if $H$ is open. Moreover, in this case, $\varphi_G|{L(H)}$ is of the form $C\cdot\varphi_H$, for some constant $C>0$ (the explicit form of $C$ can be given upon fixing a $G$-invariant Radon measure $\mu$ on $G/H$ is chosen). Professor Tomatsu told us that these facts are well known in the context of quantum groups. However, we provide an elementary proof from the perspective of group von Neumann algebras. 
By the work of Takesaki--Tasuuma \cite{TakesakiTatsuuma1971} (using Blattner's generalization \cite{Blattner1962} of Mackey's theorem on induced representation), the restriction of $\lambda_{\ell}^G$ to $H$ is quasi-equivalent to $\lambda_{\ell}^H$. Our proof rests on an explicit form of the $*$-isomorphism $\Phi\colon L(H)''\to \{\lambda_{\ell}^G(h):h\in H\}''\subseteq L(G)$ witnessing the quasi-equivalence. We remark that our method is different from Garcia Guinto--Nelson's, thus our result gives an alternative proof of (i)$\iff$(ii) part of \cite[Theorem 2.1]{guinto2025unimodulargroups} in the case of $H=\ker(\Delta_G)$. 

For (2), our study is motivated by Garcia Guinto--Nelson's description \cite[Example 3.6]{guinto2025unimodulargroups} of $G$ as a twisted crossed product $G=G_1\rtimes_{\alpha,c}\Delta_G(G)$ of $G_1=\ker{\Delta_G}$ by a cocycle action $(\alpha,c)$ of the abelian group $\Delta_G(G)$ (regarded as a discrete group).

Although our proof of Theorem \ref{tdlcfactority}  follows lines of  \cite[Example 3.8, Theorem 5.1]{guinto2025unimodulargroups} and \cite[II, Proposition 3.1.7]{Sutherland1980}, details are omitted in each of the works. Therefore, we decided to give a detailed argument for completeness. We also remark that our actual proof differs in minor details (see Remark \ref{remark1} and Remark \ref{remark2}). 
Moreover, we give the characterization of factoriality motivated by the Connes spectrum argument of the crossed products by usual actions (see \cite[Chapter III, Corollary 3.4]{ConnesTakesaki1977}). For a cocycle action $(\alpha,u)\colon G\curvearrowright M$ (see Definition \ref{defofcocycle}), we define the (left) twisted crossed product $M \rtimes_{\alpha,u}^\ell G$ (see Definition \ref{defoftwistvN}) and $\Gamma(\alpha,u)$ as the closed subgroup of $\hat{G}$ which is the dual group of the abelian locally compact group $G$ (see Definition \ref{defoftwistspctrm}). We show in Theorem \ref{twistConnesfactority} that $M \rtimes_{\alpha,u}^\ell G$ is a factor if and only if the cocycle action is centrally ergodic and $\Gamma(\alpha,u)=\hat{G}$ holds. 
This is a twisted version of Connes--Takesaki's characterization of the factoriality of crossed products \cite[III Corollary 3.4]{ConnesTakesaki1977}. As we learned from Professor Reiji Tomatsu (after the first draft has been posted online), that this can be shown using Sutherland's stabilization trick (i.e., to take the tensor product with a type I factor to remove the 2-cocycle) and then the original Connes--Takesaki's characterization may be used. Since the full proof has not been found in the literature, we include the full proof for completeness.
% Professor Reiji Tomatsu also informed us that this Theorem follows from Sutherland's stabilization trick and the original result. But we need details, and hence provide a full proof. 
For an almost unimodular locally compact group $G$, the openness of $G_1=\text{ker}\Delta_G$ implies the twisted crossed product decomposition $L(G)\cong L(G_1)\rtimes_{\alpha,u}^\ell \Delta_G(G)$ (\cite[Proposition 3.1.7]{Sutherland1980}).
This decomposition plays a crucial role. Since $\Delta_G(G)$ is a multiplicative subgroup of $\R_+=\R_{>0}$, it is abelian and hence we can apply Theorem \ref{twistConnesfactority}. As was pointed out by Garcia Guinto and Nelson \cite[Theorem 5.1]{guinto2025unimodulargroups}, the dual action $\hat{\alpha}$ of $(\alpha,c)$ (in the sense of \cite[II p.150]{Sutherland1980}, see Definition \ref{defdualaction}) is conjugate to Connes' point-modular extension of $\varphi_G$. From this, it can be shown that $\Gamma(\alpha,u)=\widehat{\Delta_G(G)}$ (Proposition \ref{fullspectrum}), and thus we conclude that $L(G)$ is a factor if and only if the cocycle action is centrally ergodic (Theorem \ref{tdlcfactority}).

The cocycle action $(\alpha,u) \colon \Delta_G(G) \curvearrowright L(G_1)$ exhibits behavior similar to that of the dual action of the modular flow of the Plancherel weight $\varphi_G$. In $\S 5$, we examine how such behavior arises in analogy with \cite{takesaki1973duality}. First, we calculate the S-invariant of $L(G)$ when $L(G)$ is a factor through the twisted crossed product decomposition $L(G)\cong L(G_1)\rtimes_{\alpha,u}^\ell \Delta_G(G)$. That is, we show the following formula in Proposition \ref{Sinvformula}:
\begin{align*}
    \text{S}(L(G))\backslash\{0\}&=\bigcap_{0\ne e\in Z(L(G_1))^P}\overline{\{k\in\Delta_G(G):\alpha_k(e)e\ne0\}}\backslash\{0\}
         \\&\supseteq \overline{\{k\in\Delta_G(G):\alpha_k|_{Z(L(G_1))}=\text{id}\}}\backslash\{0\}.
\end{align*}
Note that the twisted crossed product decomposition depends on a section $\sigma \colon \Delta_G(G) \to G$. We remark that the above formula is valid when we decompose $L(G)\cong L(G_1)\rtimes_{\alpha,u}^\ell \Delta_G(G)$ with a section $\sigma$ which satisfies $\sigma(k^{-1}) = \sigma(k)^{-1}$ for all $k \in \Delta_G(G)$, called inverse preserving. There is no concern because any almost unimodular group admits an inverse preserving section (Remark \ref{existenceofinvsec}). If $\Delta_G(G)$ is relatively discrete in $\R_+$ (or equivalently $\Delta_G(G)$ is singly generated), then the above formula turns into the following one: 
$$\text{S}(L(G))\backslash\{0\}=\{k\in\Delta_G(G):\alpha_k|_{Z(L(G_1))}=\text{id}\} \backslash\{0\}.$$
Moreover, We also characterize the semifiniteness of $L(G)$ through the twisted crossed product decomposition (Proposition \ref{sfnessbycocycleact}). As a consequence of these, we provide a sufficient condition for $L(G)$ to be a type $\rm{III}_0$-factor through conditions for the cocycle action $(\alpha,u)\colon \Delta_G(G) \curvearrowright L(G_1)$ (Proposition \ref{suffcondiforIII0}). Here, given a locally compact group $H$ such that $\Delta_H(H)$ is relatively discrete and $L(H_1)$ is a factor, \cite[Lemma 5.1]{Sutherland1980} can be applied to $H_1$ and hence we obtain the new group $G$ such that $L(G)$ is a $\rm{III_0}$-factor. We can confirm that this example is also consequence of the above proposition. Finally, we follow the same construction as in the case that $\Delta_H(H)$ is relatively discrete for the case that the one is dense in $\R_+$. We calculate the S-invariant and conclude that we obtain a type $\rm{III}_1$-factor (Corollary \ref{CorofSutex}).

\section{Analysis of the Plancherel weight}
Let $G$ be a locally compact group. In this section, we denote $\lambda=\lambda_\ell^G$ and $\rho=\lambda_r^G$ when there is no risk of confusion. Moreover, given $G$ and a closed subgroup $H$, we consider arbitrarily fixed left Haar measures $\mu_G$ on $G$ and $\mu_H$ on $H$, and denote $L^2(G)=L^2(G,\mu_G),L^2(H)=L^2(H,\mu_H)$. First, we recall the following classical result due to \cite{TakesakiTatsuuma1971} (see the proof of \cite[Theorem 6]{TakesakiTatsuuma1971}). Since we need the explicit form of the $*$-isomorphism $\Phi\colon L(H)\to \lambda_\ell^G(H)''\subseteq L(G)$, we provide an outline of the proof. We also refer to \cite[Proposition 3.4.4]{KaniuthLaubookMR3821506} for details.

\begin{proposition} \label{subisom}
    For a locally compact group $G$ and the closed subgroup $H$ of $G$, we have 
    $$L(H) \cong \{\lambda_\ell^G(h) : h \in H\}'' \subseteq \B (L^2 (G)).$$
\end{proposition}
\begin{proof}
    Let $\tilde{\rho}_h \,\,(h\in H)$ be the right regular representation of $H$. Denote $\mathcal{F}(G,\tilde{\rho})$ as the set of all continuous mappings $\xi\colon G\to L^2(H)$ with the following properties; 
    \begin{itemize}
        \item $\xi$ has a compact support modulo $H$.
        \item $\xi(xy)=\Delta_G(y)^{-1/2}\Delta_H(x)^{1/2}\tilde{\rho}_{y^{-1}}\xi(x)$ for all $x \in G, y \in H$.
    \end{itemize}
    Note that $\mathcal{F}(G,\tilde{\rho})$ is dense in $\mathcal{H}(\text{Ind}_H^G\tilde{\rho})$. Define an intertwiner operator $U\colon L^2(G)\to \mathcal{H}(\text{Ind}_H^G\tilde{\rho})$ by $$(Uf)(x)(h)=\Delta_G(x)^{-1/2}f(hx^{-1})\,\, (f\in C_c(G)).$$ In this case, we obtain $U^*\colon \mathcal{H}(\text{Ind}_H^G\tilde{\rho})\to L^2(G)$ by
    $$(U^*\xi)(x)=\Delta_G(x)^{-1/2}\xi(x^{-1})(e) \,\, (\xi\in \mathcal{F}(G,\tilde{\rho})).$$
    In addition, define $\widehat{\cdot}\colon \B(L^2(H))\to \B(\mathcal{H}(\text{Ind}_H^G\tilde{\rho}))$ by $$(\widehat{T}\xi)(x)=T\xi(x) \,\, (\xi\in \mathcal{H}(\text{Ind}_H^G\tilde{\rho}), x\in G).$$ It can be shown that $\Phi\colon L(H) \to \B(L^2(G));T \mapsto U^*\widehat{T}U$ gives a desired $*$-isomorphism can be shown using \cite{Blattner1962}. Moreover, $\Phi(\lambda_\ell^H(h))=\lambda_\ell^G(h)\,\,(h\in H)$ holds.
\end{proof}

We give the definition of the canonical weight of $L(G)$.

\begin{definition}
    For $\xi \in L^2(G)$, we say that $\xi$ is left bounded if there is a bounded operator $\pi_{l}(\xi) \in\B(L^2(G))$ such that $\pi_{l}(\xi)f=\xi * f$ for all $f \in C_c(G)$. Similarly, $\xi$ is right bounded if there is a bounded operator $\pi_{r}(\xi) \in\B(L^2(G))$ such that $\pi_{r}(\xi)f=f*\xi$ for all $f \in C_c(G)$.
\end{definition}

\begin{definition}\label{defofplweight}
    For a locally compact group $G$, define the weight $\varphi_{G}$ on $L(G)_+$ as \begin{equation*}
        \varphi_{G}(x)=
        \begin{cases}
            \|\xi\|_2^2 & x=\pi_{l}(\xi)^*\pi_{l}(\xi) \,\,\text{for some left bounded}\,\, \xi \in L^2(G) \\
            \infty & \text{otherwise}.
        \end{cases}
    \end{equation*}
    It is known to be a normal faithful semifinite weight. We call it the Plancherel weight of $L(G)$. Note that $\varphi_G$ depends on the left Haar measure $\mu_G$ on $G$ up to a constant.
\end{definition}

It is also known that $G$ is unimodular if and only if its Plancherel weight is tracial. The following lemmas are well known. We include the proof for completeness.  

\begin{lemma}
    Any $f\in C_c(G)$ is both left and right bounded.
\end{lemma}
\begin{proof}
    The left boundness follows from the inequality $$\|f*\xi\|_2 \le \|f\|_1\|\xi\|_2\,\, \text{for all} \,\, \xi \in L^2(G).$$
    Thus, we have $\pi_{l}(f)\xi=f*\xi$ for all $\xi \in L^2(G)$. For the right boundness, fix $\xi \in L^2(G)$. Then, $$\xi*f(x)=\int\xi(y)f(y^{-1}x)\,dy = \int\xi(xy)f(y^{-1})\,dy = \int\Delta_G(y)^{-1/2}\rho_y\xi(x)f(y^{-1})\,dy.$$
    Hence, $$\xi*f=\int\Delta_G(y)^{-1/2}f(y^{-1})\rho_y\xi(\cdot)\,dy.$$
    By Minkowski’s inequality for integrals, we obtain 
    \begin{align*}
        \|\xi*f\|_2 \le \int\|\Delta_G(y)^{-1/2}f(y^{-1})\rho_y\xi\|_2\,dy & =  \int |\Delta_G(y)^{-1/2}f(y^{-1})|\|\rho_y\xi\|_2\,dy \\ & = \|\xi\|_2\int |\Delta_G(y)^{-1/2}f(y^{-1})|\,dy \le C\|\xi\|_2
    \end{align*}
    for some $C > 0$. This shows that $f$ is right bounded and $\pi_r(f)\xi=\xi*f$ for all $\xi \in L^2(G)$.
\end{proof}

Now, we introduce some notation. For any complex function $f$ on $G$, we put  $$f^{\#}(x)= \Delta_G(x)^{-1}\overline{f}(x^{-1}),\,\,\, f^*(x)= \Delta_G(x)^{-1/2}\overline{f}(x^{-1}), \,\,\, f^{\flat}(x)=\overline{f}(x^{-1}).$$ 
$\#$, * and $\flat$ are involutions of the convolution algebra $C_c(G)$. Moreover, * is 
an isometry on $L^2(G)$. $\#$ makes $C_c(G)$ into the left Hilbert algebra and $\flat$ makes $C_c(G)$ into the right Hilbert algebra with $L^2$ inner product. 

\begin{lemma}\label{sumprodadj}
    For $f,g \in C_c(G)$, $$\pi_l(f+g)=\pi_l(f)+\pi_l(g),\,\,\, \pi_l(f^{\#}*g)=\pi_l(f)^*\pi_l(g),\,\,\, \pi_l(f^{\#})=\pi_l(f)^*.$$
\end{lemma}
\begin{proof}
    Fix $f,g \in C_c(G)$. It is obvious that $f+g, f^{\#},f^{\#}*g\in C_c(G)$. Then, for each $\xi, \eta \in C_c(G)$, $$(\pi_l(f)+\pi_l(g))\xi=\pi_l(f)\xi+\pi_l(g)\xi=f*\xi+g*\xi=(f+g)*\xi,$$
    $$\nai{\pi_l(f)^*\pi_l(g)\xi}{\eta}_2=\nai{\pi_l(g)\xi}{\pi_l(f)\eta}_2=\nai{g*\xi}{f*\xi}_2=\nai{f^{\#}*g*\xi}{\eta}_2
    =\nai{(f^{\#}*g)*\xi}{\eta}_2,$$ $$\nai{\pi_l(f)^*\xi}{\eta}_2=\nai{\xi}{\pi_l(f)\eta}_2=\nai{\xi}{f*\eta}_2=\nai{f^{\#}*\xi}{\eta}_2.$$ These calculations implies desired formulas.
\end{proof}

\begin{lemma}\label{rbddimp*lbdd}
    If $g \in L^2(G)$ is right bounded, then $g^*$ is left bounded.
\end{lemma}
\begin{proof}
    For $f \in C_c(G)$, 
    $$\|g^**f\|_2= \|(g^**f)^*\|_2 = \|f^**g\|_2 \le \|f^*\|_2\|\pi_r(g)\|= \|f\|_2\|\pi_r(g)\|.$$
    Hence, there is a bounded operator $\pi_l(g^*) \in \B(L^2(G))$ such that $\pi_l(g^*)f= g^**f$ for all $f \in C_c(G)$. This shows that $g^*$ is left bounded.
\end{proof}

In fact, the formula $\pi_l(f^{\#}*g)=\pi_l(f)^*\pi_l(g)$ holds for left bounded $L^2$ functions $f,g$ (c.f. \cite[Lemma 2.8]{Haagerupdualweight}).
In addition, we define $$ P (G) =\text{the set of all continuous functions of positive type on G.}$$
Following \cite[P 126]{Haagerupdualweight}, we write $\varphi\ll K\cdot \delta$ for $\varphi\in P(G)$ and $K>0$, if there exists a $K>0$ such that $\int\varphi(x)(f^{\#}*f)(x)\, dx \le K(f^{\#}*f)(e)$ for all $f \in C_c(G)$.

\begin{lemma}\label{rbddlesskdelta}
    For any right bounded $f\in L^2(G) \cap P(G)$,  $f \ll K \cdot \delta$ holds, where $K=\|\pi_r(f)\|$.
\end{lemma}
\begin{proof}
    Fix $\xi \in C_c(G)$, then
    \begin{align*}
       0\le \int f(x)(\xi^{\#}*\xi)(x)\, dx &= \nai{f}{\overline{\xi^{\#}*\xi}}_2 =\nai{f}{\overline{\xi}^{\#}*\overline{\xi}}_2 = \nai{f}{\pi_l(\overline{\xi})^*\overline{\xi}}_2 = \nai{\pi_l(\overline{\xi})f}{\overline{\xi}}_2 = \nai{\overline{\xi}*f}{\overline{\xi}}_2 \\ &= \nai{\pi_r(f)\overline{\xi}}{\overline{\xi}}_2 \le \|\pi_r(f)\|\|\overline{\xi}\|_2^2 = \|\pi_r(f)\|\|\xi\|_2^2 = \|\pi_r(f)\|(\xi^{\#}*\xi)(e).
    \end{align*}
    We use the equation $\pi_l(\xi^{\#})=\pi_l(\xi)^*$ for $\xi \in C_c(G)$. 
\end{proof}

\begin{lemma}\label{convresol}
    Suppose $\varphi \in P(G)$ and $\varphi \ll K \cdot \delta$ for some $K>0$. Then, there exists a right bounded $g \in L^2(G)$ such that $\varphi=g*g^{\flat}$.
\end{lemma}
\begin{proof}
    We apply the proof of \cite[Proposition 2.9]{Haagerupdualweight}.
    Let $(\pi_{\varphi}, \mathcal{H}_{\varphi}, \xi_{\varphi})$ be the cyclic representation induced by $\varphi$. For $f \in C_c(G)$, 
    \begin{align*}
        \|\pi_{\varphi}(f)\xi_{\varphi}\|^2 = \nai{\pi_{\varphi}(f^{\#}*f)\xi_{\varphi}}{\xi_{\varphi}} &= \int (f^{\#}*f)(x)\nai{\pi_{\varphi}(x)\xi_{\varphi}}{\xi_{\varphi}}\, dx \\ &=\int (f^{\#}*f)(x) \varphi(x)\, dx \le K(f^{\#}*f)(e) = K\|f\|_2^2 .\end{align*}
    Hence, there exists a bounded operator $T\colon L^2(G) \to \mathcal{H}_{\varphi}$ such that $Tf=\pi_{\varphi}(f)\xi_{\varphi}$ for all $f \in C_c(G)$. A simple calculation shows that $T\lambda(f)=\pi_{\varphi}(f)T \,\, ( f \in C_c(G))$. Indeed, since $\lambda(f)=\pi_l(f)$ for all $f \in C_c(G)$, for $g \in C_c(G)$, 
    $$T\lambda(f)g=T(f*g)=\pi_{\varphi}(f*g)\xi_{\varphi} = \pi_{\varphi}(f)\pi_{\varphi}(g)\xi_{\varphi} = \pi_{\varphi}(f)Tg.$$
    Let $T=U|T|$ be the polar decomposition. It follows that $U\lambda(f)=\pi_{\varphi}(f)U$ and $UU^*=P_{\overline{\text{ran}}(T)}=P_{\mathcal{H}_{\varphi}}=1$ since $\xi_{\varphi}$ is cyclic. In particular, $U\lambda_x=\pi_{\varphi}(x)U \,\,( x \in G)$. Put $\xi=U^*\xi_{\varphi} \in L^2(G)$ and $g=\overline{\xi}$. For $x \in G$, 
    \begin{align*}
        \varphi(x) & = \nai{\pi_{\varphi}(x)\xi_{\varphi}}{\xi_{\varphi}} = \nai{\pi_{\varphi}(x)UU^*\xi_{\varphi}}{\xi_{\varphi}} =\nai{U\lambda_xU^*\xi_{\varphi}}{\xi_{\varphi}} = \nai{\lambda_x\xi}{\xi}_2 \\ &= \int \xi(x^{-1}y)\overline{\xi}(y)\, dy = \int \overline{\xi}^{\flat}(y^{-1}x) \overline{\xi}(y) \, dy = (g*g^{\flat})(x). 
    \end{align*}
    It remains to show that $g$ is right bounded. It suffices to show the right boundness of $\xi$. For $f \in C_c(G)$, 
    \begin{align*}
        \|f*\xi\|_2^2=\nai{f*\xi}{f*\xi}_2
        & =\nai{f^{\#}*f*\xi}{\xi}_2 = \nai{\lambda(f^{\#}*f)\xi}{\xi}_2 \\ & = \int (f^{\#}*f)(x)\nai{\lambda_x\xi}{\xi}_2\, dx = \int (f^{\#}*f)(x)\varphi(x) \, dx \le K\|f\|_2^2 .
    \end{align*}
    Hence, there exists $\pi_r(\xi)\in \B(L^2(G))$ such that $\pi_r(\xi)f = f* \xi $ for all $f \in C_c(G)$.
\end{proof}

Now, we shall show that $\varphi_G|L(H)=C\cdot\varphi_H$ for some constant $C >0$ if and only if $H$ is open. 
Before we prove the theorem, we give a necessary and sufficient condition for the semifiniteness of the restriction of the Plancherel weight. Professor Reiji Tomatsu informed us that this was shown in the case of quantum groups (see e.g. \cite[Theorem 7.5]{KKS2016}). However, we provide the group von Neumann algebraic proof for reader's convenience.
\par
We recall a strongly quasi-invariant Radon measure on $G/H$ and a rho-function for the pair $(G,H)$. We refer to \cite[Section 2.6]{Folland} for the details. Let $G$ be a locally compact group and $H$ be a closed subgroup of $G$. Then, there exists a continuous function $\rho\colon G \to (0,\infty)$ such that 
$$\rho(xh)=\frac{\Delta_H(h)}{\Delta_G(h)}\rho(x)\,\,(g\in G, h\in H).$$ This function is called a rho-function. Moreover, suppose $\mu$ is a Radon measure on $G/H$. For $x \in G$, we define the translate $\mu_x$ of $\mu$ by $\mu_x(E)=\mu(xE)$, where $E$ is a Borel subset of $G/H$. $\mu$ is said to be strongly quasi-invariant if there is a continuous function $\lambda\colon G \times G/H \to(0,\infty)$ such that $d\mu_x(p)=\lambda(x,p)d\mu(p)$ for all $x\in G, p\in G/H$. A rho-function $\rho$ produces a strongly quasi-invariant measure $\mu$ on $G/H$ which satisfies 
$$\int_Gf(x)\rho(x)\,d\mu_G(x)=\int_{G/H}\int_H f(xh)\,d\mu_H(h)d\mu(xH)\,\,(f\in C_c(G)).$$ 
\par
Now we give a characterization of openness of $H$. This is likely to be known, but we include the proof for completeness.

\begin{proposition}\label{openiffmuposi}
    For a locally compact group $G$, a closed subgroup $H$ of $G$ and $\mu$ is a strongly quasi-invariant Radon measure on $G/H$ obtained by a rho-function $\rho$, the following conditions are equivalent.
    \begin{enumerate}
        \item[(i)] $H$ is open.
        \item[(ii)] $\mu(\{H\}) >0$.
    \end{enumerate}
\end{proposition}
\begin{proof}
    
    $\text{(i)}\implies \text{(ii)}.$ Follows from \cite[Proposition 2.60]{Folland}.
    %Assume $\mu(\{H\})=0$. Set $C_c(G)_H=\{f\in C_c(G) : {\rm supp}\,f \subseteq H\}$. Note that $C_c(G)_H \ne \{0\}$ since $H$ is open. Then the formula $$\int_G f(x)\rho(x)\, d\mu_G(x) = \mu(\{H\})\int_H f(x)\, d\mu_H(x)\,\,(f\in C_c(G)_H)$$ implies $\int_G f(x)\rho(x)\, d\mu_G(x) = 0$ for all $f\in C_c(G)_H$. Moreover, since $f/\rho \in C_c(G)_H$, we obtain $\int_G f(x)\, d\mu_G(x) = 0$ for all $f\in C_c(G)_H$. By the inner regularity of $\mu_G$, there exists a compact set $K \subseteq G$ such that $K \subseteq H$ and $0<\mu_G(K)\le \mu_G(H)$. Then, by Urysohn's lemma, there exists a $f\in C_c(G)_H$ which satisfies $0\le f\le 1$ and $f=1$ on $K$. It follows that $0<\int_G f\,d\mu_G $, which is a contradiction.
   
    \par
    $\text{(ii)}\implies \text{(i)}.$ Note that $d\mu_x(yH)=\rho(xy)/\rho(y)d\mu(yH)$ for all $x,y\in G$ (see \cite[Theorem 2.58]{Folland}). Thus, we obtain 
    \begin{align*}
    \mu(\{xH\})=\mu_x(\{H\})&=\int_{G/H}1_{\{H\}}(yH)\,d\mu_x(yH)\\&=\int_{G/H}1_{\{H\}}(yH)\rho(xy)/\rho(y)\,d\mu(yH)=\frac{\rho(x)}{\rho(e)}\mu(\{H\})
    \end{align*}
    for all $x\in G$. Note that this equation implies that $\rho$ is constant on each coset $xH \subseteq G$. Let $U \subseteq G/H$ be an open subset and $K \subseteq G/H$ be a compact subset, with $H \in U \subseteq K$. By \cite[Lemma 2.48]{Folland}, there exists a compact subset $K_0 \subseteq G$ such that $q(K_0)=K$, where $q\colon G\to G/H$ is the canonical quotient map. Then, since $\rho$ is continuous and $0< \rho < \infty$, we find a constant $M >0$ which satisfies $\rho(x) \ge M$ for all $x \in K_0$. We see that $\rho(x) \ge M$ for all $x \in G$ such that $xH\in K$. In this setting, for each finite subset $F \subseteq U$, we have 
    $$\mu(F)=\sum_{xH\in F}\mu(\{xH\})=\sum_{xH\in F}\frac{\rho(x)}{\rho(e)}\mu(\{H\})\ge\#F\cdot \frac{M}{\rho(e)}\mu(\{H\}).$$
    Combining the inequation $\mu(F) \le \mu(K)$, we obtain $\#F\le\frac{\rho(e)\mu(K)}{M \mu(\{H\})} < \infty$. Since $F$ is arbitrary, we conclude that $U$ is a finite set. Let $U=\{x_1H=H,x_2H,\ldots,x_nH\}$. If $n=1$, then $\{H\}$ is open in $G/H$. If $n \ge2$, choose an open neighborhood $V$ of $H\in G/H$ such that $x_2H,\ldots,x_nH \notin V$. Then, $\{H\}=U\cap V$ is open in $G/H$. In either case, $G/H$ is discrete, and hence $H$ is open in $G$.
\end{proof}

\begin{lemma}\label{openimpmodcondi}
    For a locally compact group $G$ and an open subgroup $H$, $\Delta_G|H=\Delta_H$ holds.
\end{lemma}
\begin{proof}
    Take an arbitrary rho-function $\rho$. There is a strongly quasi-invariant measure $\mu$ obtained by $\rho$. By Proposition \ref{openiffmuposi}, $\mu(\{H\})>0$ holds. Moreover, by the proof of $\text{(ii)}\implies \text{(i)}$, it follows that $\rho$ is constant on each coset $xH\,\,(x\in G)$. Thus, for $x\in G, h\in H$, we have $\rho(x)=\rho(xh)=\frac{\Delta_H(h)}{\Delta_G(h)}\rho(x)$. Since $0 < \rho$, we conclude that $\Delta_H(h)=\Delta_G(h)$ for all $h \in H$.
\end{proof}

We provide a sufficient condition for the semifiniteness of the restriction of the Plancherel weight. Note that under the assumption $\Delta_G|H=\Delta_H$, the rho-function $\rho$ can be taken as $\rho \equiv 1$. In this case, the strongly quasi-invariant Radon measure $\mu$ is actually the $G$-invariant Radon measure. 
 
\begin{proposition}\label{sfofrest}
    Let $G$ be a locally compact group and $H$ be an open subgroup of $G$. Then $\varphi_G|L(H)$ is semifinite.
\end{proposition}
\begin{proof}
     Note that $C_c(G)_H=\{f\in C_c(G) : {\rm supp}\,f \subseteq H\} \ne \{0\}$ since $H$ is open. Let $\mu$ be the $G$-invariant Radon measure on $G/H$ obtained by the rho-function $\rho \equiv 1$. Also, we set $C=\mu(\{H\})$ and $\mathcal{N}_H=\{\lambda_h:h\in H\}$. It follows that $C>0$ by Proposition \ref{openiffmuposi}. Since $L(H) \cong \mathcal{N}_H''$, we shall show that $\varphi_G|\mathcal{N}_H''$ is semifinite. We now prove $\Phi(\pi_l(C_c(H)))=\pi_l(C_c(G)_H)$ where $\Phi$ is the isomorphism as in Proposition \ref{subisom}. Note that $\int_G f(x)\, d\mu_G(x) = \mu(\{H\})\int_H f(x)\, d\mu_H(x)$ for all $f\in C_c(G)_H$. Take $f \in C_c(H)$. Fix $\xi \in C_c(G), x\in G$, then
    \begin{align*}
        \Phi(\pi_l(f))(\xi)(x) &= (U^*\widehat{\pi_l(f)}U\xi)(x) \\ &= \Delta_G(x)^{-1/2}(\widehat{\pi_l(f)}U\xi)(x^{-1})(e) \\ &= \Delta_G(x)^{-1/2}(\pi_l(f)(U\xi)(x^{-1}))(e) \\ &= \Delta_G(x)^{-1/2}(f*(U\xi)(x^{-1}))(e) \\ &= \Delta_G(x)^{-1/2}\int_H f(y)(U\xi)(x^{-1})(y^{-1})\, d\mu_H(y) \\ &=\Delta_G(x)^{-1/2} \int_H f(y)\Delta_G(x)^{1/2} \xi(y^{-1}x) \, d\mu_H(y) \\ &= \int_H f(y) \xi(y^{-1}x) \, d\mu_H(y).
    \end{align*}
    If we denote $\tilde{f} \in C_c(G)_H$ for the zero extension of $f$, then we obtain
    \begin{align*}
        \Phi(\pi_l(f))(\xi)(x)=\int_H f(y) \xi(y^{-1}x) \, d\mu_H(y) &= \int_H\tilde{f}(y)  \xi(y^{-1}x) \, d\mu_H(y) \\ &= C^{-1}\int_G \tilde{f}(y)\xi(y^{-1}x) \, d\mu_G(y) \\&=(\pi_l(C^{-1}\tilde{f})\xi)(x).
    \end{align*}
    This shows that $\Phi(\pi_l(f))=\pi_l(C^{-1}\tilde{f})$. It is clear that $C_c(H)\ni f \mapsto C^{-1}\tilde{f}\in C_c(G)_H$ is bijective. Therefore, we conclude that $\Phi(\pi_l(C_c(H)))=\pi_l(C_c(G)_H)$. 
    
    Next, we show $\pi_l(C_c(G)_H) \subseteq \mathfrak{n}_{\varphi_G|\mathcal{N}_H''}=\mathfrak{n}_{\varphi_G}\cap \mathcal{N}_H''$. It is clear that $\pi_l(C_c(G)_H) \subseteq \mathfrak{n}_{\varphi_G}$. Take $f\in C_c(G)_H$. For each $T\in \mathcal{N}_H'$, we have 
    \begin{align*}
        T\pi_l(f) &= T\int_G f(x)\lambda_x\, d\mu_G(x) = T\int_H f(x)\lambda_x\, d\mu_G(x) \\ &= \int_H f(x)T\lambda_x\, d\mu_G(x) = \int_H f(x)\lambda_xT\, d\mu_G(x) = \pi_l(f)T.
    \end{align*}
    Thus, $\pi_l(f)\in \mathcal{N}_H''$. To show the semifiniteness of $\varphi_G|\mathcal{N}_H''$, it suffices to show that $\pi_l(C_c(G)_H)''=\mathcal{N}_H''$. Since $\Phi$ is an isomorphism between von Neumann algebras $L(H)$ and $\mathcal{N}_H''$, and hence is normal, we have  
    \begin{align*}
        \pi_l(C_c(G)_H)''&= \Phi(\pi_l(C_c(H)))''=\Phi(\pi_l(C_c(H))'')=\Phi(L(H))\\&=\Phi(\{\lambda_\ell^H(h):h\in H\}'')=\Phi(\{\lambda_\ell^H(h):h\in H\})''=\mathcal{N}_H''.
    \end{align*}
\end{proof}

Now, we provide a necessary condition for the semifiniteness of the restriction of the Plancherel weight. We recall the support of an element of the group von Neumann algebra. We set $$A(G)=\{\xi*\eta^{\flat}:\xi,\eta\in L^2(G)\}.$$ 
$A(G)$ is called the Fourier algebra of $G$. It is well known that $A(G)$ is a dense $*$-subalgebra of the abelian $C^*$-algebra $C_0(G)$ of all continuous functions on $G$ vanishing at infinity, with $\|\cdot\|_\infty$-norm (see, e.g. \cite[Chapter VII, Lemma 3.7]{TakesakiOA2}). As a Banach space, $A(G)$ is isomorphic to the predual $L(G)_*$ of $L(G)$. The duality pairing is given by $\nai{x}{\phi}=\nai{x\xi}{\eta}_2$ for all $x\in L(G),\phi=\xi*\eta^{\flat} \in A(G)$. We see that $\nai{\lambda_x}{\phi}=\phi(x)\,\, (x\in G)$. Now, we define the multipliers on $L(G)$. For $\phi \in A(G)$, a normal completely bounded map $m_\phi\colon  L(G) \to L(G)$ is defined by the formula $$\nai{m_\phi(x)}{\psi}=\nai{x}{\psi\phi}\,\,(x\in L(G),\psi \in A(G)).$$ In this setting, $m_\phi(\lambda_x)=\phi(x)\lambda_x\,\,(x\in G)$ holds. 

We define the support of $x\in L(G)$. For $x\in L(G)$, set
$${\rm supp}(x)=\{g\in G: m_\phi(x)\ne0\,\, \text{for all}\,\,\phi\in A(G)\,\, \text{such that}\,\, \phi(g) \ne 0 \}.$$ It is known that the support is closed in $G$. We refer to \cite[Chapter 4]{Eymard1964}, \cite{TakesakiTatsuuma1972} and \cite[Section 2.6, Section 3.1]{BoutonnetBrothier2016} for details.
We provide the following lemmas. In fact, they were shown in \cite[Proposition 3.2, Theorem 5.2]{BoutonnetBrothier2016}. But, we give an outline of the proof in the simple case for completeness.

\begin{lemma}[{\cite[Proposition 3.2]{BoutonnetBrothier2016}}]\label{suppinc1}
    Let $G$ be a locally compact group, $x\in L(G)$ and $g \in G$. If $1_{gU}x1_U \ne 0$ for any open neighborhood $U$ of $e$, then $g \in {\rm supp}(x)$ holds.
\end{lemma}
\begin{proof}
    Let $\phi \in A(G)$ such that $\phi(g) \ne 0$. We shall show that $m_\phi(x) \ne 0$. Let $\mathcal{I}(\phi)$ be the closed ideal generated by $\phi$. Since $\psi_\lambda \to \psi$ in $A(G)$ implies $\nai{y}{\psi_\lambda-\psi} \to 0\,\,(y\in L(G))$, if $\psi_\lambda \to \psi$ in $A(G)$, then $m_{\psi_\lambda}(y) \to m_\psi(y)$ $\sigma$-weakly for all $y\in L(G)$ holds. Note that $\mathcal{I}(\phi)$ is the closure of the set ${\{\phi\psi:\psi\in A(G)\}}$ and $m_\phi(x)=0$ implies $m_{\phi\psi}(x)=0$ for all $\psi \in A(G)$. We conclude that if $m_\phi(x)=0$, then we have $m_\psi(x)=0$ for all $\psi \in \mathcal{I}(\phi)$. Therefore, to show $m_\phi(x) \ne 0$, it suffices to find $\psi \in \mathcal{I}(\phi)$ such that $m_\psi(x) \ne 0$. Set $$F=\{g\in G: \phi(g) =0\},\,\,h(\mathcal{I}(\phi))=\{g\in G: \psi(g)=0\,\,\text{for all}\,\,\psi \in \mathcal{I}(\phi)\}.$$ In fact, since the spectrum of the commutative Banach algebra $A(G)$ is known to be $G$, $h(\mathcal{I}(\phi))$ is the hull of $\mathcal{I}(\phi)$ in the sense of \cite[Definition 4.1.1]{Kaniuth2009}. Moreover, it can be shown that $F= h(\mathcal{I}(\phi))$. Take some compact neighborhood $V$ of $g$ such that $\phi \ne 0$ on $V$. Noting that the identification map between the spectrum of $A(G)$ and $G$ is the identity map (see \cite[Theorem 2.9.4]{Kaniuth2009}), by \cite[Theorem 4.2.8]{Kaniuth2009} (or \cite[Lemma 25C]{Loomis1953}), there exists $\psi \in \mathcal{I}(\phi)$ which satisfies $\psi =1$ on $V$. Take an open neighborhood $U$ of $e$ such that $gU\cdot U^{-1} \subseteq V$. We shall show that $$1_{gU}m_\psi(y)1_U=1_{gU}y 1_U\,\,(y\in L(G)).$$ If $h\in V$, then since $m_\psi(\lambda_h)=\psi(h)\lambda_h=\lambda_h$, the equation is obvious. Let $h \notin V$. For $\xi \in C_c(G)$ and $s\in G$, we have $$(1_{gU}m_\psi(\lambda_h)1_U\xi)(s)=\psi(h)1_{gU}(s)1_U(h^{-1}s)\xi(h^{-1}s)=\psi(h)(1_{gU}\lambda_h1_U\xi)(s).$$ By construction, the value of the above equation is zero. Therefore, we see that $1_{gU}m_\psi(y)1_U=1_{gU}y 1_U$ for all $y\in L(G)$ since $\{\lambda_h:h\in G\}''=L(G)$. We conclude that $1_{gU}m_\psi(x)1_U=1_{gU}x 1_U\ne 0$ and hence $m_\psi(x)\ne 0$.
\end{proof} 

\begin{lemma}[{\cite[Theorem 5.2]{BoutonnetBrothier2016}}]\label{sptinc}
    Let $f$ be a continuous, left bounded function in $L^2(G)$. Then the inclusion ${\rm supp}(f)\subseteq {\rm supp}(\pi_l(f))$ holds.
\end{lemma}
\begin{proof}
    Let $g\in {\rm supp}(f)$. By Lemma \ref{suppinc1}, we shall show that $1_{gU}\pi_l(f)1_U \ne 0$ for any open neighborhood $U$ of $e$. Fix an open neighborhood $U$ of $e$. For $\xi \in C_c(G)_U=\{f\in C_c(G) : {\rm supp}\,f \subseteq U\}$, it follows that $1_{gU}\pi_l(f)1_U \xi=1_{gU}\pi_l(f)\xi=1_{gU}\pi_r(\xi)f$. Since $f$ is continuous, $gU \cap \{h\in G: f(h)\ne 0\}$ is a non-empty open set. Thus, $1_{gU}f\ne0$ holds. Put $C=\|1_{gU}f\|_2>0$. For $\xi \in C_c(G)$, we denote $\check{\xi}$ for the function $h\mapsto \xi(h^{-1})$. Let $\xi \in C_c(G)_+$ such that $\int_G \xi \,d\mu_G =1$. We calculate
    \begin{align*}
        \|\pi_r(\check{\xi})f-f\|_2^2&=\int_G|\pi_r(\check{\xi})f(h)-f(h)|^2\,dh
        \\&= \int_G|\int_G\xi(s)(f(hs)-f(h))\,ds|^2dh
        \\&\le \int_G(\int_G\xi(s)|f(hs)-f(h)|\,ds)^2dh
        \\&= \int_{G \times G \times G}\xi(s)\xi(t)|f(hs)-f(h)||f(ht)-f(h)|\,ds\,dt\,dh
        \\&= \int_{G \times G \times G}\xi(s)\xi(t)|\Delta_G(s)^{-1/2}\rho_sf(h)-f(h)||\Delta_G(t)^{-1/2}\rho_tf(h)-f(h)|\,ds\,dt\,dh
        \\&\le \int_{G\times G} \xi(s)\xi(t)\|\Delta_G(s)^{-1/2}\rho_sf-f\|_2\|\Delta_G(t)^{-1/2}\rho_tf-f\|_2\,ds\,dt 
        \\&= (\int_G \xi(s)\|\Delta_G(s)^{-1/2}\rho_sf-f\|_2\,ds)^2.
    \end{align*}
    Since the map $s \mapsto \|\Delta_G(s)^{-1/2}\rho_sf-f\|_2$ is continuous and the value at $e$ is zero, if we take the support of $\xi$ sufficiently small, then $\|\pi_r(\check{\xi})f-f\|_2 \le \int_G \xi(s)\|\Delta_G(s)^{-1/2}\rho_sf-f\|_2\,ds < \frac{C}{2}$ holds. Noting that $$\|1_{gU}f\|_2-\|1_{gU}\pi_r(\check{\xi})f\|_2 \le\|1_{gU}f-1_{gU}\pi_r(\check{\xi})f\|_2 \le \|f-\pi_r(\check{\xi})f\|_2,$$ we conclude that $$\|1_{gU}\pi_r(\check{\xi})f\|_2\ge \|1_{gU}f\|_2- \|f-\pi_r(\check{\xi})f\|_2 >\frac{C}{2}>0$$ for some $\xi \in C_c(G)_U$.
    That is, $1_{gU}\pi_l(f)1_U \check{\xi}=1_{gU}\pi_r(\check{\xi})f\ne0$ and hence $1_{gU}\pi_l(f)1_U \ne0$ holds.
\end{proof}

\begin{proposition}\label{sfimpopen}
    For a locally compact group $G$ and a closed subgroup $H$ of $G$, if $\varphi_G|L(H)$ is semifinite, then $H$ is open.
\end{proposition}
\begin{proof}
    We identify $L(H)$ with $\mathcal{N}_H''$. Since $\varphi_G|\mathcal{N}_H''$ is semifinite, there is a nonzero $x\in \mathfrak{m}_{\varphi_G}^+ \cap \mathcal{N}_H''$. Then, by \cite[Proposition 2.9 (2)]{Haagerupdualweight}, $x=\pi_l(f)$ for some continuous left bounded $f\in L^2(G)$. Lemma \ref{sptinc} implies that ${\rm supp}(f) \subseteq {\rm supp}(\pi_l(f))$. Moreover, by \cite[Theorem 3]{TakesakiTatsuuma1972}, ${\rm supp}(\pi_l(f)) = {\rm supp}(x)\subseteq H$ holds. Since $f$ is continuous, nonzero and ${\rm supp}(f) \subseteq H$, it follows that $H$ is open. 
\end{proof}

Now, we obtain the following theorem. As mentioned earlier, (ii)$\implies$(i) was shown in \cite[Theorem 7.5]{KKS2016} in the context of quantum groups, and we learned about this result from Professor Reiji Tomatsu.

\begin{theorem}\label{openiffsf}
    For a locally compact group $G$ and a closed subgroup $H$, the following conditions are equivalent.
    \begin{enumerate}
        \item[(i)] $H$ is open.
        \item[(ii)] $\varphi_G|L(H)$ is semifinite.
     \end{enumerate}
\end{theorem}
\begin{proof}
    Proposition \ref{sfofrest} and Proposition \ref{sfimpopen} imply the desired equivalence.
\end{proof}

\begin{corollary}
    For a locally compact group $G$ and a closed subgroup $H$, there exists a $\varphi_G$-preserving conditional expectation from $L(G)$ onto $L(H)$ if and only if $H$ is open.
\end{corollary}
\begin{proof}
    Since $L(H)$ is globally $\sigma^{\varphi_G}$-invariant, the theorem follows from \cite{Takesaki72MR303307} and Theorem \ref{openiffsf}.
\end{proof}

As another corollary, we also recover a part of Garcia Guinto--Nelson's theorem.
\begin{corollary}[{\cite[Theorem 2.1]{guinto2025unimodulargroups}}]
    For a locally compact group $G$, the Plancherel weight of $G$ is strictly-semifinite if and only if $G_1$ is open in $G$, where $G_1$ is the kernel of $\Delta_G$.
\end{corollary}
\begin{proof}
    Note that $\varphi_G$ is strictly-semifinite if and only if $\varphi_G|L(G)^{\varphi_G}$ is semifinite. It is known that $L(G)^{\varphi_G} \cong L(G_1)$ (see \cite[Theorem 6.1]{Sutherland1973}). Thus, Theorem \ref{openiffsf} implies the desired equivalence.
\end{proof}
We remark that in \cite[Theorem 2.1]{guinto2025unimodulargroups}, the strict semifiniteness of $\varphi_G$ is also shown to be equivalent to its almost periodicity. In $\S 4$, we will use the almost periodicity of $\varphi_G$ to characterize the factoriality of $L(G)$ for almost unimodular groups. We also remark that any totally disconnected locally compact group is almost unimodular \cite[Example 2.4(4)]{guinto2025unimodulargroups}.

In the rest of this section, we show that $\varphi_G|L(H)=C\cdot\varphi_H$ for some constant $C >0$ if $\varphi_G|L(H)$ is semifinite. Note that, when $H$ is an open subgroup of $G$, the pair $(G,H)$ admits the rho-function $\rho\equiv 1$ by Lemma \ref{openimpmodcondi}.

The proof is due to following Pedersen-Takesaki theorem (see also \cite[Chapter VIII, Proposition 3.15]{TakesakiOA2}). 

\begin{proposition}[{\cite[Proposition 5.9]{PedersenTakesaki1973}}]\label{TakTats}
    Let $M$ be a von Neumann algebra. Let $\varphi$ be a normal semifinite faithful weight on $M$. If $\psi$ is a normal semifinite weight on $M$ commuting with $\varphi$, that is, $\psi=\psi \circ \sigma_t^{\varphi} \,\, ( t \in \R)$, and if there exists a $\sigma$-weakly dense *-subalgebra $\mathfrak{m}_0$ of $\mathfrak{m}_{\varphi}$ such that $\sigma_t^{\varphi}(\mathfrak{m}_0)=\mathfrak{m}_0\,\, ( t \in \R)$ and $\varphi(x)= \psi(x)$ for every $x \in \mathfrak{m}_0$, then $\varphi=\psi$.
\end{proposition}

\begin{theorem}\label{restofplweight}
    Let $G$ be a locally compact group and $H$ be an open subgroup of $G$. Then $\varphi_G|L(H)=C^{-1}\cdot\varphi_H$. Here, $C=\mu(\{H\})$ where $\mu$ is the $G$-invariant Radon measure obtained by the rho-function $\rho \equiv 1$.
\end{theorem}
\begin{proof}
    Put $\psi = \varphi_G|L(H)$. By Proposition \ref{sfofrest}, $\psi$ is semifinite. Note that $\Delta_G|H=\Delta_H$ holds by Lemma \ref{openimpmodcondi}. Since $\sigma_t^{\varphi_G}(L(H))=L(H)$, we have $$\sigma_t^{\psi}(\lambda_h)=\sigma_t^{\varphi_G}(\lambda_h)=\Delta_G(h)^{it}\lambda_h=\Delta_H(h)^{it}\lambda_h=\sigma_t^{\varphi_H}(\lambda_h)=\sigma_t^{C^{-1}\cdot\varphi_H}(\lambda_h)$$
    for all $h \in H$. Thus, $\sigma_t^{\psi}(x)=\sigma_t^{C^{-1}\cdot\varphi_H}(x)$ for all $x \in L(H)$ and then $\psi \circ \sigma_t^{C^{-1}\cdot\varphi_H}=\psi\circ\sigma_t^{\psi} = \psi$. \par
    Now, put $A=\text{span}\{f*g\,|\,f,g\in C_c(H)\} \subseteq C_c(H)$ and we will show that $\pi_l(A)$ satisfies the conditions of Proposition \ref{TakTats} for $\mathfrak{m}_0$.\par
    First, we show $\pi_l(A)$ is a $\sigma$-weakly dense *-subalgebra of $\mathfrak{m}_{\varphi_H}$ because $\mathfrak{m}_{\varphi_H}=\mathfrak{m}_{C^{-1}\cdot\varphi_H}$. Lemma \ref{sumprodadj} ensures that $\pi_l(A)$ is a *-subalgebra of $\B(L^2(H))$. Indeed, $\pi_l(f)\pi_l(g)=\pi_l(f^{\#})^*\pi_l(g)=\pi_l((f^{\#})^{\#}*g)=\pi_l(f*g)\in \pi_l(A) \,\,( f,g \in C_c(H))$.
    Moreover, for $a=\sum_{i=1}^{n}f_i*g_i \in A$, $$\pi_l(a)=\sum_{i=1}^{n}\pi_l(f_i*g_i)=\sum_{i=1}^{n}\pi_l(f_i^{\#})^*\pi_l(g_i)\in \mathfrak{n}_{\varphi_H}^*\mathfrak{n}_{\varphi_H}=\mathfrak{m}_{\varphi_H}$$ because $\pi_l(f) \in \mathfrak{n}_{\varphi_H}$ for all left bounded functions $f \in L^2(G)$. It follows that $\pi_l(A)$ is a *-subalgebra of $\mathfrak{m}_{\varphi_H}$. We show that for each $x \in \mathfrak{m}_{\varphi_H}$, there exists a net $(f_i)_{i\in I} \subseteq A$ such that $\pi_l(f_i) \to x$ strongly. Since $L(H)=\pi_l(C_c(H))''$ and $\text{span}(\pi_l(C_c(H))) \subseteq \pi_l(C_c(H))$, there exists a net $(g_{\lambda})_{\lambda\in \Lambda} \subseteq C_c(H)$ such that $\pi_l(g_\lambda) \to x$ strongly. Let $(h_\gamma)_{\gamma\in \Gamma}\subseteq C_c(H)$ be an approximate identity. Define $g_{\lambda,\gamma}=h_\gamma *g_\lambda$ for each $\lambda \in \Lambda,\gamma\in\Gamma$. Then, for all $\lambda \in \Lambda$, it follows that $(g_{\lambda,\gamma})_{\gamma\in\Gamma} \subseteq A$ is a net such that $\|g_{\lambda,\gamma}-g_{\lambda}\|_1 \to 0$. Set $$I:=\{(\lambda,\gamma,\varepsilon):\lambda\in\Lambda,\gamma\in\Gamma,\varepsilon>0,\|g_{\lambda,\gamma}-g_\lambda\|_1<\varepsilon\}$$
    and define $(\lambda_1,\gamma_1,\varepsilon_1)\le (\lambda_2,\gamma_2,\varepsilon_2)$ if $\lambda_1\le\lambda_2,\gamma_1\le\gamma_2
    $ and $\varepsilon_1\ge \varepsilon_2$. We shall now confirm that $I$ is a directed set. For all $(\lambda_1,\gamma_1,\varepsilon_1),(\lambda_2,\gamma_2,\varepsilon_2)\in I$, there exists $0<\varepsilon_3<\text{min}\{\varepsilon_1,\varepsilon_2\}$ and $\lambda_3\ge\lambda_1,\lambda_2$. Moreover, we can choose $\gamma_3\in\Gamma$ such that $\|g_{\lambda_3,\gamma_3}-g_{\lambda_3}\|_1<\varepsilon_3$ and $\gamma_3 \ge \gamma_1,\gamma_2$. Therefore, $(\lambda_3,\gamma_3,\varepsilon_3)\in I$ and $(\lambda_3,\gamma_3,\varepsilon_3) \ge (\lambda_1,\gamma_1,\varepsilon_1),(\lambda_2,\gamma_2,\varepsilon_2)$ holds. Define $p_\Lambda\colon I\to\Lambda;(\lambda,\gamma,\varepsilon)\mapsto\lambda$ and $p_\Gamma\colon I\to\Gamma;(\lambda,\gamma,\varepsilon)\mapsto\gamma$, we shall now check that $p\colon I\to \Lambda\times\Gamma;(\lambda,\gamma,\varepsilon)\mapsto(\lambda,\gamma)$ defines a subnet of $\Lambda\times\Gamma$ as a double index net. For each  $(\lambda_0,\gamma_0)\in\Lambda\times\Gamma$, there exists a $\varepsilon_0>0$ such that $\|g_{\lambda_0,\gamma_0}-g_{\lambda_0}\|_1<\varepsilon_0$. Define $i_0=(\lambda_0,\gamma_0,\varepsilon_0)\in I$. For each $i\ge i_0$, we obtain $p(i)\ge p(i_0)=(\lambda_0,\gamma_0)$. This shows $(g_{p(i)})_{i\in I}$ is a subnet of $(g_{\lambda,\gamma})_{(\lambda,\gamma)\in \Lambda\times\Gamma}\subseteq A$. Similarly, it follows that $p_\Lambda$ defines a subnet of $\Lambda$. Now, it can be shown that $\|g_{p(i)}-g_{p_\Lambda(i)}\|_1 \to 0$. For each $\varepsilon>0$, there exists $(\lambda_0.\gamma_0)\in\Lambda\times\Gamma$ such that $\|g_{\lambda_0,\gamma_0}-g_{\lambda_0}\|_1<\varepsilon$. Set $i_0=(\lambda_0,\gamma_0,\varepsilon)\in I$. For all $i\ge i_0$, we obtain $\|g_{p(i)}-g_{p_\Lambda(i)}\|_1<\varepsilon$. Hence, we conclude that $\|g_{p(i)}-g_{p_\Lambda(i)}\|_1 \to 0$. Finally, we will show $(\pi_l(f_i))_{i\in I}$ converges $x$ strongly where $f_i:=g_{p(i)}\,\,(i\in I)$. For each $\xi\in L^2(H)$,
    \begin{align*}
        \|x\xi-\pi_l(f_i)\xi\|_2 & \le \|x\xi- \pi_l(g_{p_\Lambda(i)})\xi\|_2 + \|\pi_l(g_{p_\Lambda(i)})\xi-\pi_l(f_i)\xi\|_2 \\ &= \|x\xi- \pi_l(g_{p_\Lambda(i)})\xi\|_2 + \|g_{p_\Lambda(i)}*\xi-f_i*\xi\|_2 \\ & \le \|x\xi- \pi_l(g_{p_\Lambda(i)})\xi\|_2 + \|g_{p_\Lambda(i)}-f_i\|_1\|\xi\|_2 \to 0 .
    \end{align*}
    Therefore, $(f_i)_{i\in I}$ is desired one. We conclude that $\pi_l(A)$ is a $\sigma$-weakly dense *-subalgebra of $\mathfrak{m}_{C^{-1}\cdot\varphi_H}$. \par
    Next, we show $\sigma_t^{C^{-1}\cdot\varphi_H}(\pi_l(A))=\pi_l(A)\,\,( t \in \R)$. It suffices to show that $\sigma_t^{\varphi_H}(\pi_l(A)) \subseteq \pi_l(A)\,\,( t \in \R)$ since $\sigma_t^{C^{-1}\cdot\varphi_H}=\sigma_t^{\varphi_H}\,\,( t\in\R)$. We use the form $\pi_l(f)=\int f(x)\lambda_x\, dx$ for $f \in C_c(H)$. Then, for $t\in \R$,  $$\sigma_t^{\varphi_H}(\pi_l(f))=\sigma_t^{\varphi_H}(\int f(x)\lambda_x\, dx) = \int f(x)\sigma_t^{\varphi_H}(\lambda_x) \, dx = \int f(x)\Delta_H(x)^{it}\lambda_x\, dx =\pi_l(f^{(t)})$$ where $f^{(t)}(x)=f(x)\Delta_H(x)^{it} \in C_c(H)$. 
    Moreover, \begin{align*}
         (f*g)^{(t)}(x) =\Delta_H(x)^{it}(f*g)(x)&=\Delta_H(x)^{it}\int f(y)g(y^{-1}x) \, dy \\ &=\int \Delta_H(y)^{it}f(y)\Delta_H(y)^{-it}\Delta_H(x)^{it}g(y^{-1}x)\, dy  \\ &= \int f^{(t)}(y)g^{(t)}(y^{-1}x)\, dy = f^{(t)}*g^{(t)}(x). 
    \end{align*}
    Thus, for $a=\sum_{i=1}^{n}f_i*g_i \in A$, $$\sigma_t^{\varphi_H}(\pi_l(a))=\sum_{i=1}^{n}\sigma_t^{\varphi_H}(\pi_l(f_i*g_i))=\sum_{i=1}^{n}\pi_l(f_i^{(t)}*g_i^{(t)})=\pi_l(\sum_{i=1}^{n}f_i^{(t)}*g_i^{(t)}) \in \pi_l(A).$$
    Therefore, we obtain $\sigma_t^{C^{-1}\cdot\varphi_H}(\pi_l(A)) \subseteq \pi_l(A)\,\,( t \in \R)$. \par
    Finally, we show that $\psi(x)=C^{-1}\cdot\varphi_H(x)$ for all $x \in \pi_l(A)$. It is known that $A \subseteq C_c(H) \cap P(H)$ \cite[Proposition 3.33]{Folland}. Fix $f,g \in C_c(H)$. Since $*$-operation is an involution of $C_c(H)$, we have $(f*g)^*=g^**f^*\in A \subseteq C_c(H) \cap P(H)$. Then, Lemma \ref{rbddlesskdelta} and Lemma \ref{convresol} implies $(f*g)^*=h_0*h_0^{\flat}$ for some right bounded function $h \in L^2(H)$. Hence, $$f*g=((f*g)^*)^*=(h_0*h_0^{\flat})^*=(h_0^{\flat})^**h_0^*=(h_0^*)^{\#}*h_0^*$$
    and $h=h_0^*$ is left bounded by Lemma \ref{rbddimp*lbdd}. We obtain $$\varphi_H(\pi_l(f*g))=\varphi_H(\pi_l(h^{\#}*h))=\varphi_H(\pi_l(h)^*\pi_l(h))=\|h\|_2^2=(h^{\#}*h)(e)=(f*g)(e).$$ Therefore, $\varphi_H(\pi_l(a))=a(e)\,\,( a \in A)$ holds. To complete the proof, we have to show $\psi(\pi_l(a))=C^{-1}\cdot a(e) \,\,( a\in A)$. Note that we identify $L(H)$ with $\{\lambda_h:h\in H\}'' \subseteq L(G)$ via Proposition \ref{subisom}. Let $\Phi$ be the isomorphism i.e. $\Phi(x)=U^*\widehat{x}U$ for $x \in L(H)$. We know that $\Phi(\pi_l(f))=\pi_l(C^{-1}\cdot \tilde{f})$ for $f \in C_c(H)$ as in Proposition \ref{sfofrest}. Fix $f,g\in C_c(H)$, $$\widetilde{f*g}(h)=\int_Hf(y)g(y^{-1}h)\, d\mu_H(y)=C^{-1}\int_H \tilde{f}(y)\tilde{g}(y^{-1}h)\,d\mu_G(y)=C^{-1}(\tilde{f}*\tilde{g})(h) \,\,(h\in H).$$
    Thus, we obtain $\Phi(\pi_l(f*g))=\pi_l(C^{-1}\cdot\widetilde{f*g})=\pi_l(C^{-2}\cdot\tilde{f}*\tilde{g})$. Therefore, it follows that 
    \begin{align*}        \psi(\Phi(\pi_l(f*g)))&=\psi(\pi_l(C^{-2}\cdot\tilde{f}*\tilde{g}))=\varphi_G(\pi_l(C^{-2}\cdot \tilde{f}*\tilde{g}))\\&=C^{-2}(\tilde{f}*\tilde{g})(e)=C^{-1}(\widetilde{f*g})(e)=C^{-1}(f*g)(e).
    \end{align*}
    We use the formula $\varphi_G(\pi_l(f*g))=(f*g)(e)$ for all $f,g\in C_c(G)$.
    Hence, we conclude that $\psi(\pi_l(a))=C^{-1}\cdot a(e) \,\,( a\in A)$. 
\end{proof}

\begin{remark}
    The constant $C>0$ in Theorem \ref{restofplweight} depends on the choice of $\mu_G$ and $\mu_H$. Thus, if we fix a left Haar measure $\mu_G$ on $G$, then there exists the left Haar measure $\mu_H$ on $H$ which satisfies $C=1$. 
\end{remark}

Now, we obtain the main theorem. 

\begin{theorem}\label{plweightmainthm}
    For a locally compact group $G$ and a closed subgroup $H$, the following conditions are equivalent.
    \begin{enumerate}
        \item[(i)] $H$ is open.
        \item[(ii)] $\varphi_G|L(H)$ is semifinte.
        \item[(iii)] $\varphi_G|L(H)=C\cdot\varphi_H$ for some constant $C>0$.
     \end{enumerate}
\end{theorem}
\begin{proof}
   $\text{(i)}\iff \text{(ii)}.$ Theorem \ref{openiffsf}. 
   $\text{(i)}\implies \text{(iii)}.$ Theorem \ref{restofplweight}.
   $\text{(iii)}\implies \text{(ii)}.$ Trivial.
\end{proof}

\begin{remark}
     As mentioned earlier, $\text{(ii)}\iff \text{(iii)}$ follows from the uniqueness of Haar weight on Kac algebras or quantum groups (see, e.g.\cite[Theorem 2.7.7]{EnockSchwartz1992} or \cite{JohanVaes2000}), and we learned about this result from Professor Reiji Tomatsu.
\end{remark}

\section{Twisted crossed products}
Throughout $\S3,\S4$ and $\S5$, von Neumann algebras are assumed to have separable preduals, locally compact groups are assumed to be second countable. In this section, we study twisted crossed products of von Neumann algebras. We recall the definition. 

\begin{definition}
    For locally compact groups $G,H$, a cocycle action $(\alpha,c)\colon H \curvearrowright G$ is a pair of maps $\alpha\colon H \to \text{Aut}(G)$ and $c\colon H \times H \to G$ satisfying the relations
    $$\alpha_s\alpha_t=\text{Ad}(c(s,t))\alpha_{st},\,\, c(s,t)c(st,r)=\alpha_s (c(t,r))c(s,tr),\,\,c(s,e_H)=c(e_H,s)=e_G $$
    where $s,t,r \in H$. We say the cocycle action is continuous if the maps $H \times G \ni (s,x)\mapsto \alpha_s(x) \in G$ and $H \times H \ni (s,t) \mapsto c(s,t) \in G$ are continuous where the products are equipped with the product topology.
\end{definition}

\begin{definition}\label{twistcrossprodgrp}
    For locally compact groups $G,H$ and a continuous cocycle action $(\alpha,c)\colon H \curvearrowright G$, we define twisted crossed product as a topological group, denoted by $G \rtimes_{\alpha,c}H$, consisting of $G\times H$ equipped with the product topology and the following group operations
    $$(x,s)(y,t)=(x\alpha_s(y)c(s,t),st),\,\,\,\,(x,s)^{-1}=(\alpha_s^{-1}(x^{-1}c(s,s^{-1})^{-1}),s^{-1})$$
    where $x,y \in G, s,t \in H$. Also, the unit is $(e_G,e_H)$.
\end{definition}

\begin{definition}\label{defofcocycle}
    For a von Neumann algebra $M$ and a locally compact group $G$, a cocycle action $(\alpha,u)\colon G \curvearrowright M$ is a pair of maps $\alpha\colon G \to \text{Aut}(M)$ and $u\colon G \times G \to M^{U}$ satisfying the relations $$\alpha_s\alpha_t=\text{Ad}(u(s,t))\alpha_{st},\,\,u(s,t)u(st,r)=\alpha_s(u(t,r))u(s,tr),\,\,u(s,e_G)=u(e_G,s)=1_M$$ where $s,t,r\in G$. Additionally, we demand that $\alpha$ and $u$ are Borel. That is, $G \ni g \mapsto \alpha_g\in \text{Aut}(M)$ and $u$ are Borel measurable, where a Borel structure on $\text{Aut}(M)$ is generated by the u-topology (equivalent to the topology of pointwise norm convergence against $M_*$) and on $M^U$ is generated by the strong $*$-topology on $M^{U}$.
\end{definition}

\begin{definition}\label{defoftwistvN}
    For a von Neumann algebra $M$ acting on a Hilbert space $\mathcal{H}$, a locally compact group $G$ and a cocycle action $(\alpha,u)\colon G \curvearrowright M$, we define the left twisted crossed product of $M$ by $G$ is the von Neumann algebra on $L^2(G,\mathcal{H})$ (with fixed left Haar measure) generated by the operators 
    \begin{equation*}
    \begin{cases}
        (I^{\alpha}_\ell(x)\xi)(g)=\alpha_{g^{-1}}(x)\xi(g), \\
        (\lambda_\ell^{u}(h)\xi)(g)=u(g^{-1},h)\xi(h^{-1}g)
    \end{cases}
    \end{equation*}
    where $x \in M, g,h\in G, \xi \in L^2(G, \mathcal{H})$. We denote $M \rtimes_{\alpha,u}^{\ell}G$. Also, we define the right twisted crossed product of $M$ by $G$ is the von Neumann algebra on $L^2(G,\mathcal{H})$ (with fixed right Haar measure) generated by the operators  
    \begin{equation*}
    \begin{cases}
        (I_r^{\alpha}(x)\xi)(g)=\alpha_g(x)\xi(g), \\
        (\lambda_r^{u}(h)\xi)(g)=u(g,h)\xi(gh)
    \end{cases}
    \end{equation*}
    where $x \in M, g,h\in G, \xi \in L^2(G, \mathcal{H})$. We denote $M \rtimes_{\alpha,u}^{r}G$.
\end{definition}

\begin{remark}
    Note that the right twisted crossed product is used in \cite{NakagamiSutherland1979}. Moreover, let $M$ be a von Neumann algebra acting on a Hilbert space $\mathcal{H}$. Let $G$ be a locally compact group and $(\alpha,u)\colon G \curvearrowright M$ be a cocycle action as described above. Fix a left Haar measure $\mu_G$ on $G$. Then, the right Haar measure on $G$ associated with $\mu_G$ is defined by $\rho_G(E)=\mu_G(E^{-1})$ for each Borel set $E \subseteq G$. In this case, we define the unitary $U\colon L^2(G, \mathcal{H},\mu_G) \to L^2(G, \mathcal{H},\rho_G)$ given by 
    $$(U\xi)(g)=\Delta_G(g)^{-1/2}\xi(g^{-1})\,\,(\xi\in L^2(G, \mathcal{H},\mu_G),g\in G).$$
    It can be shown that this unitary gives the spatial isomorphism $M\rtimes_{\alpha,u}^\ell G \cong M \rtimes_{\alpha,u}^r G$.
\end{remark}

\begin{remark}
    With the same notation as above, it follows that $I_\ell^\alpha$ and $\lambda_\ell^u$ satisfy the covariant relation. That is, for $x\in M, s\in G$, 
    \begin{align*}(\lambda_\ell^u(s)I_\ell^\alpha(x)\lambda_\ell^u(s)^*\xi)(t) &=u(t^{-1},s)(I_\ell^\alpha(x)\lambda_\ell^u(s)^*\xi)(s^{-1}t)
    \\&= u(t^{-1},s)\alpha_{t^{-1}s}(x)(\lambda_\ell^u(s)^*\xi)(s^{-1}t)
    \\&= u(t^{-1},s)\alpha_{t^{-1}s}(x)u(t^{-1},s)^*\xi(t)
    \\&= \text{Ad}(u(t^{-1},s))\alpha_{t^{-1}s}(x)\xi(t)
    \\&=\alpha_t^{-1}(\alpha_s(x))\xi(t)
    \\&=(I_\ell^{\alpha}(\alpha_s(x))\xi)(t)\,\,(\xi\in L^2(G,\mathcal{H})).
    \end{align*}
    Therefore, we obtain    $$\lambda_\ell^u(s)I_\ell^\alpha(x)\lambda_\ell^u(s)^*=I_\ell^{\alpha}(\alpha_s(x))$$
    for all $x\in M, s \in G$.
\end{remark}

Now, we study actions and dual actions. For a moment, we assume that $G$ has a right Haar measure. Throughout the rest of this paper, an action of $G$ on $M$ means a pointwise $\sigma$-weakly continuous action unless explicitly stated otherwise.\par
First, we note that we can identify an action of $G$ on $M$ with an isomorphism of $M$ into $M \overline{\otimes} L^\infty(G)$. We refer to \cite{Nakagami1977} for details.
That is, fix an action $\alpha$ of a locally compact group $G$ on a von Neumann algebra $M$, $\pi_\alpha\colon M\to M \overline{\otimes}L^\infty(G)$ is defined by $$\pi_\alpha(x)(g)=\alpha_g(x)\,\,(g\in G)$$ where we identify $M \overline{\otimes} L^\infty(G)$ with $L^\infty(G,M)$ of all essentially bounded $M$-valued $\sigma$-weakly measurable functions on G. In \cite[Theorem 2.1]{Nakagami1977}, it was shown that a mapping $\alpha$ of $M$ into $M \overline{\otimes}L^{\infty}(G)$ be induced an action $\sigma$ with $(\alpha(x)\xi)(g)=\sigma_g(x)\xi(g)$ if and only if $\alpha$ is an injective homomorphism which satisfies $$(\alpha\otimes \iota)\circ\alpha=(\iota\otimes\delta)\circ\alpha$$ where 
$$\delta\colon L^{\infty}(G)\to L^{\infty}(G)\overline{\otimes}L^{\infty}(G);(\delta f)(s,t)=f(st)$$ and $\iota$ means identity map on $L^{\infty}(G)$ (Here, we identify $L^{\infty}(G)\overline{\otimes}L^{\infty}(G)$ with $L^{\infty}(G\times G)$). In this view point, we define the dual action of $G$ as an isomorphism of $M$ into $M\vNtensor R(G)$ where $R(G)$ denotes the von Neumann algebra generated by the right regular representation on $L^2(G, \rho_G)$, $\rho_G$ stands for the right Haar measure on $G$.
 
\begin{definition}[{\cite[Definition 2.2]{Nakagami1977}}]
    For a von Neumann algebra $N$ acting on a Hilbert space $\mathcal{K}$ and a locally compact group $G$, a dual action of $G$ on $N$ is an isomorphism of $N$ into $N \vNtensor R(G)$ satisfying $$(\beta\otimes\iota)\circ\beta=(\iota\otimes\gamma)\circ\beta$$
    where $\gamma$ is a normal unital injective homomorphism characterized by $$\gamma\colon R(G)\to R(G) \vNtensor R(G); \gamma(\rho_g)=\rho_g \otimes\rho_g.$$
    We also define crossed dual product $N\rtimes_\beta^{d}G$ be the von Neumann algebra generated by the operators $\beta(N)$ and $1\otimes L^{\infty}(G)$. 
\end{definition}

\begin{remark}
    In the above definition, we note that dual action is not an action in usual sense. 
\end{remark}

Now, we analyze the case where $G$ is abelian. Thus, in the rest of this section, $G$ is assumed to be abelian and admits a fixed (left and right) invariant Haar measure. First, we define the action of $\hat{G}$. Fix a von Neumann algebra $M$ acting on a Hilbert space $\mathcal{H}$, a locally compact abelian group $G$ and a cocycle action $(\alpha,u)\colon G\curvearrowright M$. Following \cite[II, P150]{Sutherland1980}, we define for $p\in \hat{G}$ a unitary $\mu(p)$ on $L^2(G, \mathcal{H})$ by $$(\mu(p)\xi)(g)=\overline{\nai{g}{p}}\xi(g)\,\,\,(\xi\in L^2(G, \mathcal{H}),g\in G).$$
Here, for $x\in M$, on $M \rtimes_{\alpha,u}^{\ell}G$
\begin{align*}
    (\mu(p)I_\ell^{\alpha}(x)\mu(p)^{*}\xi)(g) &= \overline{\nai{g}{p}}(I_\ell^{\alpha}(x)\mu(p)^*\xi)(g) \\&= \overline{\nai{g}{p}}\alpha_{g^{-1}}(x)(\mu(p)^*\xi)(g) \\&= \overline{\nai{g}{p}}\alpha_{g^{-1}}(x)\nai{g}{p}\xi(g) \\&=\alpha_{g^{-1}}(x)\xi(g)\\&= (I_\ell^{\alpha}(x)\xi)(g).
\end{align*}
Also, for $h\in G$, 
\begin{align*}
    (\mu(p)\lambda_\ell^{u}(h)\mu(p)^{*}\xi)(g) &= \overline{\nai{g}{p}}(\lambda_\ell^{u}(h)\mu(p)^*\xi)(g) \\&= \overline{\nai{g}{p}}u(g^{-1},h)(\mu(p)^*\xi)(h^{-1}g) \\&= \overline{\nai{g}{p}}u(g^{-1},h)\nai{h^{-1}g}{p}\xi(h^{-1}g) \\&=\overline{\nai{h}{p}}u(g^{-1},h)\xi(h^{-1}g)\\&= (\overline{\nai{h}{p}}\lambda^u_\ell(h)\xi)(g).
\end{align*}
Thus, $\hat{\alpha}_p=\text{Ad}\mu(p)$ defines a continuous action of $\hat{G}$ on $M \rtimes_{\alpha,u}^{\ell}G$ satisfying 
\begin{align*}
    \begin{cases}        \hat{\alpha}_p(I_\ell^\alpha(x))=I_\ell^\alpha(x),\quad (x\in M) \\
    \hat{\alpha}_p(\lambda_\ell^u(h))=\overline{\nai{h}{p}}\lambda_\ell^u(h) \quad (h\in G)
    \end{cases}
\end{align*}
for all $p \in \hat{G}$.

\begin{definition}\label{defdualaction}
    We say that the continuous action of $\hat{G}$ defined as above is the action of $\hat{G}$ dual to $(\alpha,u)$. We denote $\hat{\alpha}$. 
\end{definition}

%We also define the unitaries $U,V,V'$ and $W$ on $L^2(G) \otimes L^2(G)$ by 
%$$(U\xi)(s,t)=\xi(t,s),(V\xi)(s,t)=\xi(st,t),(V'\xi)(s,t)=\xi(t^{-1}s,t),$$ and $W=UVU$, so $(W\xi)(s,t)=\xi(s,ts)$. \par 
Now, we show the duality for the action dual to a cocycle action.

\begin{lemma}[{\cite[Lemma 1]{NakagamiSutherland1979}}]\label{dualaction}
    For a von Neumann algebra $M$ acting on a Hilbert space $\mathcal{H}$, a locally compact abelian group $G$ and a cocycle action $(\alpha,u)\colon G\curvearrowright M$, if $\beta$ is defined on $M\rtimes_{\alpha,u}^{r}G$ by $$\beta(y)=\text{Ad}(1\otimes W^*)(y\otimes 1_G) \in \B(\mathcal{H})\vNtensor\B(L^2(G))\vNtensor\B(L^2(G)),$$ then $\beta$ is a dual action of $G$ on $M\rtimes_{\alpha,u}^{r}G$. Note that $y\in M\rtimes_{\alpha,u}^{r}G \subseteq \B(L^2(G,\mathcal{H})) \cong \B(\mathcal{H}\otimes L^2(G))=\B(\mathcal{H})\vNtensor\B(L^2(G))$. Here, the unitary $W$ on $L^2(G) \otimes L^2(G)$ is defined by $$(W\xi)(s,t)=\xi(s,ts)\,\,(\xi\in L^2(G) \otimes L^2(G)).$$ 
\end{lemma}

We provide the following lemmas. The first lemma is a minor variation of Tietze's extension theorem (e.g., \cite[4.34]{FollandrealanalysisMR1681462}). We include the proof for completeness.

\begin{lemma}\label{TietzeforLCH}
    For a locally compact Hausdorff space $X$, compact subset $K \subseteq X$ and open subset $K \subseteq U \subseteq X$, if $f\in C(K)$, then there exists $F \in C_c(X)$ such that $F|K=f$ and ${\rm supp}\,F \subseteq U$ hold. Moreover, we may arrange $F$ to satisfy $\|F\|_\infty \le 2$ if $\|f\|_\infty \le 1$.  
\end{lemma}
\begin{proof}
    First, we consider the case $f\in C(K,[0,1])$. By the regularity of $X$, there exists a relatively compact open subset $V \subseteq X$ such that $K \subseteq V \subseteq \overline{V} \subseteq U$. We obtain $\tilde{f} \in C(\overline{V},[0,1])$ such that $\tilde{f}|K=f$ by Tietze extension theorem. If $V=\overline{V}$, then the zero extension of $\tilde{f}$, we denote $F_0$, satisfies the desired properties. If $V \subsetneq \overline{V}$, then there exists $g\in C(X,[0,1])$ such that $g|K=1$ and $g|\overline{V}\backslash V =0$ by Urysohn's lemma. In this case, $F_0\cdot g$ satisfies the desired properties. Next, we consider the case where $f$ is real valued. Since $K$ is compact, there exist $a,b\in \R$ such that $f\in C(K,[a,b])$. Then, using the homeomorphism $[0,1]\ni t \mapsto (b-a)t+a \in[a,b]$, we may find $F\in C_c(X,[a,b])$ with the desired properties. Finally, if $f$ is complex valued, then we obtain $F_1,F_2\in C_c(X,\R)$ that satisfies $F_1|K=\text{Re}f,F_2|K=\text{Im}f$ and ${\rm supp}\,F_i \subseteq U\,\,(i=1,2)$. In this case, $F=F_1+iF_2$ satisfies the desired properties. Moreover, if $\|f\|_\infty \le 1$, then $\|\text{Re}f\|_\infty,\|\text{Im}f\|_\infty \le 1$ hold. The previous argument implies that $\|F_1\|_\infty,\|F_2\|_\infty \le 1$. Therefore, $\|F\|_\infty \le \|F_1\|_\infty +  \|F_2\|_\infty\le 2$ holds.
\end{proof}

\begin{lemma}\label{dualgrpdensity}
    Let $G$ be a locally compact abelian group. We have $m(\hat{G})''=m(L^\infty(G)) \subseteq \B(L^2(G))$, where $m(f)$ denotes the multiplication operator.
\end{lemma}
\begin{proof}
    First, we show that $m(C_0(G))''=m(L^\infty(G))$. Since $m(L^\infty(G))$ is a von Neumann algebra, it is clear that $m(C_0(G))'' \subseteq m(L^\infty(G))$. To show the converse inclusion, we shall show that $m(f)$ is in SOT-closure of the set $\{m(g):g\in C_c(G)\}$. Take $f \in L^\infty(G)$. Fix an arbitrary SOT-neighborhood of $m(f)$, that is, fix $\varepsilon>0,\xi_1,\ldots,\xi_n\in L^2(G)$ and consider the set $$\mathcal{N}(m(f);\xi_1,\ldots ,\xi_n,\varepsilon)=\{T\in \B(L^2(G)): \|m(f)\xi_i-T\xi_i\|_2<\varepsilon\,\,(i=1,\ldots,n)\}.$$ It suffices to show that there exists $g\in C_c(G)$ such that $m(g) \in \mathcal{N}(m(f);\xi_1,\ldots ,\xi_n,\varepsilon)$. We may assume that $\|f\|_\infty=1$. Since $\mu_G$ is a Radon measure, $d\nu_i=|\xi_i|^2d\mu_G\,\,(i=1,\ldots,n)$ is a finite Radon measure on $G$ (see, e.g. \cite[P220, Exercise 7.8]{FollandrealanalysisMR1681462}). Then, by the inner regularity, for each $i=1,\ldots,n$, there exists a compact set $K_i$ such that $\nu_i(G\backslash K_i) < \varepsilon^2/50$. Let $K=\bigcup_{i=1}^nK_i$. We see that $K$ is compact. Since $\mu_G|K$ is a finite Radon measure on $K$, $C_c(K)$ is dense in $L^1(K,\mu_G|K)$. Therefore, since $f|K \in L^1(K,\mu_G|K)$, there exists a sequence  $(g_m)_{m\in \N} \subseteq C_c(K)$ such that $\int_K|g_m-f|\,d\mu_G \to 0$ as $m\to \infty$ and $|g_m(x)|\le2|f(x)|$ for all $m\in\N$ and $\mu_G$-a.e. $x\in K$. Then, by passing to a subsequence, we may assume $|g_m(x)|\to|f(x)|$ as $m\to\infty$ $\mu_G$-a.e. $x\in K$. For each $i=1,\ldots,n$, we have $|g_m(x)-f(x)|^2 |\xi_i(x)|^2\le 9|\xi_i(x)|^2 $ and $|\xi_i|^2\in L^1(K,\mu_G|K)$. Therefore, by Lebesgue dominated convergence theorem, we have 
    $$\int_K|g_m-f|^2|\xi_i|^2\,d\mu_G \to 0\,\, \text{as}\,\, m \to\infty\,\, (i=1,\ldots,n). $$ It follows that there exists $N\in \N$ such that $\int_K|g_N-f|^2|\xi_i|^2\,d\mu_G < \varepsilon^2/2$ for all $i=1,\ldots,n$. By Lemma \ref{TietzeforLCH}, we obtain $g\in C_c(G)$ such that $g|K=g_N$. Since $\|g_N\|_\infty \le2$, the proof of Lemma \ref{TietzeforLCH} implies that $\|g\|_\infty \le4$. For $i=1,\ldots,n$, we have 
    $$\int_{G\backslash K}|g-f|^2|\xi_i|^2\,d\mu_G \le 25 \int_{G \backslash K}|\xi_i|^2\,d\mu_G < \varepsilon^2/2.$$
    Thus, we obtain 
    \begin{align*}
        \|m(g)\xi_i-m(f)\xi_i\|_2^2&=\int_G|g-f|^2|\xi_i|^2 \,d\mu_G \\&=\int_K|g-f|^2|\xi_i|^2\,d\mu_G+\int_{G\backslash K}|g-f|^2|\xi_i|^2\,d\mu_G
        \\&< \frac{\varepsilon^2}{2} + \frac{\varepsilon^2}{2} =\varepsilon^2\,\,(i=1,\ldots,n).
    \end{align*}
    We conclude that $m(g) \in \mathcal{N}(m(f);\xi_1,\ldots ,\xi_n,\varepsilon)$ and hence $m(f) \in m(C_c(G))''\subseteq m(C_0(G))''$. \par
    Next, we shall show that $m(\hat{G})''=m(L^\infty(G))$, in particular $m(L^\infty(G))\subseteq m(\hat{G})''$. We identify $G$ with the dual group of $\hat{G}$, let $\hat{\mathcal{F}}\colon L^1(\hat{G})\to C_0(G)$ be the Fourier transform given by $$(\hat{\mathcal{F}}f)(x)=\int_{\hat{G}} \nai{x}{p}f(p)\,d\mu_{\hat{G}}(p)\,\,(f\in L^1(\hat{G})).$$ Then, it follows that $\hat{\mathcal{F}}(L^1(\hat{G}))$ is norm dense in $C_0(G)$ and hence $m(C_0(G))'' = m(\hat{\mathcal{F}}(L^1(\hat{G})))''$. We see that $$m(\hat{\mathcal{F}}f)=\int_{\hat{G}}f(p)m(p)\,d\mu_{\hat{G}}(p)\in m(\hat{G})''\,\,(f\in L^1(\hat{G})).$$ Thus, we conclude that $$m(L^\infty(G))=m(C_0(G))'' = m(\hat{\mathcal{F}}(L^1(\hat{G})))''\subseteq m(\hat{G})''.$$
\end{proof}

Now, we provide the duality theorem for the twisted crossed product. In fact, this was shown in \cite{VaesVainerman2003} in the context of quantum groups, and we learned about this result from Professor Reiji Tomatsu. However, we give the von Neumann algebraic proof for reader's convenience.

\begin{proposition}\label{cocycleduality}
    For a von Neumann algebra $M$ acting on a Hilbert space $\mathcal{H}$, a locally compact abelian group $G$ and a cocycle action $(\alpha,u)\colon G\curvearrowright M$, we have $$M\rtimes_{\alpha,u}^{\ell}G\rtimes_{\hat{\alpha}}\hat{G}\cong M \vNtensor \B(L^2(G)).$$
\end{proposition}
\begin{proof}
    Let $\beta$ be a dual action of $G$ defined in Lemma \ref{dualaction}. Since $M\rtimes_{\alpha,u}^{r}G\rtimes_\beta^{d}G\cong M \vNtensor\B(L^2(G))$ by \cite[Theorem 2]{NakagamiSutherland1979}, we shall show $M\rtimes_{\alpha,u}^{\ell}G\rtimes_{\hat{\alpha}}\hat{G}\cong M\rtimes_{\alpha,u}^{r}G\rtimes_\beta^{d}G$. Let $\mathcal{F}\colon L^2(G)\to L^2(\hat{G)}$ be the Fourier transform defined by $$(\mathcal{F}\xi)(p)=\int_{G}\overline{\nai{s}{p}}\xi(s)\,d\mu_G(s)\,\,(\xi\in L^2(G)).$$ Put $\hat{\mathcal{F}}=\mathcal{F}^*$ be the Fourier inverse transform. We also define unitary $S\colon L^2(G)\to L^2(G)$ by $(S\xi)(g)=\xi(g^{-1})$. Set $\Phi\colon \B(\mathcal{H})\vNtensor\B(L^2(G))\vNtensor\B(L^2(\hat{G}))\to \B(\mathcal{H})\vNtensor\B(L^2(G))\vNtensor\B(L^2(G))$ by $\Phi=\text{Ad}(1\otimes S \otimes \hat{\mathcal{F}})$. Note that $M\rtimes_{\alpha,u}^{\ell}G\rtimes_{\hat{\alpha}}\hat{G}$ is generated by operators $\pi_{\hat{\alpha}}(x), \lambda_{\hat{\alpha}}(p) (x\in M\rtimes_{\alpha,u}^{\ell}G,p\in \hat{G})$ defined by $$(\pi_{\hat{\alpha}}(x)\xi)(g,p)=((\hat{\alpha}_p^{-1}(x)\otimes1_{\hat{G}})\xi)(g,p),(\lambda_{\hat{\alpha}}(q)\xi)(g,p)=\xi(g,p-q)$$ where $\xi\in\mathcal{H}\otimes L^2(G) \otimes L^2(\hat{G}),g\in G, p,q \in \hat{G}$. Then, we can calculate, for $x\in M$, 
    \begin{align*}
        (\Phi(\pi_{\hat{\alpha}}(I_\ell^{\alpha}(x)))\xi)(s,t)&=((1\otimes S \otimes \hat{\mathcal{F}})(\pi_{\hat{\alpha}}(I_\ell^{\alpha}(x)))(1\otimes S^*\otimes \mathcal{F})\xi)(s,t) \\&=\int_{\hat{G}}\nai{t}{p}(\pi_{\hat{\alpha}}(I_\ell^{\alpha}(x))(1\otimes S^*\otimes \mathcal{F})\xi)(s^{-1},p) \,d\mu_{\hat{G}}(p) \\&= \int_{\hat{G}}\nai{t}{p}((\hat{\alpha}_p^{-1}(I_\ell^{\alpha}(x))\otimes1_{\hat{G}})(1\otimes S^* \otimes\mathcal{F})\xi)(s^{-1},p)\, d\mu_{\hat{G}}(p) \\&= \int_{\hat{G}}\nai{t}{p}((I_\ell^{\alpha}(x)\otimes 1_{\hat{G}})(1\otimes S^* \otimes\mathcal{F})\xi)(s^{-1},p)\, d\mu_{\hat{G}}(p) \\&= \int_{\hat{G}}\nai{t}{p}\alpha_{s}(x)((1\otimes S^* \otimes\mathcal{F})\xi)(s^{-1},p)\, d\mu_{\hat{G}}(p) \\&= \alpha_{s}(x)\int_{\hat{G}}\nai{t}{p}((1\otimes S^* \otimes\mathcal{F})\xi)(s^{-1},p)\, d\mu_{\hat{G}}(p) \\&= \alpha_{s}(x)((1\otimes S\otimes \hat{\mathcal{F}})(1\otimes S^* \otimes\mathcal{F})\xi)(s,t) \\&= \alpha_s(x)\xi(s,t)\\&= ((I_r^{\alpha}(x)\otimes 1_{G})\xi)(s,t).  
    \end{align*}
    For $g \in G$, 
    \begin{align*}
         (\Phi(\pi_{\hat{\alpha}}(\lambda_\ell^u(g)))\xi)(s,t) 
         &= \int_{\hat{G}} \nai{t}{p}(\pi_{\hat{\alpha}}(\lambda_\ell^u(g))(1\otimes S^*\otimes \mathcal{F})\xi)(s^{-1},p)\,d\mu_{\hat{G}}(p) 
         \\&= \int_{\hat{G}} \nai{t}{p}((\hat{\alpha}_p^{-1}(\lambda_\ell^u(g))\otimes1_{\hat{G}})(1\otimes S^*\otimes \mathcal{F})\xi)(s^{-1},p)\,d\mu_{\hat{G}}(p) 
         \\&= \int_{\hat{G}}\nai{t}{p}\nai{g}{p}u(s,g)((1\otimes S^* \otimes \mathcal{F})\xi)(g^{-1}s^{-1},p)\,d\mu_{\hat{G}}(p) 
         \\&= u(s,g) \int_{\hat{G}}\nai{tg}{p}((1\otimes S^* \otimes \mathcal{F})\xi)(g^{-1}s^{-1},p)\,d\mu_{\hat{G}}(p) 
         \\&= u(s,g)((1\otimes S \otimes \hat{\mathcal{F}})(1\otimes S^* \otimes \mathcal{F})\xi)(sg,tg)
         \\&=u(s,g)\xi(sg,tg)
         \\&=((\lambda_r^u(g)\otimes \lambda_r(g))\xi)(s,t),
    \end{align*}
    where $\lambda_r$ stands for the right regular representation. Furthermore, for $q\in \hat{G}$, 
    \begin{align*}
        (\Phi(\lambda_{\hat{\alpha}}(q))\xi)(s,t) 
        &= \int_{\hat{G}} \nai{t}{p}(\lambda_{\hat{\alpha}}(q)(1\otimes S^*\otimes \mathcal{F})\xi)(s^{-1},p)\,d\mu_{\hat{G}}(p)
        \\&= \int_{\hat{G}} \nai{t}{p}((1\otimes S^*\otimes \mathcal{F})\xi)(s^{-1},p-q)\,d\mu_{\hat{G}}(p)
        \\&=\int_{\hat{G}} \nai{t}{p+q}((1\otimes S^*\otimes \mathcal{F})\xi)(s^{-1},p)\,d\mu_{\hat{G}}(p) 
        \\&= \nai{t}{q}\int_{\hat{G}} \nai{t}{p}((1\otimes S^*\otimes \mathcal{F})\xi)(s^{-1},p)\,d\mu_{\hat{G}}(p) 
        \\&= \nai{t}{q} ((1\otimes S\otimes \hat{\mathcal{F}})(1\otimes S^*\otimes \mathcal{F})\xi)(s,t) 
        \\&=\nai{t}{q}\xi(s,t)
        \\&=((1\otimes 1_G \otimes m(q))\xi)(s,t),
    \end{align*}
    where $m\colon L^{\infty}(G)\to \B(L^2(G));(m(f)\xi)(t)=f(t)\xi(t)$. Therefore, we conclude 
    \begin{align*}
        \begin{cases}
            \Phi(\pi_{\hat{\alpha}}(I_\ell^{\alpha}(x)))=I_r^{\alpha}(x)\otimes1_G \\
            \Phi(\pi_{\hat{\alpha}}(\lambda_\ell^{u}(g)))=\lambda_r^u(g)\otimes \lambda_r(g) \\
            \Phi(\lambda_{\hat{\alpha}}(q))=1\otimes1_G\otimes m(q),
        \end{cases}
    \end{align*}
    where $x \in M, g\in G, q\in \hat{G}$. By \cite[Lemma 1]{NakagamiSutherland1979}, we see that 
    \begin{align*}
        \begin{cases}
            \beta(I_r^{\alpha}(x))=I_r^{\alpha}(x) \otimes 1_G\\
            \beta(\lambda_r^{u}(g))=\lambda_r^{u}(g)\otimes\lambda_r(g),
        \end{cases}
    \end{align*}
    where $x\in M, g \in G$. By Lemma \ref{dualgrpdensity}, we conclude that $\Phi(M\rtimes_{\alpha,u}^\ell G \rtimes_{\hat{\alpha}}\hat{G})$ generates $M \rtimes_{\alpha,u}^rG \rtimes_\beta^dG$. Hence, we have $$M\rtimes_{\alpha,u}^\ell G \rtimes_{\hat{\alpha}}\hat{G}\cong M \rtimes_{\alpha,u}^rG \rtimes_\beta^d G \cong M \vNtensor\B(L^2(G)).$$
    In particular, $M\rtimes_{\alpha,u}^\ell G \rtimes_{\hat{\alpha}}\hat{G}$ and $M \rtimes_{\alpha,u}^rG \rtimes_\beta^dG$ are spatially isomorphic.
\end{proof}

\begin{remark}\label{remark1}
    Proposition \ref{cocycleduality} might be known to experts, but for the clarity of the presentation we included a detailed proof. Moreover, $\hat{\alpha}$ of Proposition \ref{cocycleduality} is employed as the point modular extension within \cite[Theorem 5.1]{guinto2025unimodulargroups} (see also Proposition \ref{fullspectrum}).   
\end{remark}

 We introduce the action $\tilde{\alpha}$ of $G$ on $M \vNtensor \B(L^2(G))$. Let $\tilde{\lambda^u}(t)$ be the unitary on $\mathcal{H}\otimes L^2(G)$ defined by $$(\tilde{\lambda^u}(t)\xi)(s)=u(t,t^{-1}s)^*\xi(t^{-1}s)$$ where $\xi\in \mathcal{H}\otimes L^2(G),t,s\in G$. Then, we see that $$(\tilde{\lambda^u}(t)^*\xi)(s)=u(t,s)\xi(ts).$$ Now, we define the action $\tilde{\alpha}$ by $$\tilde{\alpha}_t=\text{Ad}\tilde{\lambda^u}(t)\circ(\alpha_t \otimes \iota)\,\, (t\in G).$$
We can check $\tilde{\alpha}$ is a well defined action, i.e. $\tilde{\alpha}_s\tilde{\alpha}_t=\tilde{\alpha}_{st}$ for all $s,t\in G$. We refer to \cite[Corollary 3]{NakagamiSutherland1979} and \cite[Theorem 2.1]{Nakagami1977} for details.

Now, we introduce the subgroup of $\hat{G}$ defined by the cocycle action. 

\begin{definition}\label{defoftwistspctrm}
    For a von Neumann algebra $M$ acting on a Hilbert space $\mathcal{H}$, a locally compact abelian group $G$ and a cocycle action $(\alpha,u)\colon G\curvearrowright M$, $\Gamma(\alpha,u)$ is defined as the kernel of the restriction of the action $\hat{\alpha}$ of $\hat{G}$ dual to $(\alpha,u)$ to the center $Z(M\rtimes_{\alpha,u}^\ell G)$ of the twisted crossed product $M\rtimes_{\alpha,u}^\ell G$. That is, 
    $$\Gamma(\alpha,u)=\{p\in \hat{G}: \hat{\alpha}_p(n)=n \,\,(n\in Z(M\rtimes_{\alpha,u}^\ell G))\}$$
\end{definition}

We give some definitions for the factoriality results of twisted crossed products.

\begin{definition}
    For a von Neumann algebra $M$ acting on a Hilbert space $\mathcal{H}$, a locally compact abelian group $G$ and a cocycle action $(\alpha,u)\colon G\curvearrowright M$, the fixed point algebra $M^{(\alpha,u)}$ is defined by $$M^{(\alpha,u)}=\{x\in M : \alpha_t(x)=x\,(t\in G)\}.$$ The equation $\alpha_s\alpha_t=\text{Ad}(u(s,t))\alpha_{st}$ implies that $$\text{Ad}(u(s,t))(x)=x\,\,(s,t\in G)$$ for all $x \in M^{(\alpha,u)}$ automatically. Clearly, $M^{(\alpha,u)}$ is a von Neumann subalgebra of $M$.
\end{definition}

Since $\alpha_t\in \text{Aut}(M)$ for all $t\in G$, we have $\alpha_t(Z(M))=Z(M) \,\, (t\in G)$. Thus, we can define central ergodicity of a cocycle action. 

\begin{definition}
    For a von Neumann algebra $M$ acting on a Hilbert space $\mathcal{H}$, a locally compact abelian group $G$ and a cocycle action $(\alpha,u)\colon G\curvearrowright M$, we say the cocycle action is centrally ergodic if $$Z(M)^{\alpha}=\{x\in Z(M):\alpha_t(x)=x\,\,(t\in G)\}=\C1_M.$$
\end{definition}

Before proving the factoriality characterization. We recall elementary facts about an isomorphism between von Neumann algebras. Let $M,N$ be von Neumann algebras and $\Phi\colon M\to N$ be an isomorphism. Then it is straightforward to see $\Phi(Z(M))=Z(N)$. Next, if some locally compact group $G$ acts on $M$ continuously by $\alpha\colon G\to \text{Aut}(M)$, then $G\ni g\mapsto \Phi\circ\alpha_g\circ\Phi^{-1}\in \text{Aut}(N)$ defines a continuous action of $G$ on $N$. Moreover, we see that $\Phi (M^{\alpha})=N^{\Phi\circ\alpha\circ\Phi^{-1}}$. 
\par
The following theorem is a twisted version of the Connes-Takesaki Theorem \cite[III Corollary 3.4]{ConnesTakesaki1977} (or \cite[XI Corollary 2.8]{TakesakiOA2}). As is pointed out by Professor Reiji Tomtasu, it can be shown using Sutherland's stabilization trick \cite[II Proposition 2.1.3]{Sutherland1980} and Connes--Takesaki's theorem. Here, we include the full proof for completeness. 
% Although this might be obvious by using Sutherland's stabilizer trick, that is, $(M\rtimes_{\alpha,u}^\ell G )\vNtensor \B(L^2(G)) \cong (M \vNtensor\B(L^2(G)))\rtimes_{\tilde{\alpha}} G$, we provide a precise proof using $\tilde{\alpha}$. We learned about Sutherland's stabilizer trick from Professor Tomatsu.

\begin{theorem}\label{twistConnesfactority}
    For a von Neumann algebra $M$ acting on a Hilbert space $\mathcal{H}$, a locally compact abelian group $G$ and a cocycle action $(\alpha,u)\colon G\curvearrowright M$, the following two conditions are equivalent.
    \begin{enumerate}
        \item[(i)] $\Gamma(\alpha,u)=\hat{G}$ and $(\alpha,u)\colon G\curvearrowright M$ is centrally ergodic.
        \item[(ii)] $M\rtimes_{\alpha,u}^\ell G$ is a factor.
    \end{enumerate}
\end{theorem}
\begin{proof}
    $\text{(i)}\implies \text{(ii)}.$ Let $N=M\rtimes_{\alpha,u}^\ell G$. By assumption, we obtain $Z(N) \subseteq N^{\hat{\alpha}}$. Indeed, for each $n\in Z(N)$, it follows that $\hat{\alpha}_p(n)=n$ for all $p \in \hat{G}$. Since, for each $n \in Z(N)$, $(\pi_{\hat{\alpha}}(n)\xi)(p)=n\xi(p)$ for all $\xi \in L^2(\hat{G},\mathcal{K}),p\in \hat{G}$ where $\mathcal{K}=\mathcal{H}\otimes L^2(G)$, we see that $$\pi_{\hat{\alpha}}(n)\pi_{\hat{\alpha}}(x)=\pi_{\hat{\alpha}}(x)\pi_{\hat{\alpha}}(n), \pi_{\hat{\alpha}}(n)\lambda_{\hat{\alpha}}(q)=\lambda_{\hat{\alpha}}(q)\pi_{\hat{\alpha}}(n)\,\,(x\in N,q\in\hat{G}).$$ Hence, $\pi_{\hat{\alpha}}(Z(N))$ is contained in the center of $N \rtimes_{\hat{\alpha}}\hat{G}$. Put $N_r=M \rtimes_{\alpha,u}^r G \subseteq M \vNtensor \B(L^2(G))$. 
    Let $\Psi\colon M \vNtensor\B(L^2(G)) \to N \rtimes_{\hat{\alpha}}\hat{G}$ be the isomorphism obtained by Proposition \ref{cocycleduality}. In fact, $\Psi=\Phi^{-1}\circ \Theta$ where $\Phi\colon N\rtimes_{\hat{\alpha}}\hat{G}\to  N_r\rtimes_\beta^d G$ is a $*$-isomorphism defined in Proposition \ref{cocycleduality} and $\Theta \colon  M\vNtensor \B(L^2(G)) \to N_r \rtimes_\beta^d G$ is a $*$-isomorphism defined in \cite[Theorem 2]{NakagamiSutherland1979}. We see that 
    \begin{align*}
        \begin{cases}
         \Theta(I_r^\alpha(x))=I_r^\alpha(x)\otimes1_G, \quad (x\in M) \\
         \Theta(\lambda_r^u(g))=\lambda_r^u(g)\otimes\lambda_r(g), \quad(g\in G).
        \end{cases}
    \end{align*}
    Therefore, we obtain $\Psi(N_r)=\pi_{\hat{\alpha}}(N)$ by the calculation in Proposition \ref{cocycleduality}. Furthermore, since $\pi_{\hat{\alpha}}$ is an isomorphism of $N$ onto image, we have $\Psi(Z(N_r))=Z(\Psi(N_r))=Z(\pi_{\hat{\alpha}}(N))\subseteq \pi_{\hat{\alpha}}(Z(N))$.
    Noting that $N_r=(M \vNtensor \B(L^2(G)))^{\tilde{\alpha}}\subseteq M \vNtensor L^2(G)$ by \cite[Corollary 4]{NakagamiSutherland1979}, it follows that 
    \begin{align*}
        \Psi(Z(N_r))&=\Psi(Z(N_r)^{\tilde{\alpha}})= \Psi(Z(N_r))^{\Psi\circ\tilde{\alpha}\circ \Psi^{-1}} \subseteq \pi_{\hat{\alpha}}(Z(N))^{\Psi\circ\tilde{\alpha}\circ \Psi^{-1}} \subseteq Z(N\rtimes_{\hat{\alpha}}\hat{G})^{\Psi\circ\tilde{\alpha}\circ \Psi^{-1}} \\
        &= \Psi(Z(M \vNtensor \B(L^2(G))))^{\Psi\circ\tilde{\alpha}\circ \Psi^{-1}}=\Psi(Z(M\vNtensor\B
        (L^2(G)))^{\tilde{\alpha}}).
    \end{align*}
    We shall show $Z(M\vNtensor\B
    (L^2(G)))^{\tilde{\alpha}}=\C$. Since $Z(M\vNtensor \B(L^2(G)))=Z(M)\vNtensor \C$, for $x \otimes 1_G \in Z(M) \vNtensor \C
    $, we calculate 
    \begin{align*}
        (\tilde{\alpha}_t(x\otimes 1_G)\xi)(s) &=(\tilde{\lambda}^u(t)(\alpha_t\otimes\iota
        )(x\otimes 1_G)\tilde{\lambda}^u(t)^*\xi)(s) 
        \\&= u(t,t^{-1}s)^*((\alpha_t(x)\otimes 1_G)\tilde{\lambda}^u(t)^*\xi)(t^{-1}s) 
        \\&= u(t,t^{-1}s)^*\alpha_t(x)(\tilde{\lambda}^u(t)^*\xi)(t^{-1}s)
        \\&= u(t,t^{-1}s)^*\alpha_t(x)u(t,t^{-1}s)\xi(s)
        \\&= \alpha_t(x)\xi(s)
        \\&= ((\alpha_t(x)\otimes1_G)\xi)(s)            
    \end{align*}
    for all $\xi \in \mathcal{H} \otimes L^2(G), s,t\in G$. Thus, we conclude that $\tilde{\alpha}_t(x\otimes 1_G)=\alpha_t(x)\otimes1_G\,(t\in G)$. Hence, if $\tilde{\alpha}_t(x\otimes 1_G)=x\otimes1_G$ for all $t\in G$, then $\alpha_t(x)=x\, (t\in G)$ holds. Central ergodicity of the cocycle action implies $Z(M\vNtensor \B(L^2(G)))^{\tilde{\alpha}}=Z(M)^\alpha \vNtensor\C=\C$. We obtain $Z(N)\cong Z(N_r) \cong \Psi(Z(N_r)) = \C$. That is, $N$ is a factor. \\
    $\text{(ii)}\implies\text{(i)}.$ If $N=M \rtimes_{\alpha,u}^\ell G$ is a factor, then the kernel of the restriction of the action $\hat{\alpha}$ to $Z(N)=\C$ must be the entire $\hat{G}$. It remains to show the central ergodicity of the cocycle action. For, $x,y \in M, s\in G$, we see that 
    \begin{align*}
        \begin{cases}
            I_\ell^\alpha(x)I_\ell^{\alpha}(y)=I_\ell^\alpha(xy),\\
            \lambda_\ell^u(s)I_\ell^\alpha(x)\lambda_\ell^u(s)^*=I_\ell^\alpha(\alpha_s(x)).
        \end{cases}
    \end{align*}
    Now, suppose that $(\alpha,u)$ is not centrally ergodic. Then, there exists $x \in Z(M)^\alpha$ such that $x \notin \C 1_M$. For $y\in M$, since $x \in Z(M)$, $$I_\ell^\alpha(x)I_\ell^{\alpha}(y)=I_\ell^\alpha(xy)=I_\ell^\alpha(yx)=I_\ell^\alpha(y)I_\ell^{\alpha}(x).$$
    Moreover, by the covariant relation, for $s\in G$, $$\lambda_\ell^u(s)I_\ell^\alpha(x)\lambda_\ell^u(s)^*=I_\ell^\alpha(\alpha_s(x))=I_\ell^\alpha(x),$$
    that is $\lambda_\ell^u(s)I_\ell^\alpha(x)=I_\ell^\alpha(x)\lambda_\ell^u(s)$. It follows that $I_\ell^\alpha(x)$ is a non trivial central element of $M\rtimes_{\alpha,u}^\ell G$, which is a contradiction.
\end{proof}

\section{Application to almost unimodular groups}
In this section, we apply Theorem \ref{twistConnesfactority} to group von Neumann algebras of groups which satisfy that $G_1$ is open, called almost unimodular groups \cite[Definition 2.2]{guinto2025unimodulargroups}. We use the following proposition. We provide a proof for later use. Note that locally compact groups are assumed to be second countable in this section as mentioned earlier. 

\begin{proposition}
[{\cite[II, Proposition 3.1.7]{Sutherland1980}}]\label{cocycleresol}
    For a locally compact group $G$, a discrete group $K$ and an open normal subgroup $H$ of $G$, if $G/H \cong K$ as topological groups and $H$ is unimodular, then there exists a cocycle action $(\alpha,u)$ of $K$ on $L(H)$ such that $L(G) \cong L(H) \rtimes_{\alpha,u}^\ell K$ holds.
\end{proposition}
\begin{proof}
    Fix representatives $(g_i)_{i\in I}\subseteq G$ for the $H$ cosets. That is, $(g_i)_{i\in I}$ satisfies $G=\bigsqcup_{i\in I}g_i H$. We can take as $e_G\in (g_i)_{i\in I}$. By assumption, there exists an injective map $\sigma\colon K\to G$ such that $q\circ \sigma\colon K\to G/H$ is a topological group isomorphism, where $q\colon G\to G/H$ is the canonical quotient map. One can normalize this map as $\sigma(e_K)=e_G$. Moreover, we can take this map as a continuous map since $K$ is discrete. Set $\alpha^K\colon K\to \text{Aut}(H)$ as $\alpha_k^K(h)=\sigma(k)h\sigma(k)^{-1}$ and $c^K\colon K\times K \to H$ as $c^K(k_1,k_2)=\sigma(k_1)\sigma(k_2)\sigma(k_1k_2)^{-1}$. Then by a straightforward calculations, one can see that $(\alpha^K,c^K)$ is a continuous cocycle action of $K$ on $H$. Furthermore, it can be shown that $G\ni g=h\sigma(k)\mapsto (h,k)\in H\rtimes_{\alpha^K,c^K}K$ is a topological group isomorphism.\par
    We define $\beta_k\in \text{Aut}(H)$ by $\beta_k(h)=\sigma(k^{-1})^{-1}h\sigma(k^{-1})$ for each $k\in K$ and $c(k_1,k_2) \in H$ by $c(k_1,k_2)=\sigma(k_1^{-1})^{-1}\sigma(k_2^{-1})^{-1}\sigma(k_2^{-1}k_1^{-1})$ for each $k_1,k_2\in K$. Furthermore, we define the unitary $v_k$ on $L^2(H)$ by $$(v_k\xi)(h)=\delta(k)^{1/2}\xi(\beta_k^{-1}(h))$$
    for each $k \in K$ where $\xi \in L^2(H)$ and $\delta(k) \in (0,\infty)$ is a constant which satisfies $\int_H\xi(h)\,d\mu_H(h)=\delta(k)\int_H\xi(\beta_k^{-1}(h))\,d\mu_{H}(h)$ for all $\xi \in L^1(H)$. Note that $\delta\colon K\to (0,\infty)$ is a homomorphism since $H$ is unimodular. We need unimodularity because $\beta_{k_1}^{-1}\beta_{k_2}^{-1}=\beta_{k_2k_1}^{-1}\circ\text{Ad}(c(k_2,k_1))^{-1}$ for all $k_1,k_2 \in K$. For $k\in K$, since $\alpha_k(\lambda_\ell^{H}(h))=\lambda_\ell^{H}(\beta_k(h))\,\,(h\in H)$, it follows that $L(H) \ni x\mapsto v_kxv_k^* \in \B(L^2(H))$ is an automorphism on $L(H)$. Now, if we define $\alpha_k\colon L(H)\ni x\mapsto v_kxv_k^* \in L(H)$ and $u(k_1,k_2)=\lambda_\ell^H(c(k_1,k_2))$, then it follows that $(\alpha,u)\colon K\curvearrowright L(H)$ is a cocycle action. By calculation, we obtain the cocycle equations. That is, they satisfy the following equations: 
    $$\alpha_{k_1}\alpha_{k_2}=\text{Ad}(u(k_1,k_2))\alpha_{k_1 k_2}, u(k_1,k_2)u(k_1k_2,k_3)=\alpha_{k_1}(u(k_2,k_3))u(k_1,k_2k_3)$$
    for all $k_1,k_2,k_3 \in K$. Moreover, $u(k,e)=u(e,k)=1_{L(H)}\,\,(k\in K)$ holds. Since $K$ is discrete, Borel requirement is satisfied. \par 
    Finally, we shall show that $L(G) \cong L(H) \rtimes_{\alpha,u}^\ell K$. Note that $L(G) \cong L(H\rtimes_{\alpha^K,c^K}K)$. We identify $L^2(H) \otimes L^2(K)$ with $L^2(H \times K)$. Under this identification, $L(H) \rtimes_{\alpha,u}^\ell K$ is generated by the operators $I_{\ell}^{\alpha}(\lambda_{\ell}^H(h_0))\,(h_0\in H)$ and $\lambda_{\ell}^u(k_0)\,(k_0\in K)$ given by 
    \begin{align*}
        \begin{cases}
            (I_\ell^{\alpha}(\lambda_\ell^H(h_0))\xi)(h,k)=((\alpha_{k^{-1}}(\lambda_\ell^H(h_0))\otimes1_K)\xi)(h,k)=\xi(\beta_{k^{-1}}(h_0)^{-1}h,k), \\
            (\lambda_\ell^u(k_0)\xi)(h,k)=((u(k^{-1},k_0)\otimes1_K)\xi)(h,k_0^{-1}k)=\xi(c(k^{-1},k_0)^{-1}h,k_0^{-1}k)
        \end{cases}
    \end{align*} 
    where $\xi \in L^2(H \times K),h,h_0\in H,k,k_0\in K$. Note that $\delta(k)\,d\mu_{H}(h)d\mu_K(k)$ is a left Haar measure on $H\rtimes_{\alpha^K,c^K}K$. Let $U$ be the unitary from $L^2(H\times K)$ to $L^2(H\rtimes_{\alpha^K,c^K} K)$ defined by $$(U\xi)(h,k)=\xi(\beta_{k^{-1}}(h),k).$$ 
     We see that $$(U^*\xi)(h,k)=\xi(\beta_{k^{-1}}^{-1}(h),k).$$
    We calculate, for $h_0 \in H$,
    \begin{align*}
        (U(I_\ell^{\alpha}(\lambda_\ell^H(h_0))U^*\xi)(h,k)&= (I_\ell^{\alpha}((\lambda_\ell^H(h_0))U^*\xi)(\beta_{k^{-1}}(h),k) \\
        &= (U^*\xi)(\beta_{k^{-1}}(h_0)^{-1}\beta_{k^{-1}}(h),k) \\
        &= \xi(h_0^{-1}h,k).
    \end{align*}
    Also, for $k_0 \in K$,
    \begin{align*}
        (U\lambda_\ell^u(k_0)U^*\xi)(h,k) &=(\lambda_\ell^u(k_0)U^*\xi)(\beta_{k^{-1}}(h),k) \\
        &= (U^*\xi)(c(k^{-1},k_0)^{-1}\beta_{k^{-1}}(h),k_0^{-1}k) \\
        &= \xi(\beta_{k^{-1}k_0}^{-1}(c(k^{-1},k_0)^{-1}\beta_{k^{-1}}(h)),k_0^{-1}k).
    \end{align*}
    On the other hand, with the identification $G \cong H\rtimes_{\alpha^K,c^K}K$, for $h_0 \in H$, 
    \begin{align*}
        (\lambda_\ell^{G}(h_0,e_K)\xi)(h,k) &=\xi((h_0,e_K)^{-1}(h,k)) \\
        &=\xi((h_0^{-1},e_K)(h,k)) \\
        &=\xi(h_0^{-1}h,k).
    \end{align*}
    Also, for $k_0\in K$, let $\sigma(k_0^{-1})^{-1}=h'\sigma(k_0) \in G$, then
    \begin{align*}
        (\lambda_\ell^{G}(h',k_0)\xi)(h,k)
        &= \xi((h',k_0)^{-1}(h,k)) \\
        &= \xi(((\alpha_{k_0}^K)^{-1}(h'^{-1}c^K(k_0,k_0^{-1})^{-1}),k_0^{-1})(h,k)) \\
        &= \xi((\alpha^K_{k_0})^{-1}(h'^{-1}c^K(k_0,k_0^{-1})^{-1})\alpha^K_{k_0^{-1}}(h)c^K(k_0^{-1},k),k_0^{-1}k)
    \end{align*}
    Furthermore, 
    \begin{align*}
        \beta_{k^{-1}k_0}^{-1}(c(k^{-1},k_0)^{-1}\beta_{k^{-1}}(h))
        &=\sigma(k_0^{-1}k)c(k^{-1},k_0)^{-1}\beta_{k^{-1}}(h)\sigma(k_0^{-1}k)^{-1} \\
        &= \sigma(k_0^{-1}k)\sigma(k_0^{-1}k)^{-1}\sigma(k_0^{-1})\sigma(k)\sigma(k)^{-1}h\sigma(k)\sigma(k_0^{-1}k)^{-1} \\
        &=\sigma(k_0^{-1})h\sigma(k)\sigma(k_0^{-1}k)^{-1},
    \end{align*}
    \begin{multline*}
         (\alpha^K_{k_0})^{-1}(h'^{-1}c^K(k_0,k_0^{-1})^{-1})\alpha^K_{k_0^{-1}}(h)c^K(k_0^{-1},k) \\
        = \sigma(k_0)^{-1}h'^{-1}c^K(k_0,k_0^{-1})^{-1}\sigma(k_0)\sigma(k_0^{-1})h\sigma(k_0^{-1})^{-1}\sigma(k_0^{-1})\sigma(k)\sigma(k_0^{-1}k)^{-1} \\=\sigma(k_0^{-1})\sigma(k_0^{-1})^{-1}\sigma(k_0)^{-1}\sigma(k_0)\sigma(k_0^{-1})h\sigma(k)\sigma(k_0^{-1}k)^{-1} =\sigma(k_0^{-1})h\sigma(k)\sigma(k_0^{-1}k)^{-1}.
    \end{multline*}
    Therefore, we conclude that 
    \begin{align*}
        \begin{cases}
            (U(I_\ell^{\alpha}(\lambda_\ell^H(h_0))U^*\xi)(h,k)=(\lambda_\ell^{G}(h_0,e_K)\xi)(h,k),\\
            (U\lambda_\ell^u(k_0)U^*\xi)(h,k)=(\lambda_\ell^{G}(\sigma(k_0^{-1})^{-1}\sigma(k_0)^{-1},k_0)\xi)(h,k)
        \end{cases}
    \end{align*}
    for all $h,h_0\in H, k,k_0\in K$. Therefore, we conclude that $$U(L(H)\rtimes_{\alpha,u}^\ell K)U^*=L(H\rtimes_{\alpha^K,c^K} K)\cong L(G).$$    
\end{proof}

\begin{remark}
    Note that the cocycle actions $(\alpha^K,c^K) \colon K \curvearrowright H$ and $(\alpha,u) \colon K \curvearrowright L(H)$ depend on a section $\sigma \colon K \to G$. In $\S 5$, we will focus on a section $\sigma$ which satisfies $\sigma(k^{-1})=\sigma(k)^{-1}$ for all $k \in K$, called inverese preserving. With an inverese preserving section, $(\alpha^K, c^K)$ and $(\alpha,u)$ are compatible in the sense that $(\alpha^K, c^K)=(\beta,c)$ in the proof of the above proposition.
\end{remark}

For an almost unimodular group $G$, since $G_1=\text{ker}\Delta_G$ is an open normal subgroup of $G$ which is unimodular, we obtain $G/G_1 \cong \Delta_G(G)$ as topological groups, where $\Delta_G(G)$ is equipped with the discrete topology. Therefore, we obtain the following corollary.

\begin{corollary}[{\cite[Theorem 5.1]{guinto2025unimodulargroups}}]
    For an almost unimodular group $G$, there is a cocycle action $(\alpha,u)\colon \Delta_G(G) \curvearrowright L(G_1)$ such that $$L(G) \cong L(G_1)\rtimes_{\alpha,u}^\ell\Delta_G(G).$$ 
\end{corollary}

\begin{remark}\label{remark2}
    Although this fact is used in \cite[Example 3.8, Theorem 5.1]{guinto2025unimodulargroups}, we need a slight modification to the construction of the cocycle action. We denote $\Delta_G(G)=K$. In \cite[Example 3.8, Theorem 5.1]{guinto2025unimodulargroups}, the twisted crossed product decomposition of $G$ is obtained by above $(\alpha^K,c^K)$. But, the group operation is different from Definition \ref{twistcrossprodgrp}. Also, the twisted crossed product decomposition of $L(G)$ is given by $\check{\alpha}_{k_1}(\lambda_\ell^{G_1}(s))=\lambda_\ell^{G_1}(\alpha^K_{k_1}(s)), \check{c}(k_1,k_2)=\lambda_\ell^{G_1}(c^K(k_1,k_2))\,\,(s\in G_1, k_1,k_2 \in K)$ in \cite{guinto2025unimodulargroups}.
    However, in Proposition \ref{cocycleresol}, we have to construct two cocycle actions of $K$ on $G_1$. It can be observed in the proof that $(\alpha^K,c^K)$ is used for the topological group decomposition $G \cong G_1 \rtimes_{\alpha^K,c^K}K$ and $(\alpha,u)$ is used for the von Neumann algebra decomposition $L(G) \cong L(G_1)\rtimes _{\alpha,u}^\ell K$.
\end{remark}

Since $\Delta_G(G) \le \R_+$ is an abelian group, we can apply Theorem \ref{twistConnesfactority}. In fact, central ergodicity of the cocycle action is equivalent to the factoriality of $L(G)$. We shall show it. \par 
In \cite[Theorem 2.1]{guinto2025unimodulargroups}, it was shown that $\varphi_G$ is an almost periodic weight if $G$ is an almost unimodular group 
%(it suffices to assume that $G_1$ is open, that is, $G$ is an almost unimodular group defined in \cite[Definition 2.2]{guinto2025unimodulargroups}).
Therefore, fix an almost unimodular group $G$, then there exists a continuous action $\tilde{\sigma}\colon \widehat{\Delta_G(G)} \curvearrowright L(G)$ called the point modular extension of $\sigma^{\varphi_G}\colon \R \curvearrowright L(G)$ (see \cite[Proposition 1.1]{Connes1974},\cite[Section 1.3]{Dykema1995} and \cite[Section 1.4]{GuintoNelson2024}). Here, $\text{Aut}(L(G))$ is equipped with the u-topology.  Let $\iota\colon \Delta_G(G) \to \R_+$ be the inclusion map. We obtain a continuous group homomorphism which is the transpose of the inclusion map $\hat{\iota}\colon \R \to \widehat{\Delta_G(G)}$ defined by dual pairings as 
$$\nai{k}{\hat{\iota}(t)}=\nai{\iota(k)}{t}=k^{it}\,\,(k\in \Delta_G(G),t\in\R).$$
Here, we identify $\R \cong \hat{\R}_+$. Injectivity of inclusion implies the density of $\hat{\iota}(\R)$ in $\widehat{\Delta_G(G)}$. In this setting, $\tilde{\sigma}_{\hat{\iota}(t)}=\sigma^{\varphi_G}_t\,\,(t\in \R)$ holds. We have the following proposition. This was mentioned in the proof of \cite[Theorem 5.1]{guinto2025unimodulargroups}, but since we use a different definition from \cite{guinto2025unimodulargroups}, we provide full details. 

\begin{proposition}\label{fullspectrum}
    For an almost unimodular group $G$, let $(\alpha,u)\colon \Delta_G(G) \curvearrowright L(G_1)$ be the cocycle action obtained by Proposition \ref{cocycleresol}, then $\Gamma(\alpha,u)=\widehat{\Delta_G(G)}$ holds.
\end{proposition}
\begin{proof}
    We denote $K=\Delta_G(G)$. We identify $G \cong G_1 \rtimes_{\alpha^K,c^K} K$. Let $U$ be the unitary defined in the proof of Proposition \ref{cocycleresol} and $\Phi\colon L(G_1)\rtimes_{\alpha,u}^\ell K \to L(G_1 \rtimes_{\alpha^K,c^K} K)$ be the isomorphism defined by the adjoint of $U$. We calculate $\Phi\circ\hat{\alpha}\circ\Phi^{-1}$. Fix $p \in \hat{K}$. For $s\in G_1$, 
    \begin{align*}
        \Phi\circ\hat{\alpha}_p\circ\Phi^{-1}(\lambda^G_\ell(s,1))&=\Phi\circ\hat{\alpha}_p(I_\ell^\alpha(\lambda_\ell^{G_1}(s))) =\Phi(I_\ell^\alpha(\lambda_\ell^{G_1}(s)))=\lambda^G_\ell(s,1).
    \end{align*}
    Also, for $k\in K$, let $(s_0,k)\in G_1 \rtimes_{\alpha^K,c^K} K$ be the element which satisfies $\sigma(k^{-1})^{-1}=s_0\sigma(k) \in G$, then 
    \begin{align*}
        \Phi\circ\hat{\alpha}_p\circ\Phi^{-1}(\lambda^G_\ell(s_0,k))&=\Phi\circ\hat{\alpha}_p(\lambda_\ell^u(k))\\&=\Phi(\overline{\nai{k}{p}}\lambda_\ell^u(k))
        = \overline{\nai{k}{p}}\lambda^G_\ell(s_0,k).
    \end{align*}
    Now, for each $(s,k)\in G_1 \rtimes_{\alpha^K,c^K} K$, there exists $s_0 \in G_1$ which satisfies $\sigma(k^{-1})^{-1}=s_0\sigma(k) \in G$ and $(s,k)=(ss_0^{-1},1)(s_0,k)$. Therefore, 
    \begin{align*}
        \Phi\circ\hat{\alpha}_p\circ\Phi^{-1}(\lambda^G_\ell(s,k)) &= \Phi\circ\hat{\alpha}_p\circ\Phi^{-1}(\lambda^G_\ell(ss_0^{-1},1))\cdot\Phi\circ\hat{\alpha}_p\circ\Phi^{-1}(\lambda^G_\ell(s_0,k))
        \\&= \lambda^G_\ell(ss_0^{-1},1)\cdot \overline{\nai{k}{p}}\lambda^G_\ell(s_0,k)
        \\&= \overline{\nai{k}{p}}\lambda^G_\ell(ss_0^{-1},1)\lambda^G_\ell(s_0,k)
        \\&= \overline{\nai{k}{p}} \lambda_\ell^{G}(s,k).
    \end{align*}
    It follows that $\Phi\circ\hat{\alpha}_p\circ\Phi^{-1}(\lambda^G_\ell(s,k))=\overline{\nai{k}{p}} \lambda_\ell^{G}(s,k)$ for all $p\in \hat{K},s\in G_1, k\in K$. Next, we calculate $\tilde{\sigma}$. Note that $\Delta_G(s\sigma(k))=k$, whence $\Delta_{G}(s,k)=k$. For $(s,k)\in G_1 \rtimes_{\alpha^K,c^K} K$, we have
    \begin{align*}
        \tilde{\sigma}_{\hat{\iota}(t)}(\lambda_\ell^G(s,k))&=\sigma^{\varphi_G}_t(\lambda_\ell^G(s,k))=\Delta_G((s,k))^{it}\lambda_\ell^G(s,k)\\&=k^{it}\lambda_\ell^G(s,k)=\nai{\iota(k)}{t}\lambda_\ell^G(s,k)=\nai{k}{\hat{\iota}(t)}\lambda_\ell^G(s,k)\,\,(t\in\R).
    \end{align*}
    Hence, density of $\hat{\iota}(\R) \subseteq \hat{K}$ and continuity $\hat{K} \ni p \mapsto \tilde{\sigma}_p \in \text{Aut}(L(G))$ implies 
    $$\tilde{\sigma}_p(\lambda_\ell^G(s,k))=\nai{k}{p}\lambda_\ell^G(s,k),$$
    for all $p\in \hat{K},s\in G_1, k\in K$. Therefore, we obtain $$\Phi\circ\hat{\alpha}_p\circ\Phi^{-1}(\lambda^G_\ell(g))=\tilde{\sigma}_{-p}(\lambda_\ell^G(g)),$$
    for all $p\in \hat{K},g\in G$. It follows that $$\Phi\circ\hat{\alpha}_p\circ\Phi^{-1}(x)=\tilde{\sigma}_{-p}(x),$$
    for all $p\in \hat{K}, x\in L(G)$. Take $x\in Z(L(G))$, since $Z(L(G)) \subseteq L(G)^{\varphi_G}$, we have 
    $$\tilde{\sigma}_{\hat{\iota}(t)}(x)=\sigma_t^{\varphi_G}(x)=x\,\, (t\in \R).$$
    Thus, $\tilde{\sigma}_p(x)=x$ for all $p\in \hat{K}$ holds. We conclude that $\Phi\circ\hat{\alpha}_p\circ\Phi^{-1}(x)=x$ for all $x\in Z(L(G)),p\in\hat{K}$. Since $\Phi^{-1}(Z(L(G)))=Z(L(G_1)\rtimes_{\alpha,u}^\ell K)$, it follows that $\hat{\alpha}_p(x)=x$ for all $p \in \hat{K}, x\in Z(L(G_1)\rtimes_{\alpha,u}^\ell K)$. That is, $\Gamma(\alpha,u)=\hat{K}$ holds.
\end{proof}

Combining Theorem \ref{twistConnesfactority} and Proposition \ref{fullspectrum}, the following theorem holds.

\begin{theorem}\label{tdlcfactority}
    For an almost unimodular group $G$, let $(\alpha,u)\colon \Delta_G(G) \curvearrowright L(G_1)$ be the cocycle action obtained by Proposition \ref{cocycleresol}, then the following two conditions are equivalent.
    \begin{enumerate}
        \item[(i)] $(\alpha,u)\colon \Delta_G(G) \curvearrowright L(G_1)$ is centrally ergodic.
        \item[(ii)] $L(G)$ is a factor.
    \end{enumerate}  
\end{theorem}

\begin{remark}
    If $G$ is unimodular, then Theorem \ref{tdlcfactority} is a trivial result because $\Delta_G(G)=\{1\}$ and hence $(\alpha,u)$ becomes a trivial action. Also, in non-unimodular case, if $L(G_1)$ is a factor, then $L(G)$ is a factor because $L(G_1) \cong L(G)^{\varphi_G}$. In this case, the cocycle action $(\alpha,u)\colon \Delta_G(G) \curvearrowright L(G_1)$ must be centrally ergodic.
\end{remark}

 Theorem \ref{tdlcfactority} is still useful when $L(G_1)$ is not a factor because there exists an almost unimodular group $G$ which satisfies that $L(G)$ is a factor and $L(G_1)$ is not a factor. For example, let $\Q_p$ be a $p$-adic numbers as an additive group, also $\Q_p^*$ be a multiplicative group. Then $G=\Q_p \rtimes \Q_p^*$ is a desired one. It is well known that $G$ is a totally disconnected group (and hence is an almost unimodular group) and $L(G)$ is a $\rm{I}_\infty$ factor (see \cite[Section 3]{Blackadar1977}). In this case, $G_1=\Q_p \rtimes \Z_p^*=\{(a,b)\in G:|b|_p=1\}$ and it follows that  $\int_{\Z_p}\lambda^{G_1}(a,1)\,d\mu_{\Q_p}(a)\in L(G_1)$ is a non-trivial central element (see, e.g.  \cite[Example 5.3 (1)]{guinto2025unimodulargroups}). Thus, there is a centrally ergodic cocycle action even if $L(G_1)$ is not a factor. 
In fact, we can directly verify that $(\alpha,u)\colon \Delta_G(G) \curvearrowright L(G_1)$ is centrally ergodic when $G=\Q_p \rtimes \Q_p^*,G_1=\Q_p \rtimes \Z_p^*$.

\begin{example}\label{afineex}
    Let $G=\Q_p \rtimes \Q_p^*,G_1=\text{ker}\Delta_G=\Q_p \rtimes \Z_p^*$. Since $\Delta_G(a,b)=|b|_p^{-1}$, we can take $\sigma \colon \Delta_G(G) \to G$ as a homomorphism $p^n \mapsto (0,p^{n})$. In this case, the cocycle action $(\alpha,u)$ is an action $\alpha$ given by $$\alpha_{p^n}(\lambda(a,b))=\lambda(p^{n}a,b)\,\,(n\in \Z).$$ Here, $\Q_p \rtimes \Z_p^*$ is obtained by the action $\gamma\colon \Z_p^* \curvearrowright \Q_p$ defined by $\gamma_s(x)=sx\,\,(s\in\Z_p^*, x\in\Q_p)$. It follows that $$L(\Q_p \rtimes \Z_p^*) \cong L(\Q_p) \rtimes_{\overline{\gamma}}\Z_p^*\cong L^\infty(\hat{\Q}_p)\rtimes_{\tilde{\gamma}}\Z_p^* \cong L^\infty(\Q_p) \rtimes_{\dot{\gamma}}\Z_p^*,$$
    where $\dot{\gamma}\colon \Z_p^* \curvearrowright L^\infty(\Q_p)$ is induced by the isomorphism $\Phi\colon L(\Q_p) \to L^\infty(\Q_p)$. We can check that $\dot{\gamma}_s=\beta_{s^{-1}}$ where $\beta\colon \Z_p^*\curvearrowright L^\infty(\Q_p)$ is obtained by $\gamma\colon \Z_p^* \curvearrowright \Q_p$ as a group measure space construction. That is, 
    $$(\dot{\gamma}_s(m(f))\xi)(x)=f(sx)\xi(x)\,\,(f\in L^\infty(\Q_p),\xi\in L^2(\Q_p),s\in \Z_p^*,x\in\Q_p).$$ Put $\Psi\colon L(\Q_p \rtimes \Z_p^*) \to L^\infty(\Q_p) \rtimes_{\dot{\gamma}}\Z_p^*$ be the isomorphism. By calculation, we have $\Psi(\lambda_\ell^{G_1}(x,1))=\pi_{\dot{\gamma}}(m(\theta(-x)))$ where $\theta\colon \Q_p\to\hat{\Q}_p$ is an isomorphism between topological groups. Hence, for $x\in \Q_p$, 
    $$\Psi\circ\alpha_{p^n}\circ\Psi^{-1}(\pi_{\dot{\gamma}}(m(\theta(x))))=\pi_{\dot{\gamma}}(m(\theta(p^{n}x))).$$
    Note that $\nai{y}{\theta(x)}=\nai{xy}{\theta(1)}$ for all $y\in \Q_p$, we see that $\theta(p^{n}x)=\theta(x)_{p^{n}}$ where $\theta(x)_{p^{n}}(y)=\theta(x)(p^{n}y)$. By Lemma \ref{dualgrpdensity}, we conclude that 
    $$\Psi\circ\alpha_{p^n}\circ\Psi^{-1}(\pi_{\dot{\gamma}}(m(f)))=\pi_{\dot{\gamma}}(m(f_{p^{n}}))$$ for all $f\in L^\infty(\Q
    _p)$. Since $\dot{\gamma}$ is a free action in the sense of \cite[Definition 2.1]{Sauvageot1977}, we also see that $Z(L^\infty(\Q_p) \rtimes_{\dot{\gamma}}\Z_p^*)=\pi_{\dot{\gamma}}(m(L^\infty(\Q_p))^{\dot{\gamma}})$ by \cite[Corollary 2.2]{Sauvageot1977}. Thus,
    \begin{align*}
        \Psi(Z(L(\Q_p \rtimes \Z_p^*))^{\alpha})=Z(L^\infty(\Q_p) \rtimes_{\dot{\gamma}}\Z_p^*)^{\Psi\circ\alpha\circ\Psi^{-1}} = \pi_{\dot{\gamma}}(m(L^\infty(\Q_p))^{\dot{\gamma}})^{\Psi\circ\alpha\circ\Psi^{-1}}.
    \end{align*}
    Take $f\in L^\infty(\Q_p)$ such that $\pi_{\dot{\gamma}}(m(f))$ is in the right hand side of the above equation. Then, $f$ satisfies $$f(sx)=f(x)\,\,(s\in \Z_p^*),\quad f(p^{n}x)=f(x)\,\,(n\in \Z)$$
    for all $x\in \Q_p$. Since for each $x\in \Q_p$ can be factorized as $x=p^ns$ where $n\in\Z,s\in \Z_p^*$, we conclude that $f(x)=f(1)$ for all $x\in \Q_p$. That is, $f=c\cdot1_{\Q_p}$ for some constant $c \in \C$. Therefore, $Z(L(\Q_p \rtimes \Z_p^*))^{\alpha}=\Psi^{-1}(\C)=\C$ holds. That is, the action is centrally ergodic. 
\end{example}

In this case, $L(G)$ admits a crossed product decomposition because of the construction of $G$, that is, $L(G)=L(\Q_p \rtimes \Q_p^*) \cong L(\Q_p) \rtimes \Q_p^*$ (see, e.g. \cite[Proposition 2.2]{Sutherland1978}). But, we use another decomposition $L(G) \cong L(G_1) \rtimes \Delta_G(G)$ in Example \ref{afineex}. Thus, our theorem also covers groups which do not admit genuine crossed product decompositions. 

\section{The S-invariant of the group von Neumann algebras}
Let $G$ be an almost unimodular group. In this section, we study the S-invariant $S(L(G))$ when $L(G)$ is a factor. First, we define a good section for twisted crossed product decomposition $L(G) \cong L(G_1) \rtimes_{\alpha,u}^\ell \Delta_G(G)$ as in Proposition \ref{cocycleresol}. Note that locally compact groups are assumed to be second countable as mentioned earlier.

\begin{definition}
    Let $G$ be an almost unimodular group. A section $\sigma\colon\Delta_G(G) \to G$ is inverse preserving if $\sigma(k^{-1})=\sigma(k)^{-1}\,\,(k\in \Delta_G(G))$ holds. 
\end{definition}

\begin{remark}\label{existenceofinvsec}
    Any almost unimodular group admits an inverse preserving section. Indeed, for arbitrary $\tilde{\sigma}\colon\Delta_G(G) \to G$, we can construct a new section $\sigma\colon\Delta_G(G) \to G$ defined by $\sigma(k)=\tilde{\sigma}(k), \sigma(k^{-1})=\tilde{\sigma}(k)^{-1}$ for $k \ge 1$. It can be shown that $\sigma$ is an inverse preserving section. Given an inverse preserving section, it follows that $\alpha^K=\beta$ and $c^K=c$ in Proposition \ref{cocycleresol}. Moreover, it follows that $\beta_k^{-1}=\beta_{k^{-1}}$ since $c(k,k^{-1})=e$ for all $k\in \Delta_G(G)$.
\end{remark}

We use the following Theorem due to Connes \cite{Connes1973}. We denote the set of projections of a von Neumann algebra $M$ by $M^P$.

\begin{theorem}[{\cite[Corollary 3.2.5(b)]{Connes1973}}]\label{ConnesSinv}
    Let $M$ be a factor and $\varphi$ be a normal faithful semifinite weight on $M$. Then, $$\textrm{S}(M)\backslash\{0\}=\bigcap_{0\ne e\in Z(M^{\varphi})^P}\textrm{Sp}(\Delta_{\varphi_e})\backslash\{0\}$$
    where $\varphi_e$ is a weight on $M_e=eMe$ defined by $\varphi_e(x)=\varphi(x)\,\,(x\in M_e).$
\end{theorem}

Now, we shall show the following lemma.

\begin{lemma}\label{spectrumformula}
    Let $G$ be an almost unimodular group. We identify $L(G) \cong L(G_1)\rtimes_{\alpha,u}^\ell \Delta_G(G)$ with an inverse preserving section $\sigma\colon \Delta_G(G) \to G$. Then, for each nonzero $e\in Z(L(G)^{\varphi_G})^P=Z(L(G_1))^P$, $$\text{Sp}(\Delta_{(\varphi_G)_e})\backslash\{0\}=\overline{\{k\in\Delta_G(G):\alpha_k(e)e\ne0\}}\backslash\{0\}.$$
\end{lemma}
\begin{proof}
    Note that $\Delta_{\varphi_G}$ is the linear operator on the Hilbert space $L^2(L(G),\varphi_G)$ and we can identify this operator with the operator $\Delta$ on the Hilbert space $L^2(G)$, which has the domain $\mathcal{D}(\Delta)=\{\xi\in L^2(G) : \int_G|\Delta_G(x)\xi(x)|^2\, d\mu_G(x)<\infty\}$ and is defined by $(\Delta\xi)(x)=\Delta_G(x)\xi(x)\,(\xi\in \mathcal{D}(\Delta))$. That is, there exists a unitary operator $U\colon L^2(G) \to L^2(L(G),\varphi_G)$ which satisfies $U\xi=\pi_l(\xi)$ for each $\xi \in C_c(G)$ and $U^*\Delta_{\varphi_G}U=\Delta$. Here, we use the fact that $C_C(G)$ is the core of $\mathcal{D}(\Delta)$. Let $\pi_{\varphi_G}$ be the semicyclic representation of $L(G)$ with respect to $\varphi_G$. For $e \in Z(L(G)^{\varphi_G})^P$, $\Delta_{(\varphi_G)_e}$ is the linear operator on the Hilbert space $L^2(eL(G)e,(\varphi_G)_e)=\pi_{\varphi_G}(e)J_{\varphi_G}\pi_{\varphi_G}(e) J_{\varphi_G}L^2(G, \varphi_G)$ which satisfies $$\Delta_{(\varphi_G)_e}=\Delta_{\varphi_G}\pi_{\varphi_G}(e)J_{\varphi_G}\pi_{\varphi_G}(e) J_{\varphi_G}$$ on $\pi_{\varphi_G}(e)J_{\varphi_G}\pi_{\varphi_G}(e) J_{\varphi_G}L^2(G, \varphi_G)$. Here, $J_{\varphi_G}$ is the modular conjugation associated with $\varphi_G$. Therefore, it follows that $\textrm{Sp}(\Delta_{(\varphi_G)_e})\backslash\{0\}=\textrm{Sp}(\Delta_{\varphi_G}\pi_{\varphi_G}(e)J_{\varphi_G}\pi_{\varphi_G}(e) J_{\varphi_G})\backslash\{0\}.$ 
    Let $J$ be the modular conjugation on $L^2(G)$, that is, $J$ is the antilinear operator defined by $(J\xi)(x)=\Delta_G(x)^{-1/2}\overline{\xi(x^{-1})}=\xi^*(x)$ for all $\xi\in L^2(G)$. It can be shown that $U^*J_{\varphi_G}U=J$ and $U^*\pi_{\varphi_G}(e)U=e$, and hence $\textrm{Sp}(\Delta_{\varphi_G}\pi_{\varphi_G}(e)J_{\varphi_G}\pi_{\varphi_G}(e) J_{\varphi_G})\backslash\{0\}=\textrm{Sp}(\Delta eJeJ)\backslash\{0\}$ holds. 
    Let $V$ be the unitary operator from $L^2(G_1 \rtimes_{\beta,c} \Delta_G(G))$ to $L^2(G)$ obtained by the topological group isomorphism $G_1\rtimes_{\beta,c}\Delta_G(G) \cong G$ as in Proposition \ref{cocycleresol}. That is, $V\colon L^2(G_1 \rtimes_{\beta,c} \Delta_G(G)) \to L^2(G)$ satisfies $$(V^*\xi)(s,k)=\xi(s\sigma(k))\,\,(\xi\in L^2(G), s\in G_1, k\in\Delta_G(G)).$$
    Let $W$ be the unitary operator from $L^2(G_1 \times \Delta_G(G))$ to $L^2(G_1 \rtimes_{\beta,c} \Delta_G(G))$ which gives the spatial isomorphism $L(G_1) \rtimes_{\alpha,u}^\ell \Delta_G(G) \cong L(G_1 \rtimes_{\beta,c}\Delta_G(G))$ as in Proposition \ref{cocycleresol}. That is, $W\colon L^2(G_1 \times \Delta_G(G))\to L^2(G_1 \rtimes_{\beta,c} \Delta_G(G))$ satisfies 
    $$(W\xi)(s,k)=\xi(\beta_{k^{-1}}(s),k),\,\,(W^*\eta)(s,k)=\eta(\beta_{k^{-1}}^{-1}(s),k)$$ for all $\xi \in L^2(G_1 \times \Delta_G(G)), \eta\in L^2(G_1 \rtimes_{\beta,c} \Delta_G(G)), s\in G_1, k\in\Delta_G(G)$ and $W (L(G_1)\rtimes_{\alpha,u}^\ell \Delta_G(G))W^*=L(G_1\rtimes_{\beta,c}\Delta_G(G))$ holds. It follows that $U_0:=VW$ gives a spatial isomorphism $L(G_1) \rtimes_{\alpha,u}^\ell \Delta_G(G) \cong L(G)$.
    Thus, we see that $$\textrm{Sp}(\Delta eJeJ)\backslash\{0\}=\textrm{Sp}((U_0^*\Delta U_0)(U_0^* eU_0)(U_0^*JU_0)(U_0^*eU_0)(U_0^*JU_0))\backslash\{0\}.$$ Put $\tilde{\Delta}=U_0^*\Delta U_0$ with the domain $\mathcal{D}(\tilde{\Delta})=U_0^*\mathcal{D}(\Delta)$. For $\xi \in \mathcal{D}(\tilde{\Delta})$, 
    \begin{align*}
        (\tilde{\Delta}\xi)(s,k)=(W^*V^*\Delta VW\xi)(s,k)&=(\Delta VW\xi)(\beta_{k^{-1}}^{-1}(s)\sigma(k))\\&=k(VW\xi)(\beta_{k^{-1}}^{-1}(s)\sigma(k))\\&= k(W\xi)(\beta_{k^{-1}}^{-1}(s),k) \\&=k\cdot\xi(s,k).
    \end{align*}
    Moreover, put $\tilde{J}=U_0^*JU_0$. Then, for $\xi \in L^2(G_1 \times \Delta_G(G))$, we have 
    \begin{align*}
        (\tilde{J}\xi)(s,k)=(W^*V^*JVW\xi)(s,k)&= (JVW\xi)(\beta_{k^{-1}}^{-1}(s)\sigma(k)) \\&=(JVW\xi)(\sigma(k)s\sigma(k)^{-1}\sigma(k))
        \\&=(JVW\xi)(\sigma(k)s)
        \\&=k^{-1/2}\overline{(VW\xi)(s^{-1}\sigma(k^{-1}))}
        \\&=k^{-1/2}\overline{(W\xi)(s^{-1},k^{-1})}
        \\&=k^{-1/2}\overline{\xi(\beta_{k}(s^{-1}),k^{-1})}.
    \end{align*}
    Let $J_1$ be the modular conjugation on $L^2(G_1)$ and $v_k^*$ be the unitary on $L^2(G_1)$ for each $k\in \Delta_G(G)$ defined as in Proposition \ref{cocycleresol}, which satisfies $$(v_k^*\xi)(s)=(v_{k^{-1}}\xi)(s)=k^{-1/2}\xi(\beta_k(s))\,\,(\xi\in L^2(G_1),s\in G_1).$$
    It then follows that 
    $$(\tilde{J}\xi)(s,k)=(v_k^*\otimes1)(J_1\otimes1)\xi(s,k^{-1})=(J_1\otimes1)(v_k^*\otimes1)\xi(s,k^{-1}).$$
    Since $e\in L(G_1)$, we have $U_0^*eU_0=I_\ell^\alpha(e)$. Note that $J_1xJ_1=x^*$ for all $x\in Z(L(G_1))$, for $\xi\in L^2(G_1\times \Delta_G(G))$, 
    \begin{align*}
        (\tilde{J}I_\ell^\alpha(e)\tilde{J}\xi)(s,k)&=(v_k^*\otimes1)(J_1\otimes1)(I_\ell^\alpha(e)\tilde{J}\xi)(s,k^{-1})
        \\&=(v_k^*\otimes1)(J_1\otimes1)(\alpha_k(e)\otimes1)(\tilde{J}\xi)(s,k^{-1})
        \\&=(v_k^*\otimes1)(J_1\otimes1)(\alpha_k(e)\otimes1)(J_1\otimes1)(v_k\otimes1)\xi(s,k)
        \\&=(v_k^*\otimes1)(\alpha_k(e)\otimes1)(v_k\otimes1)\xi(s,k)
        \\&=(\alpha_k^{-1}(\alpha_k(e))\otimes1)\xi(s,k)=(e\otimes1)\xi(s,k).
    \end{align*}
    Now, we shall show that $\text{Sp}(\tilde{\Delta }I_\ell^\alpha(e)\tilde{J}I_\ell^\alpha(e)\tilde{J}) \backslash\{0\}=\overline{\{k\in\Delta_G(G):\alpha_k(e)e\ne0\}}\backslash\{0\}$. Take $k\in\Delta_G(G)$ such that $\alpha_k(e)e\ne0$. Since $\sigma$ is an inverse preserving section, we have $\alpha_{k^{-1}}=\alpha_{k}^{-1}$ and hence $\alpha_{k^{-1}}(e)e=\alpha_{k^{-1}}(\alpha_{k}(e)e)$ is a nonzero projection in $Z(L(G_1))$. Thus, there exists a nonzero $\xi_0\in L^2(G_1)$ such that $\alpha_{k^{-1}}(e)e\xi_0=\xi_0$. With the identification $L^2(G_1 \times \Delta_G(G)) \cong \ell^2(\Delta_G(G), L^2(G_1))$, let $\xi\in \ell^2(\Delta_G(G), L^2(G_1))$ be the function such that $\xi(k)=\xi_0$ and $\xi(l)=0\,\,(l\ne k)$. Then, we see that $$(I_\ell^\alpha(e)\tilde{J}I_\ell^\alpha(e)\tilde{J}\xi)(s,l)=(\alpha_{l^{-1}}(e)e\otimes1)\xi(s,l)=\xi(s,l).$$
    Moreover, we obtain $(\tilde{\Delta}\xi)(s,l)=k\xi(s,l)$, and hence $k$ is a point spectrum of $\tilde{\Delta}I_\ell^\alpha(e)\tilde{J}I_\ell^\alpha(e)\tilde{J}$. That is, $k \in\text{Sp}(\tilde{\Delta }I_\ell^\alpha(e)\tilde{J}I_\ell^\alpha(e)\tilde{J})$. Since a spectrum is closed, it follows that $$\overline{\{k\in\Delta_G(G):\alpha_k(e)e\ne0\}} \subseteq \text{Sp}(\tilde{\Delta }I_\ell^\alpha(e)\tilde{J}I_\ell^\alpha(e)\tilde{J}).$$ 
    To show the converse inclusion, take $r\in\R_{>0}$ which satisfies $r \notin \overline{\{k\in\Delta_G(G):\alpha_k(e)e\ne0\}}$. Then it follows that there exists a $\varepsilon>0$ such that $(r-\varepsilon,r+\varepsilon) \cap \{k\in\Delta_G(G):\alpha_k(e)e\ne0\} = \emptyset$. Since $\alpha_k(e)e=0$ if and only if $\alpha_{k^{-1}}(e)e=0$, we see that $(r-\varepsilon,r+\varepsilon) \cap \{k\in\Delta_G(G):\alpha_{k^{-1}}(e)e\ne0\} = \emptyset$. Define the operator $T_r\colon L^2(G_1\times \Delta_G(G))\to L^2(G_1\times \Delta_G(G))$ as follows: 
    \begin{align*}
        (T_r\xi)(s,k)=((k\alpha_{k^{-1}}(e)e-r)^{-1}\otimes 1)\xi(s,k)\,\,(\xi\in L^2(G_1\times \Delta_G(G))).
    \end{align*}
    If $\alpha_{k^{-1}}(e)e \ne0$, then $\|(k\alpha_{k^{-1}}(e)e-r)^{-1}\| \le |k-r|^{-1}$. Also, if $\alpha_{k^{-1}}(e)e = 0$, then $\|(k\alpha_{k^{-1}}(e)e-r)^{-1}\| =  |r|^{-1}$. By assumption, we have $\text{sup}_{k\in\Delta_G(G)}\|(k\alpha_{k^{-1}}(e)e-r)^{-1}\| \le \text{max}\{\varepsilon^{-1}, |r|^{-1}\}$. Therefore, we conclude that $T_r \in \B(L^2(G_1\times \Delta_G(G)))$ and $T_r=(\tilde{\Delta }I_\ell^\alpha(e)\tilde{J}I_\ell^\alpha(e)\tilde{J}-r)^{-1}$. Thus, $r \notin \text{Sp}(\tilde{\Delta }I_\ell^\alpha(e)\tilde{J}I_\ell^\alpha(e)\tilde{J})$. We have shown that 
    $$\text{Sp}(\Delta_{(\varphi_G)_e})\backslash\{0\}=\text{Sp}(\tilde{\Delta }I_\ell^\alpha(e)\tilde{J}I_\ell^\alpha(e)\tilde{J})\backslash\{0\}=\overline{\{k\in\Delta_G(G):\alpha_k(e)e\ne0\}}\backslash\{0\}$$ as desired.
\end{proof}

\begin{lemma}\label{intersetofspectrum}
    With the same notation as in Lemma \ref{spectrumformula}, we have $$\bigcap_{0\ne e\in Z(L(G_1))^P}\{k\in\Delta_G(G):\alpha_k(e)e\ne0\}=\{k\in\Delta_G(G):\alpha_k|_{Z(L(G_1))}=\text{id}\}.$$
\end{lemma}
\begin{proof}
    We proceed with a proof similar to that of \cite[Lemma 9.5]{takesaki1973duality}. It is clear that $\{k\in\Delta_G(G):\alpha_k|_{Z(L(G_1))}=\text{id}\} \subseteq \bigcap_{0\ne e\in Z(L(G_1))^P}\{k\in\Delta_G(G):\alpha_k(e)e\ne0\}$. We shall show that if $\alpha_k|_{Z(L(G_1))}\ne \text{id}$, then there exists a nonzero $e \in Z(L(G_1))^P$ which satisfies $\alpha_k(e)e=0$ for each $k\in \Delta_G(G)$. Denote $A=Z(L(G_1))$ for simplicity. Let $\Omega$ be the spectrum of $A$. Since $A$ is unital, $\Omega$ is compact. We identify $A$ with $C(\Omega)$. For each $k\in \Delta_G(G)$, we define the automorphism $\alpha_k^*$ on $\Omega$ as $\alpha_k(a)(\omega)=a(\alpha_k^*(\omega))\,\,(a\in A, \omega\in \Omega)$. That is, $\alpha_k^*(\omega)=\omega \circ \alpha_k\,\,(k\in \Delta_G(G), \omega \in \Omega)$. In this setting, it follows that $\alpha_k^*$ is continuous for each $k\in \Delta_G(G)$. Take $k\in \Delta_G(G)$ such that $\alpha_k|_A \ne \text{id}$. Then $\alpha_k^*\ne \text{id}$ holds. Therefore, there exists a $\omega_0$ which satisfies $\alpha_k^*(\omega_0) \ne \omega_0$. We can take an open neighborhood $U$ of $\omega_0$ such that $U \cap \alpha_k^*(U) = \emptyset$. Let $a\in A \cong C(\Omega)$ be a positive element which satisfies $a(\omega_0)=1, a(U^c)=0$ and $0 \le a \le1$. Then it follows that $\alpha_k(a)a=0$. Put $e=1_{[1/2,1]}(a)$ be the spectral projection of $a$. We see that $e \le a$ and hence $\alpha_k(e)e=0$ holds.
\end{proof}

\begin{definition}
    For a locally compact group $G$, we say that $\Delta_G(G)$ is relatively discrete if the relative topology $\Delta_G(G) \subseteq \R_+$ is discrete. Equivalently, for each $k\in \Delta_G(G)$, there exists a $\varepsilon_k>0$ such that $(k-\varepsilon_k, k+\varepsilon_k)\cap \Delta_G(G)=\{k\}$.
\end{definition}

\begin{remark}
    For any multiplicative subgroup $K \subseteq \R_+$, if there are two elements $\lambda, \tilde{\lambda} \in K$ such that $\frac{\log \lambda}{\log \tilde{\lambda}}\notin \Q$, then it can be shown that $K$ is dense in $\R_{\ge0}$. Therefore, it follows that $K \subseteq \R_+$ is relatively discrete if and only if $K$ is singly generated. Moreover, we see that $K$ is not relatively discrete if and only if $K$ is dense in $\R_{\ge0}$.
\end{remark}

\begin{remark}
    Let $G$ be a locally compact group such that $\Delta_G(G)$ is relatively discrete. Then, by the above remark, there exists $\varepsilon >0$ such that $\Delta_G(G) \cap (1- \varepsilon, 1 + \varepsilon)=\{1\}$. Since $\Delta_G \colon G \to \R_+$ is continuous, $G_1 = \{g\in G : \Delta_G(g)=1\}=\Delta_G^{-1}(\{ 1\})=\Delta_G^{-1}((1- \varepsilon, 1 + \varepsilon))$ is open in $G$. Hence, we conclude that $G$ is an almost unimodular group.
\end{remark}

\begin{proposition}\label{Sinvformula}
     Let $G$ be an almost unimodular group. We identify $L(G) \cong L(G_1)\rtimes_{\alpha,u}^\ell \Delta_G(G)$ with an inverse preserving section $\sigma\colon \Delta_G(G) \to G$. If $L(G)$ is a factor then,
     \begin{align*}
         \text{S}(L(G))\backslash\{0\}&=\bigcap_{0\ne e\in Z(L(G_1))^P}\overline{\{k\in\Delta_G(G):\alpha_k(e)e\ne0\}}\backslash\{0\}
         \\&\supseteq \overline{\{k\in\Delta_G(G):\alpha_k|_{Z(L(G_1))}=\text{id}\}}\backslash\{0\}.
     \end{align*}
     Moreover, if $\Delta_G(G)$ is relatively discrete, then 
     $$\text{S}(L(G))\backslash\{0\}=\{k\in\Delta_G(G):\alpha_k|_{Z(L(G_1))}=\text{id}\} \backslash\{0\}.$$
\end{proposition}
\begin{proof}
    Theorem \ref{ConnesSinv}, Lemma \ref{spectrumformula} and Lemma \ref{intersetofspectrum} implies the desired formula.
\end{proof}

There is a characterization of semifiniteness of the group von Neumann algebras of almost unimodular groups.

\begin{proposition}\label{sfnessbycocycleact}
    Let $G$ be an almost unimodular group. We identify $L(G) \cong L(G_1)\rtimes_{\alpha,u}^\ell \Delta_G(G)$ with an arbitrary section $\sigma\colon \Delta_G(G) \to G$. Then, the following conditions are equivalent. 
    \begin{enumerate}
        \item [(i)] $L(G)$ is semifinite.
        \item [(ii)] There exists a one parameter unitary group $(v(t))_{t\in \R} \subseteq Z(L(G_1))$ which satisfies $\alpha_k(v(t))=k^{-it}v(t)\,\,(k\in \Delta_G(G)).$
    \end{enumerate}
\end{proposition}
\begin{proof}
    We provide the proof following \cite[Theorem 8.6]{takesaki1973duality}. 
    \par
    $\text{(i)}\implies \text{(ii)}.$ By the assumption, there exists a one parameter unitary group $(v(t))_{t\in\R}$ which satisfies $\sigma_t^{\varphi_G}=\text{Ad}(v(t))$. Then, we have $$\sigma_t^{\varphi_G}(v(s))=v(t)v(s)v(t)^*=v(t+s-t)=v(s)\,\,(t,s\in \R).$$ Therefore, $(v(t))\subseteq L(G)^{\varphi_G} \cong L(G_1)$ holds. Moreover, we obtain $$v(t)xv(t)^*=\sigma_t^{\varphi_G}(x)=x\,\,(x\in L(G)^{\varphi_G} \cong L(G_1))$$ for all $t\in \R$. We conclude that $(v(t))_{t\in\R} \subseteq Z(L(G_1))$. With the identification $L(G) \cong L(G_1)\rtimes_{\alpha,u}^\ell \Delta_G(G)$, put $w(t)=\Phi(v(t))=I_\ell^\alpha(v(t))\,\,(t\in \R)$ where $\Phi\colon L(G) \to L(G_1)\rtimes_{\alpha,u}^\ell \Delta_G(G)$ is the isomorphism. If we set $\psi=\varphi_G\circ \Phi^{-1}$, then $\sigma_t^{\psi}=\Phi\circ \sigma_t^{\varphi_G}\circ\Phi^{-1}$ and hence $\sigma_t^{\psi}=\text{Ad}(w(t))$ holds. We calculate $$\sigma_t^{\psi}(\lambda_\ell^u(k))=\Phi(\sigma_t^{\varphi_G}(\lambda_\ell^{G}(\sigma(k^{-1})^{-1})))=\Delta_G(\sigma(k^{-1})^{-1})^{it}\Phi(\lambda_\ell^{G}(\sigma(k^{-1})^{-1}))=k^{it}\lambda_\ell^u(k).$$ Hence, $$w(t)\lambda_\ell^u(k)^*w(t)^*=\sigma_t^{\psi}(\lambda_\ell^u(k)^*)=\sigma_t^{\psi}(\lambda_\ell^u(k))^*=k^{-it}\lambda_\ell^u(k)^*$$
    holds. Therefore, by the covariant relation, we have $$I_\ell^{\alpha}(\alpha_k(v(t)))w(t)^*=\lambda_\ell^u(k)w(t)\lambda_\ell^u(k)^*w(t)^*=k^{-it}\lambda_\ell^u(k)\lambda_\ell^u(k)^*=k^{-it}.$$ We conclude that $I_\ell^{\alpha}(\alpha_k(v(t)))=k^{-it}u(t)=I_\ell^\alpha(k^{-it}v(t))$ and hence $\alpha_k(v(t))=k^{-it}v(t)$.
    \par
    $\text{(ii)}\implies \text{(i)}.$ Put $w(t)=I_\ell^\alpha(v(t)).$ Then, since $(v(t))_{t\in \R} \subseteq Z(L(G_1))$, we have $$\text{Ad}(w(t))(I_\ell^\alpha(x))=I_\ell^\alpha(v(t)xv(t)^*)=I_\ell^\alpha(x)\,\,(x\in L(G_1)).$$
    Moreover, we calculate $$w(t)\lambda_\ell^u(k)w(t)^*\lambda_\ell^u(k)^*=w(t)I_\ell^\alpha(\alpha_k(v(t)))^*=k^{it}w(t)w(t)^*=k^{it}.$$
    Thus, it follows that $\text{Ad}(w(t))(\lambda_\ell^u(k))=k^{it}\lambda_\ell^u(k).$ 
    With the same notation in the proof of $\text{(i)}\implies \text{(ii)}$, we see that
    \begin{align*}
        \begin{cases}
        \sigma_t^{\psi}(I_\ell^\alpha(x))=\Phi(\sigma_t^{\varphi_G}(x))=I_\ell^\alpha(x)=\text{Ad}(w(t))(I_\ell^\alpha(x))\,\,(x\in L(G_1))
        \\
        \sigma_t^{\psi}(\lambda_\ell^u(k))=k^{it}\lambda_\ell^u(k)=\text{Ad}(w(t))(\lambda_\ell^u(k))\,\,(k\in \Delta_G(G)).
        \end{cases}
    \end{align*}
    Therefore, we conclude that $\sigma_t^{\psi}=\text{Ad}(w(t))\,\,(t\in \R)$ and hence $L(G) \cong L(G_1)\rtimes_{\alpha,u}^\ell \Delta_G(G)$ is semifinite.
\end{proof}

Similarly, we can show the following proposition. Since the proof is essentially identical, we leave it to the reader. Here, $\text{T}(M)$ is the T-invariant of a von Neumann algebra $M$. 

\begin{proposition}
    With the same notation as in Proposition \ref{sfnessbycocycleact}, the following conditions are equivalent.
    \begin{enumerate}
        \item [(i)] $t \in \text{T}(L(G))$.
        \item [(ii)] There exists a unitary $v \in Z(L(G_1))$ which satisfies $\alpha_k(v)=k^{-it}v\,\,(k\in \Delta_G(G)).$
    \end{enumerate}
\end{proposition}

There is an example which can be applied to the above propositions. 

\begin{proposition}\label{suffcondiforIII0}
    Let $G$ be a locally compact group such that $\Delta_G(G)$ is relatively discrete. Assume $L(G_1)$ is not a factor. We identify $L(G) \cong L(G_1)\rtimes_{\alpha,u}^\ell \Delta_G(G)$ with an inverse preserving section $\sigma\colon \Delta_G(G) \to G$. If $\alpha_k|_{Z(L(G_1))}$ is ergodic for each $k \in \Delta_G(G) \backslash \{1\}$, then $L(G)$ is a $\rm{III}_0$-factor.
\end{proposition}
\begin{proof}
    Since $\alpha_k|_{Z(L(G_1))}$ is ergodic for each $k\in \Delta_G(G)$, $(\alpha,u):\Delta_G(G) \curvearrowright L(G_1)$ is centrally ergodic. Thus, by Theorem \ref{tdlcfactority}, $L(G)$ is a factor. \par
    Moreover, since $L(G_1)$ is not a factor, for each $k\in \Delta_G(G) \backslash\{1\}$, there exists $x\in Z(L(G_1))$ such that $\alpha_k(x) \ne x$. Hence, by Proposition \ref{Sinvformula}, it follows that $$S(L(G))\backslash\{0\}=\{k\in\Delta_G(G):\alpha_k|_{Z(L(G_1))}=\text{id}\} \backslash\{0\}=\{1\}.$$
    Finally, we shall show that $L(G)$ is not semifinite. Assume that this is not the case. Then, there exists a one parameter unitary group $(v(t))_{t\in \R} \subseteq Z(L(G_1))$ which satisfies $\alpha_k(v(t))=k^{-it}v(t)\,\, (k \in \Delta_G(G))$ by Proposition \ref{sfnessbycocycleact}. Fix $k\in \Delta_G(G) \backslash\{1\}$ such that $k^{1/2}\in \Delta_G(G)$. Set $T=-\frac{2\pi}{\text{log}k}$ . We have $\alpha_k(v(T))=k^{-iT}v(T)=e^{2\pi i}v(T)=v(T)$. By the ergodicity of $\alpha_k|_{Z(L(G_1))}$, we see that $v(T)=c$ for some constant $c \in \mathbb{T}$. However, it follows that $$c=\alpha_{k^{1/2}}(c)=\alpha_{k^{1/2}}(v(T))=(k^{1/2})^{-iT}v(T)=e^{\pi i}v(T)=-c,$$ 
    which is a contradiction to $c\in \mathbb{T}$. Therefore, we conclude that $L(G)$ is not semifinite and hence is a $\rm{III}_0$-factor.
\end{proof}

\begin{example}
    We use the same notation as in Example \ref{afineex}. That is, $G=\Q_p \rtimes \Q_p^*,G_1=\Q_p \rtimes \Z_p^*$. We know that $Z(L(G_1)) \cong L^{\infty}(\Q_p)^{\dot{\gamma}}$ where $\dot{\gamma}\colon \Z_p^* \acts L^{\infty}(\Q_p)$ is obtained by $(\dot{\gamma}_sf)(x)=f(sx)\,\,(f\in L^{\infty}(\Q_p), s\in \Z_p^*, x\in \Q
    _p).$ For $t\in \R$, let $f_t(x)=|x|_p^{it}\,\,(x\in \Q_p)$. Then, it follows that $f_t \in L^{\infty}(\Q_p)^{\dot{\gamma}}$ since $|s|_p=1\,\,(s\in \Z_P^*)$ holds. Moreover, since $(f_t\overline{f_t})(x)=|x|_p^{it}|x|_p^{-it}=1$ and $(f_tf_s)(x)=|x|_p^{it}|x|_p^{is}=|x|_p^{i(t+s)}=f_{t+s}(x)$, we conclude that $(f_t)_{t\in \R}\subseteq Z(L(G_1))$ is a one parameter unitary group. Recall that $(\alpha_{p^n}(f))(x)=f(p^nx)$ for all $f\in L^{\infty}(\Q_p)$. We have $(\alpha_{p^n}(f_t))(x)=|p^nx|_p^{it}=|p^n|_p^{it}|x|_p^{it}=(p^{-n})^{it}|x|_p^{it}=(p^n)^{-it}f_t(x)$ for all $n\in \Z, t\in \R$. Therefore, we conclude that $L(G)$ is semifinite by Proposition \ref{sfnessbycocycleact}.
\end{example}

\begin{example}
    We still use the same notation as in the above example. Let $f=1_{\Z_p} \in L^{\infty}(\Q_p)$. Since $s\Z_p=\Z_p$ for all $s\in \Z_p^*$, we see that $f\in L^{\infty}(\Q_p)^{\dot{\gamma}} \cong Z(L(G_1))$. Then, for each nonzero $n\in\Z$, we have $\alpha_{p^n}(f)=1_{p^{-n}\Z_p}\ne f$. By Proposition \ref{Sinvformula}, we conclude that $S(L(G))\backslash\{0\}=\{1\}$ because $\Delta_G(G)=\{p^n:n\in \Z\}$ is relatively discrete.
\end{example}

\begin{example}\label{Sutherlandex}
    We also compute the case as in \cite[Lemma 5.1]{Sutherland1978}. Suppose $H$ is a unimodular group such that $L(H)$ is a factor. Let $a_\lambda$ be a continuous automorphism of $H$ which satisfies $\mu_H\circ a_\lambda^{-1}=\lambda \mu_H$ for some $\lambda\in(0,1)$. Also, let $K=\{0,1\}^{\Z}$ be a compact group. We define $s\colon K \to K$ by $(s(x))_i=x_{i-1}\,\, (x\in K, i\in \Z)$. Let $\hat{K}$ be the dual group of $K$, and $\hat{s}$ the automorphism of $\hat{K}$ dual to $s$ on $K$. That is, $\nai{x}{\hat{s}(p)}=\nai{s(x)}{p}$ for all $x\in K, p\in \hat{K}$. Define the automorphism $b_\lambda$ of $H \times \hat{K}$ by $b_\lambda(h,p)=(a_\lambda(h),\hat{s}(p))$. For each $n\in \Z$, we define an automorphism of $H\times \hat{K}$ by $\alpha_n(h,p)=b_\lambda^n(h,p)$. Then, $\alpha\colon \Z \acts H\times \hat{K}$ becomes an action of the discrete group $\Z$. Consider the crossed product group $G=(H\times \hat{K})\rtimes_{\alpha}\Z$. It can be shown that $\Delta_G(h,p,n)=\lambda^n \,\,(h\in H, p\in \hat{K}, n\in \Z)$ and hence $G$ is almost unimodular. Therefore, $G_1=\{(h,p,0): h\in H, p\in \hat{K}\} \cong H\times \hat{K}$ is open in $G$. Since there is a homomorphism section $\sigma\colon \Delta_G(G) \to G$ defined by $\sigma(\lambda^n)=(e_H,e_{\hat{K}},n)$ and $\sigma(\lambda^n)(h,p,0)\sigma(\lambda^n)^{-1}=(e_H,e_{\hat{K}},n)(h,p,0)(e_H,e_{\hat{K}},-n)=(a_\lambda^n(h),\hat{s}^n(p),0)$, it follows that the twisted crossed product decomposition $L(G) \cong L(G_1) \rtimes_{\alpha,u}^\ell \Delta_G(G)$ via Proposition \ref{cocycleresol} is truly the crossed product decomposition $L(G)=L((H\times\hat{K})\rtimes_{\alpha}\Z) \cong L(H\times\hat{K})\rtimes_{\overline{\alpha}}\Z$.
    Here, $\overline{\alpha}$ is the action of $\Z$ on $L(H\times\hat{K})$ induced by $\alpha\colon \Z \acts H\times \hat{K}$. Since $L(H\times \hat{K}) \cong L(H) \vNtensor L(\hat{K}) \cong L(H) \vNtensor L^{\infty}(K)$, we see that $Z(L(G_1)) \cong Z(L(H)) \vNtensor Z(L^{\infty}(K))\cong L^{\infty}(K)$. It can be shown that $(\overline{\alpha}_n(f))(x)=f(s^n(x))$ for $f\in L^{\infty}(K), x\in K$ with this identification. Then, it is known that this action is ergodic for each $n\in \Z$, and hence $L(G)$ is a $\rm{III}_0$-factor by Proposition \ref{suffcondiforIII0}. 
    %Let $U_n=\Pi_{n\ne m\in \Z}\{0,1\} \times \{1\} \subseteq K$ be an open set for each $n \in \Z$. Then,  $\overline{\alpha}_n(1_{U_0})=1_{U_{-n}} \ne 1_{U_0}$ holds for all nonzero $n \in \Z$. By Proposition \ref{Sinvformula}, $S(L(G))\backslash\{0\}=\{1\}$ holds. Moreover, we assume there is a one parameter unitary group $(f_t)_{t\in \R} \subseteq L^{\infty}(K)$ which satisfies $\overline{\alpha}_n(f_t)=\lambda^{-int}f_t\,\,(n\in\Z, t\in\R)$. Set $t_0= -\pi/\text{log}\lambda$. Then, it follows that $\overline{\alpha}_2(f_{t_0})=\lambda^{-2it_0}f_{t_0}=e^{2\pi i}f_{t_0}=f_{t_0}$ for all $n\in \Z$. The ergodicity of $s^2$ implies that $f_{t_0}\equiv c$ for some constant $c \in \mathbb{T}$. But, we have $(\overline{\alpha}_1 (f_{t_0}))(x)=f_{t_0}(s(x))=c$, while $\overline{\alpha}_1(f_{t_0})=\lambda^{-it_0}f_{t_0}=-f_{t_0}\equiv -c$, which is a contradiction. By Proposition \ref{sfnessbycocycleact}, we conclude that $L(G)$ is not semifinite and hence is a $\rm{III}_0$-factor. 
\end{example}

\begin{remark}
    Let $G$ be a second countable almost unimodular group such that $L(G_1)$ is a factor. If $\Delta_G(G)$ is relatively discrete, then we have $\Delta_G(G)=\{\lambda^n : n \in \Z \}$ for some $\lambda \in (0,1)$. Fix $g \in G$ such that $\Delta_G(g) = \lambda$. Then, we can construct a homomorphism section $\sigma \colon \Delta_G(G) \to G$ defined by $\lambda^n \mapsto g^n\,(n \in \Z)$. Let $a_\lambda$ be a continuous automorphism on $G_1$ defined by $a_\lambda = \text{Ad}g$. Then, the argument in Example \ref{Sutherlandex} applies to $G_1$ and $a_\lambda$. Moreover, we see that $G_1 \rtimes \Z \cong G_1 \rtimes \Delta_G(G)$ with the homomorphism section $\sigma$.
\end{remark}

This section ends with an extended example based on Example \ref{Sutherlandex}, demonstrating the case that $\Delta_G(G)$ is dense in $\R_+$.

\begin{example}\label{Sutherlandextension}
    Suppose $H$ is a second countable almost unimodular group which satisfies that $L(H_1)$ is a factor and $\overline{\Delta_H(H)}=\R_{\ge0}$. (Such group exists; e.g. Take second countable totally disconnected groups $G,\tilde{G}$ such that $L(G_1)$ and $L(\tilde{G}_1)$ are factors and $\Delta_G(G)=\{\lambda^{n}: n\in\Z\}, \Delta_{\tilde{G}}(\tilde{G})=\{\tilde{\lambda}^n : n \in \Z\}$ holds for some $\lambda , \tilde{\lambda} \in (0,1)$ which satisfies $\frac{\log \lambda}{\log \tilde{\lambda}}\notin \Q$. Then, $H=G \times \tilde{G}$ does the work.) Denote $\Gamma=\Delta_H(H)$ and regard it as a discrete group. Note that $\Gamma$ is countable because of \cite[Proposition2.3]{guinto2025unimodulargroups}. Also, let $K=\{0,1\}^{\Gamma}$ be a compact group. There is an action $s:\Gamma \curvearrowright K$ defined by $$(s_\gamma(x))_\delta=x_{\gamma^{-1}\delta}\,\,(\gamma,\delta \in \Gamma, x\in K).$$
    Moreover, taking a dual map, there is an action $\hat{s}\colon \Gamma \curvearrowright \hat{K}$ defined by $$\nai{x}{\hat{s}_{\gamma}(p)}=\nai{s_{\gamma}(x)}{p}\,\,(x\in K, \gamma \in \Gamma, p \in \hat{K}).$$
    Also, there is a cocycle action $(\alpha^{\Gamma}, c^{\Gamma})\colon \Gamma \curvearrowright H_1$ such that $H \cong H_1 \rtimes_{\alpha^{\Gamma}, c^{\Gamma}} \Gamma$ associated with an inverse preserving section $\sigma \colon \Gamma \to H$ as in Proposition \ref{cocycleresol}. In this setting, it follows that $L(H)\cong L(H_1) \rtimes_{\alpha,u}\Gamma$ which satisfies $\alpha_\gamma(\lambda_\ell^{H_1}(h))=\lambda_\ell^{H_1}(\alpha_\gamma^\Gamma(h)) \,(\gamma\in \Gamma, h\in H_1)$ and $u(\gamma,\delta)=\lambda_\ell^{H_1}(c^{\Gamma}(\gamma,\delta))\, (\gamma,\delta\in \Gamma)$. \par
    Now, we can construct a new cocycle action $(\beta,v)\colon \Gamma \curvearrowright H_1 \times \hat{K}$ defined by 
    \begin{align*}
        \begin{cases}
            \beta_\gamma^\Gamma(h,p)=(\alpha_\gamma^{\Gamma}(h),\hat{s}_\gamma(p)), \\
            v^\Gamma(\gamma, \delta)=(c^{\Gamma}(\gamma,\delta), e_{\hat{K}}) \,\,(\gamma, \delta \in \Gamma, h\in H_1, p \in \hat{K}).
        \end{cases}
    \end{align*}
    Then, we denote $G=(H_1 \times \hat{K})\rtimes_{\beta^\Ga,v^\Gamma} \Gamma$ and consider $L(G)$. First, we calculate $\Delta_G$. For each $g \in H_1, q\in \hat{K}, \delta \in \Gamma$ and non-negative measurable function $f$ on $G$, we calculate 
    \begin{align*}
        \quad & \int_\Gamma \int_{\hat{K}} \int_{H_1} f((g,q,\delta)(h,p,\gamma))\cdot\gamma\,d\mu_{H_1}(h)d\mu_{\hat{K}}(p)d\mu_{\Gamma}(\gamma)
        \\ & = \int_\Gamma \int_{\hat{K}} \int_{H_1} f((g\alpha_\delta^\Gamma(h)c^\Gamma(\delta,\gamma),q\hat{s}_\delta(p),\delta\gamma))\cdot\gamma \,d\mu_{H_1}(h)d\mu_{\hat{K}}(p)d\mu_{\Gamma}(\gamma)
        \\ & = \int_\Gamma \int_{\hat{K}} \int_{H_1} f((\alpha_\delta^\Gamma (\alpha_{\delta^{-1}}^\Gamma(g)h \alpha_{\delta^{-1}}^\Gamma(c^\Gamma(\delta,\gamma))), \hat{s}_\delta(\hat{s}_{\delta^{-1}}(q)p),\gamma))\cdot\delta^{-1}\gamma\,d\mu_{H_1}(h)d\mu_{\hat{K}}(p)d\mu_{\Gamma}(\gamma) 
        \\ & = \int_\Gamma \int_{\hat{K}} \int_{H_1} f((\alpha_\delta^\Gamma (h ), \hat{s}_\delta(p),\gamma))\cdot\delta^{-1}\gamma\,d\mu_{H_1}(h)d\mu_{\hat{K}}(p)d\mu_{\Gamma}(\gamma) 
        \\ & = \int_\Gamma \int_{\hat{K}} \int_{H_1} f((h,p,\gamma)) \gamma\,d\mu_{H_1}(h)d\mu_{\hat{K}}(p)d\mu_{\Gamma}(\gamma).
    \end{align*}
    Here, we use the fact that $d\mu_{H_1}(\alpha_{\delta^{-1}}^\Gamma(h))=\delta \cdot d\mu_{H_1}(h)$ and $\hat{K}$ and $\Gamma$ are discrete. Therefore, we conclude that $\gamma\cdot d\mu_{H_1}(h)d\mu_{\hat{K}}(p)d\mu_{\Gamma}(\gamma)$ is a left Haar measure on $G$. With the same notation as above, we have
    \begin{align*}
        & \int_\Gamma \int_{\hat{K}} \int_{H_1} f((h,p,\gamma)(g,q,\delta))\cdot\gamma\,d\mu_{H_1}(h)d\mu_{\hat{K}}(p)d\mu_{\Gamma}(\gamma)
        \\ & = \int_\Gamma \int_{\hat{K}} \int_{H_1} f((h\alpha_{\gamma}^\Gamma(g)c^{\Gamma}(\gamma,\delta),p\hat{s}_\gamma(q),\gamma\delta))\cdot\gamma\,d\mu_{H_1}(h)d\mu_{\hat{K}}(p)d\mu_{\Gamma}(\gamma)
        \\ & = \int_\Gamma \int_{\hat{K}} \int_{H_1} f((h,p,\gamma))\cdot\delta^{-1}\gamma\,d\mu_{H_1}(h)d\mu_{\hat{K}}(p)d\mu_{\Gamma}(\gamma)
        \\ & = \delta^{-1} \int_\Gamma \int_{\hat{K}} \int_{H_1} f((h,p,\gamma))\cdot\gamma\,d\mu_{H_1}(h)d\mu_{\hat{K}}(p)d\mu_{\Gamma}(\gamma).
    \end{align*}
    Thus, we see that $\Delta_G(g,q,\delta)=\delta$ for each $(g,q,\delta)\in G$. It follows that $\Delta_G(G)=\Gamma$ and $G_1=(H_1\times \hat{K})\rtimes_{\beta^\Gamma,v^\Gamma}\{1\} \cong H_1\times \hat{K}$. Hence, $G_1$ is open in $G$. We can construct an inverse preserving section $\sigma^G \colon \Gamma \to G$ defined by $\sigma^G(\gamma)=(e_{H_1},e_{\hat{K}},\gamma)\,(\gamma\in \Gamma)$. With this section, we see that $\sigma^G(\gamma)(h,p,1)\sigma^G(\gamma)^{-1}=(\alpha_\gamma^{\Gamma}(h),\hat{s}_\gamma(p),1)$ for each $h\in H_1, p \in \hat{K}, \gamma\in \Gamma$. We can decompose $L(G)$ as $L(G_1)\rtimes_{\beta,v}\Gamma$ by Proposition \ref{cocycleresol}. Note that $(\beta,v) \colon \Gamma \curvearrowright L(G_1)$ satisfies $\beta_{\gamma}(\lambda_\ell^{G_1}(h,p,1))=\lambda_\ell^{G_1}(\alpha_\gamma^\Gamma(h),\hat{s}_\gamma(p),1)$.
    Using the isomorphism $L(G_1) \cong L(H_1 \times \hat{K}) \cong L(H_1) \vNtensor L(\hat{K}) \cong L(H_1) \vNtensor L^{\infty}(K)$, we have $Z(L(G_1)) \cong Z(L(H_1))\vNtensor Z(L^\infty(K))=L^\infty(K)$ by the assumption that $L(H_1)$ is a factor. Moreover, if we set $\Phi\colon L(G_1) \to L(H_1) \vNtensor L^{\infty}(K)$ be the above isomorphism, then we have $ \Phi\circ \beta_\gamma \circ\Phi^{-1}(f)=\overline{s}_\gamma(f)$ for $f \in L^{\infty}(K)$ where $(\overline{s}_\gamma(f))(x)=f(s_\gamma(x))\,(x\in K)$. Therefore, we analyze $\overline{s}_\gamma$ instead of $\beta_\gamma|_{Z(L(G_1))}$. \par
    First, we recall that the Bernoulli shift $s \colon \Gamma \curvearrowright K$ is mixing (see, e.g. \cite[Section 2.3.1]{kerr2016}). Since $\Gamma$ is a subgroup of $\R_+$, it follows that each $\gamma \in \Gamma$ has an infinite order, and hence $s_\gamma$ is ergodic. Therefore, we see that $\overline{s}_\gamma$ is ergodic for each $\gamma \in \Gamma$. Thus, it can be shown that $L(G)$ is a factor and not semifinite as in the proof of Proposition \ref{suffcondiforIII0}. We shall determine the type of $L(G)$.  By Proposition \ref{Sinvformula}, we have 
    \begin{align*}
         \text{S}(L(G))\backslash\{0\} &=\bigcap_{0\ne e\in Z(L(G_1))^P}\overline{\{\gamma\in \Gamma:\beta_\gamma(e)e\ne0\}}\backslash\{0\}
         \\&= \bigcap_{0 \ne 1_E \in L^{\infty}(K)}\overline{\{\gamma\in \Gamma:\overline{s}_\gamma(1_E)1_E\ne0\}}\backslash\{0\}
         \\&= \bigcap_{E\in \mathcal{B}(K),\mu_K(E)>0}\overline{\{\gamma\in \Gamma:\mu_K (s_\gamma(E)\cap E) >0\}}\backslash\{0\},
    \end{align*}
    where $\mathcal{B}(K)$ is the set of Borel sets of $K$. We may assume $\mu_K(K)=1$ since $K$ is compact. Denote $A_\gamma^t=\{x\in K :x_\gamma=t \}$ where $\gamma\in \Gamma, t\in\{0,1\}$. Set $$\mathscr{C}:=\{\bigcap_{\gamma\in F}A_\gamma^{t_\gamma}: F \subseteq\Gamma, 0 \le\#F < \infty,t_\gamma\in\{0,1\}\}.$$ Note that $\mathscr{C}$ generates the topology of $K$ and hence generates $\mathcal{B}(K)$. Fix arbitrary $r\in \R_+$ and $E\in \mathcal{B}(K)$ which satisfies $\mu_K(E)>0$. We shall show that $r \in \overline{\{\gamma\in \Gamma:\mu_K (s_\gamma(E)\cap E) >0\}}$. \par
    Since $\mu_K(E) > 0$, there exists a compact set $C \subseteq E$ which satisfies $\mu_K(C)>0$ by the inner regularity of $\mu_K$. Let $\varepsilon=\mu_K(C)^2/4$. By the outer regularity of $\mu_K$, there exists an open set $U \supseteq C$ which satisfies $\mu_K(U \backslash C) \le \varepsilon$. Since $\mathscr{C}$ is a basis for $K$, we obtain $U=\bigcup_{n=1}^{\infty}U_n$ for some $U_1, U_2, \ldots \in\mathscr{C}$. Then, the compactness of $C$ implies that there is a finite subset $I \subseteq \N$ such that $C \subseteq \bigcup_{n\in I}U_n\subseteq U$. We denote $V=\bigcup_{n\in I}U_n$ and $U_n=\bigcap_{\gamma\in F_n}A_{\gamma}^{t_{\gamma}^n}$ for each $n \in \N$. Set $F:=\bigcup_{n\in I}F_n$ and $\gamma F=\{\gamma\delta:\delta\in F\}\,(\gamma\in\Gamma)$. Then, we see that $F$ is a finite subset. Therefore, it follows that $$\#\{\gamma\in\Gamma:\gamma F \cap F \ne \emptyset\}=\#\{\frac{\gamma}{\delta}:\gamma,\delta\in F\}<\infty.$$
    Now, we shall show that $\mu_K(s_\gamma(V)\cap V)=\mu_K(V)^2$ if $\gamma F\cap F=\emptyset$. Recall that $\mu_K(A)=2^{-\#F_A}$ for any $A=\bigcap_{\gamma\in F_A}A_\gamma^{t_\gamma} \in \mathscr{C}$. Therefore, for $A,B \in\mathscr{C}$, if $F_A \cap F_B=\emptyset$, then we have $\mu_K(A\cap B)=\mu_K(\bigcap_{\gamma\in F_A \cup F_B}A_\gamma^{t_\gamma})=2^{-(\#F_A+\#F_B)}=\mu_K(A)\mu_K(B)$. Moreover, it can be shown that $A^c=\bigsqcup_{i=1}^{2^{\#F_A}-1}\tilde{A}_i$ for some $\tilde{A}_1,\ldots, \tilde{A}_{2^{\#F_A}-1}\in\mathscr{C}$ such that $F_{\tilde{A}_i}=F_A\,(i=1,\ldots,2^{\#F_A}-1)$. Thus, if we define $U_n^{0}=U_n,U_n^1=U_n^c$, then we have
    $$V=\bigsqcup_{\ell\in \{0,1\}^I\backslash\{(1,1,\cdots,1)\}}\bigcap_{n\in I}U_n^{\ell_n}=\bigsqcup_{n=1}^{N}V_n$$
    for some $N\in\N$ and $V_1,\ldots,V_N \in \mathscr{C}$ because $\mathscr{C}$ is closed under finite intersection. Take $\gamma \in \Gamma$ such that $\gamma F \cap F= \emptyset$. Note that $s_\gamma(U_n)=\bigcap_{\delta\in \gamma F_n}A_\delta^{t_{\gamma^{-1}\delta}^n}$, we have
    \begin{align*}
        \mu_K(s_\gamma(V)\cap V)&=\mu_K(s_\gamma(V)\cap(\bigsqcup_{n=1}^NV_n))=\mu_K(\bigsqcup_{n=1}^N(s_\gamma(V)\cap V_n))
        \\&= \sum_{n=1}^{N}\mu_K(s_\gamma(V)\cap V_n) =\sum_{n=1}^N\mu_K((\bigsqcup_{m=1}^Ns_\gamma(V_m))\cap V_n)
        \\&=\sum_{n=1}^N\sum_{m=1}^N\mu_K(s_\gamma(V_m)\cap V_n)=\sum_{n=1}^N\sum_{m=1}^N \mu_K(s_\gamma(V_m))\mu_K(V_n)
        \\&=\sum_{n=1}^N\sum_{m=1}^N \mu_K(V_m)\mu_K(V_n) =\mu_K(V)^2.
    \end{align*}
    Here, we use the fact that $s_\gamma$ is measure preserving. \\
    Next, we shall show that $\mu_K(s_\gamma(C)\cap C)>0$ if $\gamma F \cap F = \emptyset$. Take such $\gamma \in \Gamma$. It is clear that $$\mu_K(s_\gamma(V)\cap C)+\mu_K((s_\gamma(V)\cap V)\backslash C)=\mu_K(s_\gamma(V)\cap V)=\mu_K(V)^2.$$
    Moreover, $(s_\gamma(V) \cap V) \backslash C \subseteq V \backslash C$ implies that $\mu_K((s_\gamma(V) \cap V) \backslash C) \le \mu_K(V \backslash C) \le \varepsilon$. Thus, we have $$\mu_K(s_\gamma(V)\cap C) \ge \mu_K(V)^2-\varepsilon.$$
    Similarly, we obtain $$\mu_K(s_\gamma(C)\cap V) \ge \mu_K(V)^2-\varepsilon.$$ Note that $(s_\gamma(V)\cap C)\cup(s_\gamma(C)\cap V) \subseteq s_\gamma(V) \cap V$, we have 
    $$\mu_K(s_\gamma(V)\cap C)+\mu_K(s_\gamma(C)\cap V)-\mu_K(s_\gamma(V)\cap C\cap s_\gamma(C)\cap V) \le\mu_K(V)^2.$$
    Therefore, we conclude that 
    \begin{align*}
        \mu_K(s_\gamma(C)\cap C) & \ge \mu_K(s_\gamma(V)\cap C)+\mu_K(s_\gamma(C)\cap V) - \mu_K(V)^2
        \\& \ge\mu_K(V)^2- 2\varepsilon \\&=\mu_K(V)^2-\mu_K(C)^2 /2
        \\& \ge \mu_K(C)^2/2 >0.
    \end{align*}
    By the preceding discussion, we have $\mu_K(s_\gamma(C)\cap C)>0$ if $\gamma F \cap F= \emptyset$. In this case, it follows that $\mu_K(s_\gamma(E)\cap E) \ge \mu_K(s_\gamma(C)\cap C)>0.$
    Recall that $\Gamma$ is dense in $\R_{\ge0}$, there is a sequence $(\gamma_n)_{n\in \N} \subseteq \Gamma$ which satisfies $|r-\gamma_n| \to0$ as $n\to \infty$. Then, since $\#((\gamma_n)_{n\in \N}\cap \{\gamma\in\Gamma:\gamma F \cap F \ne \emptyset\}) \le \#\{\gamma\in\Gamma:\gamma F \cap F \ne \emptyset\} < \infty$, we can construct a sequence $(\tilde{\gamma}_{n})_{n\in \N}$ which satisfies $|r-\tilde{\gamma}_n|\to0$ and $\mu_K(s_{\tilde{\gamma}_n}(E)\cap E)>0$ for all $n \in \N$ by removing $\gamma_n$ such that $\gamma_n F \cap F \ne \emptyset$. Thus, we conclude that $r \in \overline{\{\gamma\in \Gamma:\mu_K (s_\gamma(E)\cap E) >0\}}$. 
    \par
    Since $r\in \R_+$ and $E\in \mathcal{B}(K)$ are arbitrary, we obtain
    $$\text{S}(L(G))\backslash\{0\}=\bigcap_{E\in \mathcal{B}(K),\mu_K(E)>0}\overline{\{\gamma\in \Gamma:\mu_K (s_\gamma(E)\cap E) >0\}}\backslash\{0\}=\R_+$$
    and hence $L(G)$ is a type $\rm{III}_1$-factor.
\end{example}

\begin{remark}
    In Example \ref{Sutherlandextension}, since $\overline{s}_\gamma$ is ergodic for each $\gamma \in\Gamma \backslash \{1\}$, it follows that 
    \begin{align*}
        \overline{\bigcap_{0\ne e\in Z(L(G_1))^P}\{\gamma\in \Gamma:\beta_\gamma(e)e\ne0\}}\backslash\{0\}&=\overline{\{\gamma\in\Gamma:\beta_\gamma|_{Z(L(G_1))}=\text{id}\}}\backslash\{0\}\\&=\overline{\{\gamma\in\Gamma : \overline{s}_\gamma=id\}}\backslash\{0\}=\{1\}
    \end{align*}
    while $$\bigcap_{0\ne e\in Z(L(G_1))^P} \overline{\{\gamma\in \Gamma:\beta_\gamma(e)e\ne0\}}\backslash\{0\}=\R_+.$$
\end{remark}

\begin{corollary}\label{CorofSutex}
    Given a non-unimodular almost unimodular group $H$ such that $L(H_1)$ is a factor, we can always construct a new almost unimodular group $G=(H_1 \times \hat{K}) \rtimes_{\alpha, c} \Delta_H(H)$ as in Example \ref{Sutherlandex} and Example \ref{Sutherlandextension}. We conclude that $L(G)$ is a factor of type
    \begin{align*}
         \begin{cases}
        \rm{III}_0 \quad \text{when}\; \Delta_H(H) \; \text{is relatively discrete,} \\
        \rm{III}_1 \quad \text{when}\; \Delta_H(H) \; \text{is dense in} \; \R_{\ge 0}.
    \end{cases}
    \end{align*}    
\end{corollary}

\section*{Acknowledgments}
I would like to express my deepest gratitude to Associate Professor Hiroshi Ando for his invaluable guidance, unwavering support, and insightful discussions throughout the course of this research. I sincerely thank Professor Reiji Tomatsu for helpful discussions about 2-cocycles and thoughtful feedback, especially for pointing us to the literature regarding Theorem \ref{plweightmainthm} and Theorem \ref{twistConnesfactority}. I would also like to thank Professor Brent Nelson and Aldo Garcia Guinto at  Michigan State University for their valuable comments on the first version of the preprint. This work was supported by JST SPRING, Grant Number JPMJSP2109.  
% As mentioned in introduction, he informed me that Theorem \ref{plweightmainthm} (ii)$\implies$(i) and (ii)$\iff$(iii) follows from the facts of locally compact quantum groups. Moreover, he also informed me that Theorem \ref{twistConnesfactority} follows from Sutherland's stabilization trick and the original Connes--Takesaki Theorem. I would also like to thank Professor Brent Nelson for his valuable comments.
\bibliographystyle{siam}
\bibliography{references}

\begin{thebibliography}{10}

\bibitem{Blackadar1977}
{\sc B.~E. Blackadar}, {\em The regular representation of restricted direct product groups}, J. Functional Analysis, 25 (1977), pp.~267--274.

\bibitem{Blattner1962}
{\sc R.~J. Blattner}, {\em On a theorem of {G}. {W}. {M}ackey}, Bull. Amer. Math. Soc., 68 (1962), pp.~585--587.

\bibitem{BoutonnetBrothier2016}
{\sc R.~Boutonnet and A.~Brothier}, {\em Crossed-products by locally compact groups: intermediate subfactors}, J. Operator Theory, 79 (2018), pp.~101--137.

\bibitem{CombesMR288589}
{\sc F.~c. Combes}, {\em Poids et esp\'erances conditionnelles dans les alg\`ebres de von {N}eumann}, Bull. Soc. Math. France, 99 (1971), pp.~73--112.

\bibitem{Connes1973}
{\sc A.~Connes}, {\em Une classification des facteurs de type {${\rm III}$}}, Ann. Sci. \'Ecole Norm. Sup. (4), 6 (1973), pp.~133--252.

\bibitem{Connes1974}
{\sc A.~Connes}, {\em Almost periodic states and factors of type {${\rm III}\sb{1}$}}, J. Functional Analysis, 16 (1974), pp.~415--445.

\bibitem{ConnesTakesaki1977}
{\sc A.~Connes and M.~Takesaki}, {\em The flow of weights on factors of type {${\rm III}$}}, Tohoku Math. J. (2), 29 (1977), pp.~473--575.

\bibitem{Dykema1995}
{\sc K.~Dykema}, {\em Crossed product decompositions of a purely infinite von {N}eumann algebra with faithful, almost periodic weight}, Indiana Univ. Math. J., 44 (1995), pp.~433--450.

\bibitem{EnockSchwartz1992}
{\sc M.~Enock and J.-M. Schwartz}, {\em Kac algebras and duality of locally compact groups}, Springer-Verlag, Berlin, 1992.
\newblock With a preface by Alain Connes, With a postface by Adrian Ocneanu.

\bibitem{Eymard1964}
{\sc P.~Eymard}, {\em L'alg\`ebre de {F}ourier d'un groupe localement compact}, Bull. Soc. Math. France, 92 (1964), pp.~181--236.

\bibitem{FollandrealanalysisMR1681462}
{\sc G.~B. Folland}, {\em Real analysis}, Pure and Applied Mathematics (New York), John Wiley \& Sons, Inc., New York, second~ed., 1999.
\newblock Modern techniques and their applications, A Wiley-Interscience Publication.

\bibitem{Folland}
\leavevmode\vrule height 2pt depth -1.6pt width 23pt, {\em A course in abstract harmonic analysis}, Textbooks in Mathematics, CRC Press, Boca Raton, FL, second~ed., 2016.

\bibitem{GuintoNelson2024}
{\sc A.~Garcia~Guinto, M.~Lorentz, and B.~Nelson}, {\em Murray--von {N}eumann dimension for strictly semifinite weights}, J. Funct. Anal., 289 (2025), pp.~Paper No. 110938, 58.

\bibitem{guinto2025unimodulargroups}
{\sc A.~Garcia~Guinto and B.~Nelson}, {\em Almost unimodular groups}, arXiv:2504.08107,  (2025).

\bibitem{HaageruptypeIIIcentralizerMR430801}
{\sc U.~Haagerup}, {\em An example of a weight with type {III} centralizer}, Proc. Amer. Math. Soc., 62 (1977), pp.~278--280.

\bibitem{Haagerupdualweight}
\leavevmode\vrule height 2pt depth -1.6pt width 23pt, {\em On the dual weights for crossed products of von {N}eumann algebras. {II}. {A}pplication of operator-valued weights}, Math. Scand., 43 (1978/79), pp.~119--140.

\bibitem{KKS2016}
{\sc M.~Kalantar, P.~Kasprzak, and A.~Skalski}, {\em Open quantum subgroups of locally compact quantum groups}, Adv. Math., 303 (2016), pp.~322--359.

\bibitem{Kaniuth2009}
{\sc E.~Kaniuth}, {\em A course in commutative {B}anach algebras}, vol.~246 of Graduate Texts in Mathematics, Springer, New York, 2009.

\bibitem{KaniuthLaubookMR3821506}
{\sc E.~Kaniuth and A.~T.-M. Lau}, {\em Fourier and {F}ourier-{S}tieltjes algebras on locally compact groups}, vol.~231 of Mathematical Surveys and Monographs, American Mathematical Society, Providence, RI, 2018.

\bibitem{kerr2016}
{\sc D.~Kerr and H.~Li}, {\em Ergodic theory}, Springer Monographs in Mathematics, Springer, Cham, 2016.
\newblock Independence and dichotomies.

\bibitem{JohanVaes2000}
{\sc J.~Kustermans and S.~Vaes}, {\em Locally compact quantum groups}, Ann. Sci. \'Ecole Norm. Sup. (4), 33 (2000), pp.~837--934.

\bibitem{Loomis1953}
{\sc L.~H. Loomis}, {\em An introduction to abstract harmonic analysis}, D. Van Nostrand Co., Inc., Toronto-New York-London, 1953.

\bibitem{Nakagami1977}
{\sc Y.~Nakagami}, {\em Dual action on a von {N}eumann algebra and {T}akesaki's duality for a locally compact group}, Publ. Res. Inst. Math. Sci., 12 (1976/77), pp.~727--775.

\bibitem{NakagamiSutherland1979}
{\sc Y.~Nakagami and C.~Sutherland}, {\em Takesaki's duality for regular extensions of von {N}eumann algebras}, Pacific J. Math., 83 (1979), pp.~221--229.

\bibitem{PedersenTakesaki1973}
{\sc G.~K. Pedersen and M.~Takesaki}, {\em The {R}adon-{N}ikodym theorem for von {N}eumann algebras}, Acta Math., 130 (1973), pp.~53--87.

\bibitem{Sauvageot1977}
{\sc J.-L. Sauvageot}, {\em Sur le type du produit crois\'e{} d'une alg\`ebre de von {N}eumann par un groupe localement compact}, Bull. Soc. Math. France, 105 (1977), pp.~349--368.

\bibitem{Sutherland1973}
{\sc C.~E. Sutherland}, {\em T{HE} {DIRECT} {INTEGRAL} {THEORY} {OF} {WEIGHTS}, {AND} {THE} {PLANCHEREL} {FORMULA}}, ProQuest LLC, Ann Arbor, MI, 1973.
\newblock Thesis (Ph.D.)--University of California, Los Angeles.

\bibitem{Sutherland1978}
{\sc C.~E. Sutherland}, {\em Type analysis of the regular representation of a nonunimodular group}, Pacific J. Math., 79 (1978), pp.~225--250.

\bibitem{Sutherland1980}
\leavevmode\vrule height 2pt depth -1.6pt width 23pt, {\em Cohomology and extensions of von {N}eumann algebras. {I}, {II}}, Publ. Res. Inst. Math. Sci., 16 (1980), pp.~105--133, 135--174.

\bibitem{Takesaki72MR303307}
{\sc M.~Takesaki}, {\em Conditional expectations in von {N}eumann algebras}, J. Functional Analysis, 9 (1972), pp.~306--321.

\bibitem{takesaki1973duality}
{\sc M.~Takesaki}, {\em Duality for crossed products and the structure of von neumann algebras of type iii}, Acta Mathematica, 131 (1973), pp.~249--310.

\bibitem{TakesakiOA2}
{\sc M.~Takesaki}, {\em Theory of operator algebras. {II}}, vol.~125 of Encyclopaedia of Mathematical Sciences, Springer-Verlag, Berlin, 2003.
\newblock Operator Algebras and Non-commutative Geometry, 6.

\bibitem{TakesakiTatsuuma1971}
{\sc M.~Takesaki and N.~Tatsuuma}, {\em Duality and subgroups}, Ann. of Math. (2), 93 (1971), pp.~344--364.

\bibitem{TakesakiTatsuuma1972}
\leavevmode\vrule height 2pt depth -1.6pt width 23pt, {\em Duality and subgroups. {II}}, J. Functional Analysis, 11 (1972), pp.~184--190.

\bibitem{VaesVainerman2003}
{\sc S.~Vaes and L.~Vainerman}, {\em Extensions of locally compact quantum groups and the bicrossed product construction}, Adv. Math., 175 (2003), pp.~1--101.

\end{thebibliography}

\end{document}